\documentclass[11pt]{article}

\usepackage{geometry,color}

\usepackage{amssymb}
\usepackage{amsmath,amsfonts,amsthm}
\usepackage{geometry,cite,enumerate,graphicx}
\usepackage{caption,hyperref}
\usepackage{subcaption}
\usepackage{pgfplots}
\usepackage{authblk}
\pgfplotsset{compat=newest}
\usepackage{graphicx}
\newcommand{\BigO}[1]{\ensuremath{\operatorname{O}\left(#1\right)}}

\newcommand{\eps}{\varepsilon}
\newcommand{\R}{\mathbb{R}}
\newcommand{\Om}{\Omega}

\DeclareMathOperator{\arctanh}{arctanh}

\numberwithin{equation}{section}

\newtheorem{theorem}{Theorem}[section]
\newtheorem{proposition}[theorem]{Proposition}
\newtheorem{definition}[theorem]{Definition}
\newtheorem{lemma}[theorem]{Lemma}

\newtheorem{remark}[theorem]{Remark}

\begin{document}
\title{\bf A variational model of charged drops in dielectrically
  matched binary fluids: the effect of charge discreteness}

\author[1,2]{Cyrill B. Muratov}

\author[2]{Matteo Novaga}

\author[1]{Philip Zaleski}

\affil[1]{Department of Mathematical Sciences, New Jersey Institute of
  Technology, Newark, NJ 07102, USA}

  \affil[2]{Dipartimento di Matematica, Universit\`a di Pisa, Largo
    B. Pontecorvo 5, Pisa 56127, Italy}
  
\maketitle

\begin{abstract}
  This paper addresses the ill-posedness of the classical Rayleigh
  variational model of conducting charged liquid drops by
  incorporating the discreteness of the elementary
  charges. Introducing the model that describes two immiscible fluids
  with the same dielectric constant, with a drop of one fluid
  containing a fixed number of elementary charges together with their
  solvation spheres, we interpret the equilibrium shape of the drop as
  a global minimizer of the sum of its surface energy and the
  electrostatic repulsive energy between the charges under fixed drop
  volume. For all model parameters, we establish existence of
  generalized minimizers that consist of at most a finite number of
  components ``at infinity''. We also give several existence and
  non-existence results for classical minimizers consisting of only a
  single component. In particular, we identify an asymptotically sharp
  threshold for the number of charges to yield existence of minimizers
  in a regime corresponding to macroscopically large drops containing
  a large number of charges. The obtained non-trivial threshold is
  significantly below the corresponding threshold for the Rayleigh
  model, consistently with the ill-posedness of the latter and
  demonstrating a particular regularizing effect of the charge
  discreteness. However, when a minimizer does exist in this regime,
  it approaches a ball with the charge uniformly distributed on the
  surface as the number of charges goes to infinity, just as in the
  Rayleigh model. Finally, we provide an explicit solution for the
  problem with two charges and a macroscopically large drop.
\end{abstract}

\newpage

\tableofcontents

\section{Introduction}
\label{sec:introduction}

There has recently been a growing interest in geometric variational
problems featuring a competition of attractive and repulsive
interactions \cite{cmt:nams17}. A prototypical model giving rise to
the problems of this kind is the celebrated Gamow's liquid drop model
of the atomic nucleus \cite{gamow30}, in which a competition of the
cohesive action of the surface tension with the Coulombic repulsion
gives rise to delicate questions about the existence and the shape of
minimizers, etc. There are now many studies of this model and its
various generalizations and extensions that are too numerous to list
here (for some recent works, see, e.g.,
\cite{frank21,merlet22,novaga22} and references therein).

We focus on a closely related problem arising from the classical model
introduced by Lord Rayleigh that describes the energetics of a
perfectly conducting charged liquid drop \cite{rayleigh1882} (for the
technical details of the model, see section \ref{sec:model}). In 1882,
Rayleigh demonstrated that a spherical liquid drop becomes linearly
unstable with respect to asymmetric distortions of its shape when the
amount of charge on the droplet exceeds a critical value called the
Rayleigh charge. Such an interfacial instability driven by the
electric field was first observed experimentally by Zeleny
\cite{zeleny14,zeleny17} and subsequently studied by great many
authors (see, e.g., \cite{doyle64,abbas67,gomez94,duft03,giglio08}),
not least because of its important applications to analytical
chemistry \cite{gaskell97}.  Surprisingly, however, the linear
stability of the charged drop below the critical charge in the
Rayleigh model was recently shown not to imply stability of a
spherical drop with respect to arbitrarily small perturbations of its
shape \cite{mn:prsa16}. In fact, the Rayleigh model leads to a
problem that is variationally ill-posed \cite{goldman15,
  mn:prsa16,goldman17}. Mathematically, this is because the
regularizing action of the perimeter is not sufficient to control the
electric charges at small scales \cite{goldman22}. Physically, it
manifests itself in the formation of singularities in the form of
Taylor cones and jets \cite{taylor64,kebarle00,fernandezdelamora07}.

The variational ill-posedness of the above problem indicates that the
Rayleigh model does not contain all the physics that is necessary to
describe the equilibrium shapes of conducting charged drops. Several
regularizing mechanisms have, therefore, been proposed, including
thermal effects that restore existence of minimizers under certain
conditions due to the spreading of the charges into a thin Debye layer
beneath the droplet surface
\cite{mn:prsa16,dephilippis23,mukoseeva23}. Nevertheless, in some
situation such as cryogenic liquids or nanoscale droplets, in which
the thermal motion of free charges is suppressed, another physical
mechanisms may be necessary. One such mechanism relies on the
fundamental discreteness of the electric charges
\cite{iribane76,fenn93,labowsky98,kebarle00}. In this paper, we
explore this possibility in the special case of dielectrically matched
fluids, in which there is no dielectric contrast between the droplet
and its surroundings (again, see section \ref{sec:model} for technical
details).

For a model that keeps track of the positions of individual charges
inside the droplet, we establish existence of generalized minimizers,
a suitable notion of minimality for this kind of problems that
accounts for a possibility of components that are infinitely far
apart, first introduced in \cite{kmn:cmp16}. We also establish the
regularity and connectedness of the components of the generalized
minimizers. We then proceed to investigate under which conditions
classical minimizers, consisting of only a single component, are
possible in the physically important regime of sufficiently strong
repulsion between the charges in comparison to the surface
tension. Here we establish a sharp existence/non-existence criterion
in the case of many charges, which yields a critical charge for
existence that is significantly smaller than the Rayleigh charge. We
also establish some structural information about the locations of the
charges when the minimizers do exist and show that in a suitable
continuum limit within the existence range the minimizer converges in
an appropriate sense to a ball with the charges uniformly distributed
on its surface. Lastly, we present an explicit solution of the
variational problem in the case of only two point charges.

Our paper is organized as follows. In section \ref{sec:model}, we
introduce the model considered in this paper and discuss the relevant
parameter ranges. In section \ref{sec:main-results}, we state the main
results of our paper. In section \ref{sec:gener-minim}, we present the
proof of Theorem \ref{exgen} that gives existence of generalized
minimizers. In section \ref{sec:case-many-charges}, we present the
proofs of the existence result of Theorem \ref{t:exist}, the
non-existence result of Theorem \ref{t:non}, and the asymptotic
characterization of minimizers with many charges in Theorem
\ref{t:ball}. Lastly, in section \ref{s:two} we present the analysis
of the two-charge problem that yields Theorem \ref{t:2}. This section
also gives an explicit characterization of the energy minimizers.

\paragraph{Acknowledgements} The work of C.B.M. was supported, in
part, by NSF via grant DMS-1908709. M. N. was supported by the PRIN
2022 Project 2022E9CF89. C. B. M. and M. N. are members of
INdAM/GNAMPA and acknowledge the MUR Excellence Department Project
awarded to the Department of Mathematics, University of Pisa, CUP
I57G22000700001. This study also received funding from the European
Union – Next Generation EU – PRIN 2022 PNRR Project P2022WJW9H.

\section{Model}
\label{sec:model}

We consider a system consisting of two immiscible fluids with matched
dielectric constants, i.e., both fluids have the relative dielectric
constant equal to $\eps_d$. Because of this, we do not need to worry
about the shape dependent dielectric polarization of the liquid drop
in the presence of charges, which would otherwise considerably
complicate the analysis \cite{dephilippis23}. In the following, we
simply refer to the first fluid of finite volume surrounded by the
second ambient fluid as the {\em liquid drop}. A notable example of
such a fluid system is liquid helium in equilibrium with its vapor,
which has been used to investigate the phenomenon of Wigner
crystallization of charges at the liquid-vapor interface and is known
to undergo charge-driven interfacial instabilities
\cite{grimes79,gorkov73,tsao98,abdurahimov12}. More recently,
charge-containing helium nanodroplets have been considered as a host
medium for a variety of applications in molecular spectroscopy and
quantum chemistry \cite{barranco06}.

At the level of the continuum, the equilibrium shape of a charged,
perfectly conducting liquid drop may be investigated with the help of
a model that goes back over 140 years to Lord Rayleigh
\cite{rayleigh1882}. In this model, an equilibrium drop is viewed as a
minimizer (at least local) of the energy
\begin{align}
  \label{eq:E0}
  \mathcal E(\Omega) := \sigma P(\Omega) + {Q^2 \over 2 C(\Omega)}, 
\end{align}
where $\Omega \subset \mathbb R^3$ is the set occupied by the drop
that carries the charge $Q$, with the volume of the drop
$|\Omega| = m$. Here, $\sigma$ is the surface tension of the liquid
interface, $P(\Omega)$ is the perimeter of the set $\Omega$ defined by
\begin{equation}\label{aper} 
  P(\Omega) :=    \sup \left\{ \int_\Omega \nabla \cdot \phi (y) \, dy:
    \, \phi \in C^1_c(\mathbb R^3;\mathbb R^3), \ |\phi|
    \leq 1 \right\},
\end{equation}
which is a suitable measure-theoretic generalization of the surface
measure for smooth sets, and $C$ is the electrostatic capacity defined
by
\begin{align}
  \label{eq:C0}
  C^{-1}(\Omega) := \inf_{\mu(\Omega) = 1} 
  \int_\Omega \int_{\Omega} {1 \over 4 \pi \eps_0 \eps_d |x - y|} \, d \mu(x)
  \, d \mu(y), 
\end{align}
where $\eps_0$ is the permeability of vacuum, and the minimization is
carried out over probability measures $\mu$ supported on
$\Omega$. However, as was already mentioned, this model was recently
shown to be variationally ill-posed \cite{goldman15,mn:prsa16}. Thus,
a regularization of the electrostatic problem is necessary to enable
existence of even local energy minimizers in the natural classes of
liquid configurations.

In this paper, we appeal to the discrete nature of electric charges as
a possible physical regularizing mechanism
\cite{iribane76,fenn93,labowsky98,kebarle00}, while ignoring the
entropic effects associated with thermal agitation of the charges
(appropriate for nanoscale droplets or cryogenic fluids). This amounts
to restricting the measures appearing in \eqref{eq:C0} to those
associated with $N$ point charges:
\begin{align}
  \label{eq:muN}
  \mu = {1 \over N} \sum_{i=1}^N \delta_{x_i},
\end{align}
where $x_i \in \mathbb R^3$ are the positions of the charges and
$\delta_{x_i}$ are the Dirac delta-measures centered at $x_i$. Note
that in doing so we must exclude the self-interaction of charges.
Setting $x \not= y$ in the integral in \eqref{eq:C0} then yields the
following expression for the energy:
\begin{align}
  \label{eq:EN0}
  \mathcal E_N(\Omega, X) := \sigma P(\Omega) + {e^2 \over 8 \pi \eps_0
  \eps_d} \sum_{i \not= j} {1 \over |x_i - x_j|}.
\end{align}
Here the set $\Omega \subset \R^3$ again denotes the domain occupied
by the liquid drop, the discrete set
$X = \cup_{i=1}^N \{ x_i \} \subset \R^3$ specifies the positions of
$N$ point charges, and $e$ is the elementary charge (positive), so
that $|Q| = N e$. For simplicity, we assume a single species of
monovalent ions dissolved in the liquid drop, with the ambient fluid a
perfect dielectric.

Notice that every charge in the liquid drop strongly attracts a
cluster of liquid (solvent) molecules forming a {\em solvation shell}
around the charge (ion). We model this effect by requiring that the
liquid drop contains a ball of radius $r_0$, called the solvation
radius, around each charge \cite{iribane76}, i.e., we have
$B_{r_0}(x_i) \subset \Omega$ for each $i = 1, \ldots, N$, with
$B_{r_0}(x_i)$ mutually disjoint. The solvation radius of simple
monoatomic ions in polar solvents like water usually measures to
fractions of a nanometer.

To assess the relative strengths of the two terms in the energy and to
carry out an appropriate non-dimensionalization, we introduce the
molecular length scale
\begin{align}
  \label{eq:rsigma}
  r_\sigma := \sqrt{k_B T \over \sigma},   
\end{align}
where $k_B T$ is the temperature in the energy units, above which the
interface may be considered as sharp and well defined in the presence
of thermal noise. For low molecular weight liquids at room
temperature, $r_\sigma$ is on the order of a fraction of a
nanometer. This scale may be compared with the Bjerrum length
\begin{align}
  \label{eq:rBjerrum}
  r_B := {e^2 \over 4 \pi \eps_0 \eps_d k_B T},
\end{align}
which measures the scale at which the Coulombic energy of a pair of
elementary charges in a dielectric liquid is comparable to the thermal
energy. In polar solvents at room temperature, this length is on the
order of a few nanometers.  Rescaling lengths with $r_\sigma$ and
measuring the energy in the units of $k_B T$ then yields
$\mathcal E_N(r_\sigma \Omega, r_\sigma X) = k_B T E_{\rho,\lambda,N}
(\Omega, X)$, where
\begin{align}
  \label{eq:EN}
  E_{\rho,\lambda,N}(\Omega, X) := P(\Omega) + \lambda \sum_{i =
  1}^{N-1} \sum_{j=i+1}^N {1 \over |x_i - x_j|},
\end{align}
now with $B_\rho(x_i) \subset \Omega$ disjoint, where we introduced
the dimensionless parameters
\begin{align}
  \label{eq:lambdarho}
  \rho := {r_0 \over
  r_\sigma}, \qquad \lambda := {r_B \over
  r_\sigma}. 
\end{align}
From the basic physical considerations already mentioned, for typical
liquids at room temperature both $\rho$ and $\lambda$ are expected to
be of order one \cite{iribane76}. For example, for small monovalent
ions in ethanol (a common solvent for electrospray) we have
$\rho \approx 1$ and $\lambda \approx 5$. In contrast, for liquid
helium at $T = 2$ K, for which $r_\sigma \approx 0.3$ nm and
$r_B \approx 8 \, \mu$m we get $\rho \sim 1$ and
$\lambda \sim 10^6 \gg 1$. As a point of reference, let us note that
for the parameters of liquid helium above our Theorem \ref{t:2} yields
existence of an equilibrium configuration only for droplets whose
volume corresponds to a ball of radius greater than $\sim 10 \, \mu$m
even with just two point charges.

%

The case of the main physical interest corresponds to that of the
volume of the charged drop becoming macroscopically large
($m \to \infty$), while the number of charges $N$ simultaneously tends
to infinity with a suitable rate. To study this regime, we can carry
out another rescaling in which the volume is instead normalized to a
constant while the radius of the solvation sphere
vanishes. Introducing the parameter $\eps > 0$ that will eventually be
sent to zero, we have
$E_{\rho,\lambda,N} (\eps^{-1} \rho \Omega, \eps^{-1} \rho X) =
\eps^{-2} \rho^2 E_\eps(\Omega, X)$, where
\begin{align}
  \label{eq:Eeps}
  E_\eps(\Omega, X) := P(\Omega) + \gamma \eps^3 \sum_{i
  = 1}^{N_\eps-1} \sum_{j=i+1}^{N_\eps} {1 
  \over |x_i - x_j|}, 
\end{align}
$B_\eps(x_i) \subset \Omega$ are disjoint for all
$1 \leq i \leq N_\eps$, and $\gamma := \lambda / \rho^3$ is a single
dimensionless parameter that characterizes the physical properties of
the liquid and is kept fixed throughout the analysis. The
considerations following \eqref{eq:lambdarho} motivate us to focus on
the physically most relevant regime of $\gamma \gtrsim 1$. The
assumptions on the dependence of $N_\eps \to \infty$ on $\eps \to 0$
that yield information about the equilibrium shape of the charged
drops turn out to be non-trivial and will be specified in the
following sections.

\section{Main results}
\label{sec:main-results}

We now state the main results of our paper concerning the minimizers
of the energy $E_{\rho,\lambda,N}$ and its rescaled version
$E_\eps$. We begin by defining the admissible class
$\mathcal A_{m, N, \rho}$ of configurations consisting of a set of
finite perimeter $\Omega \subset \R^3$ of volume $m > 0$ and
$N \in \mathbb N$ non-overlapping charges of radius $\rho > 0$
contained in $\Omega$, whose centers are collected into a discrete
set $X \subset \R^3$:
\begin{equation}
  \label{AmNr}
  \begin{aligned}
    \mathcal A_{m,N,\rho} := \{ & (\Omega, X) \ : \\
    & \Omega \subset \mathbb R^3 \ \text{measurable}, \ |\Omega| = m, 
    \ P(\Omega) <
    \infty, \\
    & X = \cup_{i=1}^N \{x_i\}, \ (x_i)_{i=1}^N \in \mathbb
    R^3, \\
    & |\Omega \cap B_\rho(x_i)| = |B_\rho(0)| \ \text{for all} \ 1
    \leq
    i \leq N,  \\
    & B_\rho(x_i) \cap B_\rho(x_j) = \varnothing \ \text{for all} \ 1
    \leq i < j \leq N\}.
\end{aligned}
\end{equation}
An example of an admissible configuration is shown in
Fig. \ref{f:example}. Notice that the set $\cup_{i=1}^N B_\rho(x_i)$
representing the charges is assumed to be contained inside the set
$\Omega$ in the measure theoretic sense.

\begin{figure}
    \centering
\includegraphics[scale=1.1]{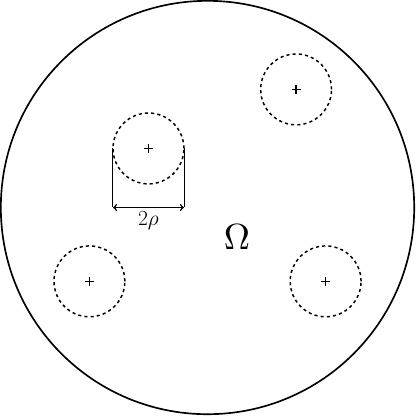}
\caption{Schematics of an admissible configuration consisting of
  $N = 4$ charges indicated with ``+'' when $\Omega$ is a ball.}
  \label{f:example}

\end{figure}

We would like to investigate under which conditions the energy
$E_{\rho, \lambda, N}$ admits a minimizer in the class
$\mathcal A_{m, N, \rho}$. Notice that the question of existence of
such minimizers is far from obvious because of the possibility of
splitting of the set $\Omega$ into disjoint pieces that carry the
charges apart to lower the Coulombic energy at the expense of
increasing the interfacial energy. This issue is well known in the
studies of geometric variational problems with competing interactions
\cite{cmt:nams17}. In the context of Gamow's liquid drop model, it was
shown that an appropriate extension of the notion of minimizers for
this kind of problems is given by {\em generalized minimizers}
\cite{kmn:cmp16}. In our problem, these are defined as follows.

\begin{definition}
  \label{d:gen}
  Let $\rho,\,\lambda>0$, $N\in\mathbb N$ and
  $m\ge \frac{4\pi}{3} N \rho^3$. Suppose there exists
  $K\in\mathbb N$, $m_k>0$ and $N_k\in\mathbb N \cup \{0\}$ with
  $m=\sum_{k=1}^K m_k$, $N=\sum_{k=1}^K N_k$, and a family of
  minimizers $(\Omega_k,X_k)\in \mathcal A_{m_k,N_k,\rho}$ of
  $ E_{\rho,\lambda,N_k}$ which satisfies
  \begin{align}
    \sum_{k=1}^K  E_{\rho,\lambda,N_k}(\Omega_k, X_k) = \inf_{(\Omega,
    X) \in \mathcal A_{m,N,\rho}}  E_{\rho,\lambda,N}(\Omega, X).
  \end{align}
  Then the family of $(\Omega_k, X_k)$ is called a generalized
  minimizer of $E_{\rho, \lambda, N}$ over $\mathcal A_{m,N,\rho}$. 
\end{definition}

Intuitively, a generalized minimizer can be thought of as a {\em
  finite} collection of droplets containing all the charges, with each
droplet being a minimizer for the charge it contains and different
droplets being ``infinitely far apart'' and thus not interacting. Each
set $\Omega_k$ in a generalized minimizer is referred to as a {\em
  component}. Notice that a generalized minimizer is simply a
minimizer if and only if it has only one component.  An illustration
of a classical minimizer of $E_{\rho, \lambda, N}$ with $N = 7$ is
presented in Fig. \ref{f:class}, while a possible generalized
minimizer is shown in Fig. \ref{f:gen}. Our first result establishes
existence of generalized minimizers for all nontrivial values of the
parameters.

\begin{figure}
    \centering
     
\includegraphics[scale=1.1]{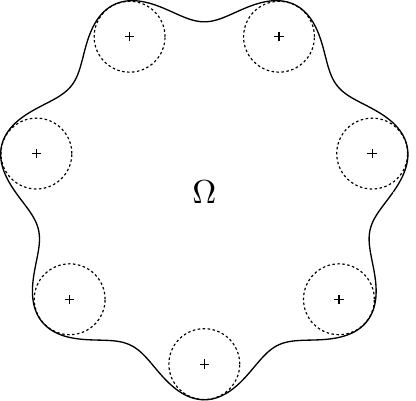}

\caption{A schematic of a classical minimizer for $N = 7$. }
\label{f:class}
\end{figure}

\begin{figure}
    \centering
     
\includegraphics[scale=1.1]{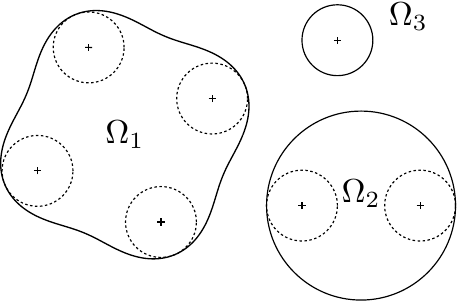}

\caption{A schematic of a generalized minimizer with $K = 3$ for
  $N = 7$.}
\label{f:gen}
\end{figure}

\begin{theorem}\label{exgen}
  Let $\rho,\,\lambda>0$, $N\in\mathbb N$ and
  $m\ge \frac{4\pi}{3} N \rho^3$. Then there exists a generalized
  minimizer of $E_{\rho, \lambda, N}$ over $\mathcal A_{m,N,\rho}$.
  Moreover, each component of the generalized minimizer has boundary
  of class $C^{1,1}$, is bounded and connected, and away from the set
  of charges is smooth and has constant mean curvature that is the
  same for all the components.
\end{theorem}

In view of the regularity of the components of generalized minimizers,
in the following we always refer to the regular representatives when
talking about the energy minimizing sets. In particular, we can choose
these sets to be open.

We next establish a parameter regime in which the generalized
minimizers are also classical, i.e., when there is a minimizer of
$E_{\rho, \lambda, N}$ over $\mathcal A_{m,N,\rho}$. Naturally, as the
most interesting case to consider is that of many charges, we will
instead work with the energy $E_\eps$ defined in \eqref{eq:Eeps} and
minimize it over the class $\mathcal A_\eps$ obtained as a suitable
modification of the definition in \eqref{AmNr} corresponding to sets
of volume $m = {4 \pi \over 3}$ of a unit ball (without loss of
generality) containing $N_\eps \gg 1$ charges of radius $\eps \ll 1$:
\begin{align}
  \label{eq:Aeps}
  \mathcal A_\eps := \mathcal A_{{4 \pi \over 3}, N_\eps, \eps}.
\end{align}

Notice that existence vs. non-existence of classical minimizers in the
class $\mathcal A_\eps$ for $\eps \ll 1$ must clearly depend on the
rate of $N_\eps \to \infty$ as $\eps \to 0$. To begin with, due to the
constraint $B_\eps(x_i) \cap B_\eps(x_j) = \varnothing$ for
$i \not= j$ we must have $N_\eps \lesssim \eps^{-3}$ in order for the
admissible class $\mathcal A_\eps$ to be non-empty, limiting the
possible growth rate of $N_\eps$. On the other hand, for $\Omega$
fixed and $\eps$ sufficiently small depending on $N_\eps \gg 1$, one
would be able to approximate
$\inf_{X \subset \Omega} E_\eps(\Omega, X)$ by
\begin{align}
  \label{eq:Econt}
  E_0(\Omega) := P(\Omega) + {q^2 \over 2} \inf_{\mu(\Omega)
  = 1} \int_\Omega \int_\Omega {d \mu(x) \, d \mu(y) \over |x - y|},
\end{align}
with $q = \gamma^{\frac12} \eps^{\frac32} N_\eps$, which is nothing
but the dimensionless continuum energy in \eqref{eq:E0}. Nevertheless,
this energy is known to give
$\inf_{|\Omega| = {4 \pi \over 3}} E_0(\Omega) = 4 \pi$ for all
$q \geq 0$, thus failing to yield a minimizer for any $q > 0$
\cite{goldman15}. Therefore, it would be natural to expect existence
of minimizers of $E_\eps$ over $\mathcal A_\eps$ only for
$N_\eps \ll \gamma^{-\frac12} \eps^{-\frac32}$. Still, the threshold
$N_\eps$ for existence of minimizers in this regime is far from
obvious.

We begin with the following existence result, which shows that for
$\gamma \gtrsim 1$ classical minimizers exist as soon as
$N_\eps \lesssim \gamma^{-1} \eps^{-1}$ and all $\eps > 0$
sufficiently small universal.

\begin{theorem}
  \label{t:exist}
  There exist universal constants $\eps_0 > 0$, $\gamma_0 > 0$ and
  $C>0 $ such that for all $\gamma > \gamma_0$ and
  $1 < N_\eps < \frac{C}{\varepsilon \gamma}$ there exists a minimizer
  of $E_\eps$ over $\mathcal A_\eps$ for all $\eps \in (0,
  \eps_0)$. Furthermore, if $(\Omega, X) \in \mathcal A_\eps$ is a
  minimizer of $E_\eps$ then
  $\mathrm{dist}(x_i, \partial \Omega) = \eps$ and
  $\mathrm{dist}(x_i, X \backslash x_i) \geq c \gamma \eps$ for all
  $x_i \in X$ with $1 \leq i \leq N_\eps$ and $c > 0$ universal.
\end{theorem}

We note that one of the conclusions of the above theorem is that all
the balls $B_\eps(x_i)$ containing the charges $x_i \in X$ in a
minimizer touch the drop boundary $\partial \Omega$. This is
consistent with the expectation at the level of the continuum that the
measure $\mu$ minimizing the Coulombic energy in \eqref{eq:Econt} is
supported on $\partial \Omega$. Furthermore, we find that in this
regime the charges are uniformly separated from one another at scale
$O(\gamma \eps)$ which exceeds that imposed by the constraint
$B_\eps(x_i) \cap B_\eps(x_j) = \varnothing$ for $i \not= j$.

Surprisingly, the existence threshold in Theorem \ref{t:exist} is
considerably lower than
$N_\eps \sim \gamma^{-\frac12} \eps^{-\frac32}$ for which the
Coulombic energy matches the perimeter in the continuum as
$\eps \to 0$, see \eqref{eq:Econt}. Nevertheless, this is not simply a
limitation of our analysis, as we demonstrate with our next
non-existence result. To give some heuristics for the threshold
appearing in Theorem \ref{t:exist}, consider the basic mechanism in
which a drop may lose its energy minimizing property by {\em
  evaporating} a single charge
\cite{iribane76,fenn93,labowsky98,kebarle00}. If $(\Omega, X)$ is a
minimizer of $E_\eps$, then $(\Omega', X')$ with
$\Omega' = (\Omega \backslash B_\eps(x_i)) \cup B_\eps(R e_1)$ and
$X' = (X \backslash \{x_i\}) \cup \{R e_1\}$, obtained by cutting a
single ball $B_\eps(x_i)$ with a charge in its center and sending it
far off, is an admissible configuration. Here $e_1$ is the unit vector
along the first coordinate direction, $x_i \in X$ with
$1 \leq i \leq N_\eps$ arbitrary, and $R > 0$ is sufficiently large.
Letting $R \to \infty$, we then conclude that
\begin{align}
  \label{eq:evapor}
  E_\eps(\Omega, X) \leq E_\eps(\Omega', X') \leq E_\eps(\Omega, X) +
  8 \pi \eps^2 - \gamma \eps^3 \sum_{j \not= i} {1 \over |x_i -
  x_j|},
\end{align}
which implies that
\begin{align}
  \label{eq:diamNeps}
  \mathrm{diam}(\Omega) \geq C \gamma \eps N_\eps,
\end{align}
for some $C > 0$ universal and all $N_\eps > 1$.

We would expect that at least whenever the perimeter is not
overwhelmed by the Coulombic energy the diameter of a minimizer of
$E_\eps$, if it exists, should not greatly exceed that of a unit ball
corresponding to the mass constraint. From this and
\eqref{eq:diamNeps}, we immediately get a contradiction if
$N_\eps \gg \gamma^{-1} \eps^{-1}$, suggesting that in this regime the
existence should fail, provided that the perimeter term indeed
dominates the Coulombic energy. For the latter, we can consider a
competitor of the form $(\Omega, X)$, where
$\Omega = B_r(0) \cup_{i=1}^{N_\eps} B_\eps(i R e_1)$ and
$X = \cup_{i=1}^{N_\eps} \{i R e_1\}$, for $r^3 + \eps^3 N_\eps = 1$
with $\eps \ll 1$ and $R \gg 1$, corresponding to {\em all} charges
evaporated from the drop. This yields
$\inf_{\mathcal A_\eps} E_\eps \leq 4 \pi (r^2 + \eps^2 N_\eps)$ by
sending $R \to \infty$. Thus, we have
$\inf_{\mathcal A_\eps} E_\eps \lesssim 1$ whenever
$N_\eps \lesssim \eps^{-2}$ and $\eps \ll 1$ independently of
$\gamma$, and the isoperimetric deficit becomes small when
$N_\eps \ll \eps^{-2}$.

Under the condition of smallness of $\eps^2 N_\eps$, we now get our
non-existence result that yields a sharp scaling for the threshold
value of $N_\eps$ with $\gamma \gtrsim 1$ for $\eps \ll 1$.

\begin{theorem}
  \label{t:non}
  Let $\gamma > \gamma_0$, where $\gamma_0$ is as in Theorem
  \ref{t:exist}. Then there exists a universal constant $C>0$ and
  constants $\varepsilon_0, \delta_0 >0$ depending only on $\gamma$
  such that if $\varepsilon \in (0, \varepsilon_0)$ and
  $\frac{C}{\gamma \varepsilon } <N_\eps <
  \frac{\delta_0}{\varepsilon^2}$ then $E_\eps$ does not attain its
  infimum in $\mathcal A_\eps$.
\end{theorem}

We note that by \eqref{eq:diamNeps} the minimizer is expected to be
highly elongated for $N_\eps \gtrsim \eps^{-2}$, if it exists. Thus,
although we do not believe minimizers could exist in this Coulombic
dominated regime far beyond Rayleigh instability, i.e., for all
$N_\eps \gg \gamma^{-\frac12} \eps^{-\frac32}$, a different
approach would be needed to rule out existence of minimizers in this
regime.

We now turn to the asymptotic behavior of the minimizers in the range
of existence given by Theorem \ref{t:exist}. In the next theorem, we
show that when the minimizers of $E_\eps$ exist for $\eps \ll 1$, they
are always nearly spherical with a uniformly distributed charge over
the boundary, as one would have expected on physical grounds. Notice
that as was already mentioned, in the regime of Theorem \ref{t:exist}
the isoperimetric deficit for minimizers vanishes as $\eps \to 0$,
which by the quantitative isoperimetric inequality implies that the
minimizers converge to balls in the $L^1$ topology after suitable
translations \cite{fusco08}. Nevertheless, a stronger control on the
deviation of the minimizer $\Omega$ from a ball is necessary to
establish convergence of the Coulombic energy and, as a result, of the
charge density, which is given by the following theorem.

\begin{theorem}
\label{t:ball}
Let $\eps_n > 0$ and $N_n \in \mathbb N$ be such that $\eps_n \to 0$
and $N_n \to \infty$ as $n \to \infty$, and
$N_n < {C \over \gamma \eps_n}$ for $\gamma > \gamma_0$, where $C$ and
$\gamma_0$ are as in Theorem \ref{t:exist}. Then if
$(\Omega_n, X_n) \in \mathcal A_{\eps_n}$ are minimizers of
$E_{\eps_n}$ and $X_n = \cup_{i=1}^{N_n} \{x_{i,n}\}$, we have, up to
translations, $\Omega_n \subset B_{1+\delta}(0)$ for all $\delta > 0$
and all $n \in \mathbb N$ large enough, and
\begin{align}
  \label{eq:measconv}
  {1 \over N_n} \sum_{i=1}^{N_n} \delta_{x_{i,n}} \rightharpoonup {1 \over 4 \pi}
  \mathcal H^2\lfloor_{\partial B_1(0)},
\end{align}
in the sense of measures, as $n \to \infty$.
\end{theorem}

Lastly, we present an asymptotically sharp existence result for the
minimization problem in the special case of $N_\eps = 2$ charges and
$\eps \ll 1$. Actually, in this case the minimization problem admits
and explicit solution in terms of the unduloid surfaces that span the
space between the two charges. We present the rather technical details
of these solutions in section \ref{s:two}. Here instead we summarize
our existence results for minimizers of $E_\eps$ with $N_\eps = 2$ for
$\eps \ll 1$.

\begin{theorem}
  \label{t:2}
  Let $N_\eps = 2$ and $c > 0$. Then there exists $\eps_0 > 0$ such
  that for all $\eps \in (0, \eps_0)$ we have:
  \begin{enumerate}[(i)]
  \item if $c < 8 \pi$ and $\gamma < {c\over \eps}$ then there exists
    a unique, up to translations and rotations, minimizer of $E_\eps$
    in $\mathcal A_\eps$.
    
  \item if $c > 8 \pi$ and $\gamma > {c\over \eps}$ then there is no
    minimizer of $E_\eps$ in $\mathcal A_\eps$.
  \end{enumerate}
  
\end{theorem}

Note that the threshold for existence in the above theorem is
consistent with the one found in Theorems \ref{t:exist} and
\ref{t:non}, but without an a priori assumption on $\gamma$.  A
further quantitative characterization of these minimizers is presented
in Theorem \ref{Theorem_main2}, with all the necessary notations
defined in section \ref{s:two}.  The proof of the latter is rather
technical and involves a careful asymptotic analysis of the exact
global minimizers constructed in that case. Finally, we note that in
the case $N_\eps = 1$ the minimizers are trivially balls, so in the
following we can always assume $N_\eps \geq 2$ without loss of
generality.

  \subsection{Structure of the proof}

  Our starting point is establishing existence of generalized
  minimizers in Theorem \ref{exgen}, which is done with the help of
  the standard concentration compactness argument that exploits
  uniform regularity for volume-constrained minimizers of the
  perimeter in Lemma \ref{density}, a result which could be of
  independent interest in itself.

  Our next result is existence of classical minimizers of $E_\eps$
  when $\gamma$ is sufficiently large universal, $\eps$ is
  sufficiently small and the number of charges $N_\eps$ is not too
  large. Here, arguing by contradiction, we assume that no classical
  minimizer exists and first show that in the considered parameter
  regime all generalized minimizers consist of only one large
  component and one or more balls of radius $\eps$ each containing one
  charge, see Lemma \ref{Min_class}. Once this result is established,
  the proof of Theorem \ref{t:exist} follows from a sharp quantitative
  estimate of the Coulombic energy of a finite number of charges on a
  unit ball, together with an estimate of the energy gain resulting
  from suitably merging one small component of the generalized
  minimizer with the large component, which leads to a
  contradiction. Conversely, for the number of charges exceeding the
  existence threshold for $N_\eps$ in Theorem \ref{t:exist} we obtain
  a contradiction to the existence of a classical minimizer in Theorem
  \ref{t:non} by combining the characterization of minimizers in Lemma
  \ref{Outter_Ball} with a sharp lower bound on the Coulombic energy
  of finitely many charges confined to a ball. Furthermore, within the
  regime of validity of Theorem \ref{t:exist} we are able to pass to
  the limit $\eps \to 0$ and $N_\eps \to \infty$ with a suitable rate
  in Theorem \ref{t:ball} by combining the well-known
  $\Gamma$-convergence results for the Coulombic energy of many point
  charges with the asymptotic characterization of minimizers in Lemma
  \ref{Outter_Ball}.

  The key technical result used in the proofs is contained in Lemma
  \ref{Outter_Ball}, whose proof with some minor modifications is also
  used to obtain the diameter bounds for the components of generalized
  minimizers in Lemma \ref{Min_Component} and, as a consequence, a
  characterization of generalized minimizers in Lemma
  \ref{Min_class}. Lemma \ref{Outter_Ball} states that any classical
  minimizer of $E_\eps$ is close to a ball in a certain sense, at
  least for $N_\eps$ not too large. The proof utilizes an upper bound
  on the total mass of the minimizer contained outside a ball of
  radius slightly greater than 1 in Lemma \ref{Ball} with the uniform
  density estimate from Lemma \ref{density_est} and uses careful
  cutting arguments, together with a cutting estimate on the perimeter
  in Lemma \ref{Cutting_perimeter} and the connectivity of minimizers
  to show that every classical minimizer of $E_\eps$ in the considered
  range of $N_\eps$ is contained in some ball of radius arbitrarily
  close to 1. The rest of the lemmas of section
  \ref{sec:case-many-charges} establish the basic estimates on
  generalized minimizers that are used throughout the rest of the
  proofs.

  Finally, in section \ref{s:two} we prove the result in Theorem
  \ref{t:2} by first establishing existence of minimizers from the
  results of section \ref{sec:case-many-charges} and then showing that
  any minimizer of $E_\eps$ with $N = 2$ is an axisymmetric set, see
  Lemma \ref{rotational}, that for small $\eps$ is close in a certain
  sense to a unit ball, see Lemma \ref{N2bbd}. From that we conclude,
  in Lemma \ref{unduloid_arc}, that the free surface of the minimizer
  is a single section of an unduloid, which is quantified in several
  subsequent lemmas. The proof is concluded by enumerating the
  possibilities of joining the free surface with the obstacles and
  comparing their energies, based on careful expansions of the
  different contributions to the energy for $\eps$ small. The
  conclusion is then given in Theorem \ref{Theorem_main2} that
  provides a sharp asymptotic characterization of the minimizer in the
  considered regime.

\section{Existence of generalized minimizers}
\label{sec:gener-minim}

Following \cite{dephilippis23}, we first derive a uniform density
estimate for volume constrained minimizers of the perimeter in the
presence of spherical obstacles. The starting point of our analysis is
an Almgren type lemma that provides existence of a diffeomorphism
whose constants depend only on the volume and the perimeter of the
set, but not on the set itself, which is then used to compensate the
changes of volume under small perturbations of the set.

\begin{lemma}\label{field}
  For every $P>0$ there exists $\gamma>0$ such that if
  $\Om\subset \mathring{B}_R^c(0)$ for some $R > 0$ and
\begin{equation}\label{colcol}
|\Om|\ge 1,\qquad P(\Om)\le P,
\end{equation}
then there exists a vector field
$\eta\in C^1_c(\mathring{B}_R ^c (0))$ with
$\|\eta\|_{C^1(\mathring{B}_R ^c (0))}\le 1$ such that
\begin{equation}\label{eqeta}
  \int_\Om {\rm div} \, \eta\,dx \geq \gamma.
\end{equation}
\end{lemma}

\begin{proof}
  We reason as in \cite[Lemma 3.5]{dephilippis23} and assume by
  contradiction that there exist a sequence of radii $R_k > 0$ and a
  sequence of sets $\Om_k$ satisfying \eqref{colcol} such that
\begin{equation}\label{eqcontra}
  \lim_{k \to \infty} \sup_{\eta\in \mathcal A_k}\,\int_{\Om_k} {\rm
    div} \, \eta\,dx = 0, 
\end{equation}
where
$\mathcal A_k:=\{ \eta\in C^1_c(\mathring{B}_{R_k}^c (0))\text{ such
  that } \|\eta\|_{C^1(\mathring{B}_{R_k}^c (0))}\le 1\}$.  By
\cite[Remark 29.11]{maggi} for all $k \in \mathbb N$ there exists
$x_k\in \R^3$ such that
\begin{equation}\label{ddt}
|\Om_k\cap B_{1}(x_k)|\ge \bar \delta,
\end{equation}
with $\bar\delta = \bar\delta(P)>0$. Letting $F_k = \Om_k-x_k$, up to
a subsequence we have that $R_k\to R\in [0,+\infty]$,
$F_k\to F\subset \R^3$ in $L^1_{loc}(\R^3)$, with $P(F)\le P$. We only
deal with the case $R<+\infty$ and $x_k\to x \in \R^3$, since the other
cases can be treated analogously and are easier.

Passing to the limit in \eqref{ddt}, we get that
$|F\cap B_1(0)|\ge \bar \delta$.  In particular, by Almgren lemma
\cite[Lemma 3.4]{dephilippis23} (see also \cite{maggi,almgren76})
there exists $\eta_F\in C^1_c(\mathring{B}_R ^c (-x))$ with
$\|\eta_F\|_{C^1(\mathring{B}_R ^c (-x))}\le 1$ such that
\begin{align}
    \int_F {\rm div} \, \eta_F \,dx \geq \gamma_F,
\end{align}
for some $\gamma_F>0$. Letting now $\eta_k :=\eta_F(\cdot+x_k)$, which
belongs to $C^1_c(\mathring{B}_{R_k}^c (0))$ for $k$ large enough, we
have that
\begin{align}
  \lim_{k \to \infty} \int_{\Om_k} {\rm div} \, \eta_k\,dx \geq
  \gamma_F>0,
\end{align}
thus contradicting \eqref{eqcontra}.
\end{proof}

\begin{remark}
  It is not difficult to see that the conclusion of Lemma \ref{field}
  in fact holds for general sets $\Om \subset \R^n$ of finite
  perimeter and supported on a complement of a bounded open set
  $U \subset \R^n$, with the constant $\gamma$ depending only on the
  perimeter of $\Omega$ and $n$.
\end{remark}

From Lemma \ref{field}, we derive the following uniform density
estimate.

\begin{lemma}\label{density}
For every $\eta,\delta > 0 $ there exist a $c_0>0$ depending only on $\eta$ such that if $\Om\subset \R^3$ is
  a minimizer of
\begin{align}
  \min\big\{P(E,B_R ^c (0))\,:\, E\cap B_R(0)=\Om\cap
  B_R(0),\,|E|=|\Om| \big\},
\end{align}
with $P(\Omega \cap B_R^c(0)) < \eta  |\Omega \cap B_R^c(0)|^{2/3} $ and $|\Omega \cap B_R^c(0)| > \delta$,
then
\begin{align}
|\Om\cap B_r(x)|\ge c r^3,
\end{align}
for all $x\in \overline\Om\setminus B_R(0)$ and
$r\in (0,c_0 \delta^\frac13)$ such that $B_r(x)\subset B_R ^c (0)$,
where the constant $c>0$ is universal and $\overline \Omega$ is
understood in the measure theoretic sense.
\end{lemma}

\begin{proof}
  Up to a rescaling, we can assume that $\delta = 1$. Notice that by a
  projection argument we have
  \begin{align}
    P(\Omega, B_R^c(0)) +  \mathcal H^2(\Omega \cap \partial
    B_R(0)) = P(\Omega \cap B_R^c(0)) \geq 2 \mathcal H^2(\Omega \cap
    \partial B_R(0)).
  \end{align}
  Then reasoning as in \cite[Proposition 4.4]{dephilippis23} and
  applying Lemma \ref{field} to the set $\Omega \cap B_R^c(0)$, we get
  that $\Om$ is a $(\Lambda,r_0)$-minimizer of the perimeter in
  $B_R ^c (0)$, where $\Lambda, r_0$ are positive constants depending
  only on $P$.  The result then follows by \cite[Theorem
  21.11]{maggi}.
\end{proof}


\begin{proof}[Proof of Theorem \ref{exgen}]
Let $(\Omega_n, X_n)$ be a minimizing sequence and let
$X_n = \cup_{i=1}^N \{x_{i,n}\}$. As the total number of charges is
fixed, up to extraction of a subsequence (not relabeled) the charges
segregate into $1 \leq K \leq N$ clusters moving apart as
$n \to \infty$. More precisely, for each $k \in \{1, 2, \ldots, K\}$
there exist $N_k \in \mathbb N$ and an index set
$I_k = \{i^k_1, i^k_2, \ldots, i^k_{N_k}\}$ such that
$\cup_{k=1}^K I_k$ forms a disjoint partition of $\{1, \ldots, N\}$
for each $n \in \mathbb N$ and
\begin{align}
  \label{limsupX}
  & \limsup_{n \to \infty} \left| x_{i,n} - x_{j,n} \right|
    < \infty \qquad \forall i \in I_k \ \text{and} \ \forall j \in 
    I_k, \\
  \label{liminfX}
  & 
    \liminf_{n \to \infty} \left| x_{i,n} - x_{j,n} \right|
    = \infty \qquad \forall i \in I_k \ \text{and} \ \forall j
    \not\in I_k. 
\end{align}

Consider now $\Omega_n^k := \Omega_n - x_{i^k_1,n}$ and
$X_n^k := \cup_{i \in I_k} \{ x_{i,n} - x_{i^k_1,n} \}$. By
\eqref{limsupX} and \eqref{liminfX}, there exists $R_0 \ge 1$ such
that $B_\rho(x_{i,n}) \subset B_{R_0}(x_{i_1^k,n})$ for all
$i \in I_k$ and all $1 \leq k \leq K$, and for every
$\widetilde R > 0$ we have
$B_\rho(x_{i,n}) \subset B_{\widetilde R}^c(x_{i_1^k,n})$ for all
$i \not\in I^k_n$ and all $n$ large
enough. 
Then, for $R_0 < R < \widetilde R$ and $L > 0$ we define a competitor set
\begin{align}
  \label{OmRL}
  \widetilde \Omega_n^{R,L} := \left( \bigcup_{k=1}^K (\Omega_n^k \cap
  B_R(0)) + e_1 k L) \right) \bigcup \Omega_n^0,
\end{align}
where for
$r := \left( {3 \over 4 \pi} |\Omega_n \setminus (\cup_{k=1}^K
  B_R(x_{i_1^k,n}))| \right)^{1/3}$ we have
$\Omega_n^0 := \varnothing$ if $r = 0$ or $\Omega_n^0 := B_r(0)$ if
$r > 0$, together with
\begin{align}
  \label{XRL}
  \widetilde X_n^{R,L} = \bigcup_{k=1}^K (X_n^k + e_1 k L). 
\end{align}

By construction,
$(\widetilde \Omega_n^{R,L}, \widetilde X_n^{R,L}) \in \mathcal
A_{m,N,\rho}$ for all $n$ and $L$ large enough independent of
$R$. Notice that
\begin{align}
  P(\widetilde \Omega_n^{R,L}) = \sum_{k=1}^K P(\Omega_n^k, B_R(0)) +
  \sum_{k=1}^K \mathcal H^2(\Omega_n^k \cap \partial B_R(0))
  + 4 \pi r^2, 
\end{align}
for almost all $R_0 < R < \widetilde R$, and
\begin{align}
  \sum_{i \not= j} {1 \over |x_{i,n} - x_{j,n}|} \geq \sum_{k=1}^K
  \sum_{\stackrel{i,j \in I_k}{i \not= j}}  {1 \over |\tilde x_{i,n} -
  \tilde x_{j,n}|},
\end{align}
where
$\widetilde X_n^{R,L} = \cup_{i = 1}^N \{ \tilde x_{i,n}^{R,L}
\}$. Thus, by the isoperimetric inequality we have
\begin{align}
  \label{Erow}
  E_{\rho, \lambda,
  N}(\widetilde \Omega_n^{R,L}, \widetilde X_n^{R,L})  \leq E_{\rho,
  \lambda, 
  N}(\Omega_n, X_n) + 2  \sum_{k=1}^K \mathcal 
  H^2(\Omega_n^k \cap \partial B_R(0))+ {C 
  \over L},
\end{align}
for some $C > 0$ independent of $n$, $L$ and $R$. Furthermore, since
\begin{align}
  \sum_{k=1}^K |\Omega_n^k \cap (B_{\widetilde R}(0) \backslash
  B_{R_0}(0))| = 
  \sum_{k=1}^K \int_{R_0}^{\widetilde R} \mathcal 
  H^2(\Omega_n^k \cap \partial B_R(0)) \, dR \leq m,
\end{align}
for every $\widetilde R \geq 2 R_0$ it is possible to choose
$R \in (\widetilde R/2, \widetilde R) \subset (R_0, \widetilde R)$
such that
\begin{align}
  \label{Rn}
  \sum_{k=1}^K \mathcal H^2(\Omega_n^k \cap \partial B_R(0)) \leq {2 m
  \over \widetilde R}.  
\end{align}
Therefore, up to a subsequence (again, not relabeled) we can choose
$\widetilde R = \widetilde R_n \to \infty$ and $R = R_n \to \infty$
such that \eqref{Rn} holds, as well as $L = L_n \to \infty$
sufficiently fast, so that by \eqref{Erow} we have that
$(\widetilde \Omega_n, \widetilde X_n) := (\widetilde\Omega_n^{R_n,
  L_n}, \widetilde X_n^{R_n, L_n})$ is also a minimizing sequence.

We now modify the sets $\widetilde \Omega_n$ as follows to further
reduce the energy: For each $1 \leq k \leq K$, we replace the set
$(\widetilde \Omega_n - \tilde x_{i_1^k,n}) \cap B_{R_n}(0)$, where
$\widetilde X_n = \cup_{i=1}^N \{ \widetilde x_{i,n} \}$, with the
minimizer $\widetilde \Omega_n^k$ of the perimeter among all sets
supported in $B_{R_n}(0)$, 
containing $\cup_{\tilde x\in \widetilde X^k_n}B_\rho(\tilde x)$, and
satisfying $|\widetilde \Omega_n^k| = |\Omega_n^k \cap
B_{R_n}(0)|$. Existence of such a minimizer follows from the direct
method of calculus of variations (see, e.g., \cite[Section
12.5]{maggi}). We may also assume that each set
$\widetilde \Omega_n^k \cup B_{R_0}(0)$ is connected, since otherwise
all the mass of the disconnected pieces of
$\widetilde \Omega_n^k \backslash B_{R_0}(0)$ may be absorbed into the
ball $\Omega_n^0$ at the origin, producing a new set
$\widetilde \Omega_n^0$ without increasing the perimeter while
conserving the total mass. We denote by $\overline \Omega_n$ the set
obtained by replacing $\Omega_n^k$ with $\widetilde \Omega_n^k$ in the
definition of $\widetilde \Omega_n$. By construction we have
$(\overline \Omega_n, \widetilde X_n) \in \mathcal A_{m,N,\rho}$ and
$E_{\rho,\lambda,N}(\overline \Omega_n, \widetilde X_n) \leq
E_{\rho,\lambda,N}(\widetilde \Omega_n, \widetilde X_n)$, so
$(\overline \Omega_n, \widetilde X_n)$ is again a minimizing sequence.

Notice that if we compare the perimeter of $\widetilde \Omega_n^k$
with the one of $(\widetilde \Omega_n^k\cap B_r(0))\cup B$, where
$B \subset B_{R_n}^c(0)$ is a ball of volume
$v(r):=|\widetilde \Omega_n^k\setminus B_r(0)|$, after some simple
calculations we get that
\begin{align}
  P(\widetilde \Omega_n^k\setminus B_r(0))\le \bar c v ^\frac 23 (r)
  - 2{dv(r) \over dr} 
\end{align}
for a.e. $r\in (R_0,R_n)$, where $\bar c=(36\pi)^\frac 13$. It follows
that if $v(R_0 + 1) > 0$, then for all large enough $n$ we have
\begin{align}
  \int_{R_0}^{R_0+1}\frac{P(\widetilde \Omega_n^k\setminus
  B_r(0))}{v ^\frac 23 (r)}\,dr\le \bar c + 6 v^{\frac13}(R_0)\le \bar
  c + 6 m^{\frac13}. 
\end{align}
In particular, there exists $R_0'\in (R_0,R_0+1)$, depending on $n$
and $k$, such that
\begin{align}
  \frac{P(\widetilde \Omega_n^k\setminus B_{R_0'}(0))}{|\widetilde
  \Omega_n^k\setminus B_{R_0'}(0)|^\frac 23}\le \bar c + 6 m^{\frac13}. 
\end{align}
Then, by Lemma \ref{density}
applied with  $\Omega=\widetilde \Omega_n^k$ and $R=R_0'$,
the minimizer $\widetilde \Omega_n^k$
satisfies a uniform density estimate of the form
\begin{align}
  \label{dens}
  |\widetilde \Omega_n^k \cap B_r(x)| \geq c r^3,
\end{align}  
for some universal $c > 0$ and for all
$x \in \widetilde \Omega_n^k \backslash \overline B_{R_0'}(0)$ and
$0 < r \leq r_0:= c(m) |\widetilde \Omega_n^k \backslash
B_{R_0'}(0)|^{1/3}$.  Moreover, we claim that
$\widetilde \Omega_n^k \subset B_{R_\infty}(0)$ for some
$R_\infty > 0$ independent of $n$. Indeed, if
$|\widetilde \Omega_n^k \backslash B_{R_0 + 1}(0)| = 0$, there is
nothing to prove. At the same time, in view of the connectedness of
$\widetilde \Omega_n^k \cup B_{R_0}(0)$, the claim follows easily
by applying the density estimate in \eqref{dens} with $r = r_0$ to a
sequence of
$x = x_l \in \widetilde \Omega_n^k \cap \left( \partial B_{R_{0}+1 +
    (3 l - 1) r_0}(0) \backslash B_{R_{0}+1 + (3 l - 2) r_0}(0)
\right)$, for $l \in \mathbb N$, and the fact that
$|\widetilde \Omega_n^k \backslash B_{R_0}(0)|$ is bounded by $m$.

We now send $n \to \infty$. By compactness in $BV(B_{R_\infty}(0))$,
upon extraction of a subsequence we have
$\widetilde \Omega_n^k \to \Omega_\infty^k$ in the $L^1$-topology for
all $1 \leq k \leq K$. Also, since by construction
$\widetilde \Omega_n^0$ are balls containing the excess mass or are
empty, we likewise have $\widetilde \Omega_n^0 \to \Omega_\infty^0$ in
$L^1(\R^3)$ and $P(\widetilde \Omega_n^0) \to
P(\Omega_\infty^0)$. Then, by the lower-semicontinuity of the
perimeter we have
$\liminf_{n \to \infty} P(\widetilde \Omega_n^k) \geq
P(\Omega_\infty^k)$ for all $1 \leq k \leq K$. Upon a further
extraction of a subsequence we may also assume that
$x_{i,n} - x_{i_1^k,n} \to x_{i,\infty}^k$ for all $i \in I_k$, and by
continuity of the Coulombic energy we have
\begin{align}
  \lim_{n \to \infty} \sum_{\stackrel{i,j \in I_k}{i \not= j}} {1
  \over |x_{i,n} - x_{j,n}|} = 
  \sum_{\stackrel{i,j \in I_k}{i \not= j}}  {1 \over |\tilde x_{i,\infty}^k -
  \tilde x_{j,\infty}^k|}.
\end{align}
Thus, letting $X_\infty^k := \cup_{i \in I_k} \{ x_i^k \}$ we have
\begin{align}
  \label{infEK}
  \inf_{(\Omega, X) \in \mathcal A_{m,N,\rho}} E_{\rho,\lambda,N} =
  \liminf_{n \to \infty} E_{\rho,\lambda,N} (\Omega_n, X_n) \geq
  \liminf_{n \to \infty} E_{\rho,\lambda,N} (\overline \Omega_n,
  \widetilde X_n)  \notag \\
  \geq P(\Omega_\infty^0) + \sum_{k=1}^K \left( P(\Omega_\infty^k) +
  {\lambda \over 2}
  \sum_{\stackrel{i,j \in I_k}{i \not= j}}  {1 \over |\tilde
  x_{i,\infty}^k - \tilde x_{j,\infty}^k|} \right) \\
  = P(\Omega_\infty^0) + \sum_{k=1}^K E_{\rho,\lambda,N_k}
  (\Omega_\infty^k, X_\infty^k). \notag
\end{align}
Moreover, $ (\Omega_\infty^k, X_\infty^k)$ minimize
$E_{\rho,\lambda,N_k}$ over $\mathcal A_{m_k, N_k, \rho}$, where
$m_k := |\Omega_\infty^k|$, and by construction $\Omega_\infty^0$
minimizes the perimeter among all sets with mass
$m_0 = m - \sum_{k=1}^K m_k$. Indeed, otherwise it would be possible
to construct a test configuration of the form of \eqref{OmRL} from
those in $\mathcal A_{m_k, N_k, \rho}$ such that \eqref{infEK} is
violated. Finally, using $(\Omega_\infty^k, X_\infty^k)$ to form a
test function of the form of \eqref{OmRL} and sending $L \to \infty$
yields equality in \eqref{infEK}.

Finally, the regularity of $\partial\Omega_j$ follows by standard
regularity theory for minimal surfaces with smooth obstacles (see for
instance \cite[Theorem 21.8]{maggi}). Regularity, in turn, implies
connectedness, as otherwise the energy of two pieces that both contain
charges can be decreased by moving them far apart, while any two
pieces such that at least one piece does not contain any charges (and
hence is a ball) can be made to touch without changing energy,
contradicting the regularity of minimizers. Finally, the same constant
mean curvature for all the components away from the set of charges
follows from the arguments of \cite[Section 17.3]{maggi}.
\end{proof}

\section{Case of many charges}
\label{sec:case-many-charges}

\subsection{Preliminaries}

From here on we are concerned with minimizing the energy given in
\eqref{eq:Eeps} among $(\Omega, X) \in \mathcal A_\eps$. Note that
since by Theorem \ref{exgen} generalized minimizers always exist
whenever $\mathcal A_\eps$ is non-empty, it is convenient to formulate
our energy estimates in terms of the energy of such
minimizers. Furthermore, as competitors we may consider finite
collections of pairs $(\Omega_i, X_i)$, where $\Omega_i \subset \R^3$
are open sets with sufficiently smooth boundaries and
$X_i \subset \R^3$ are finite discrete sets satisfying
\begin{align}
  \sum_i |\Omega_i| = \frac{4\pi}{3}, \qquad \sum_i |X_i| = N_\eps.
\end{align}
With some abuse of notation, we will denote $k$ copies of the
component $(\Omega_i, X_i)$ of a competitor as $(\Omega_i, X_i)^k$,
with the obvious convention that
$(\Omega_i, X_i)^0 = (\varnothing, \varnothing)$.  We also define the
Coulombic interaction energy $V_\varepsilon(X)$ as
\begin{equation}
V_\varepsilon(X) :=\gamma \eps^3 \sum_{i
  = 1}^{N_\eps-1} \sum_{j=i+1}^{N_\eps} {1 
  \over |x_i - x_j|}. 
\end{equation}
Lastly, we note that in the statements and proofs that follow we
sometimes utilize explicit constants in the estimates, which, however,
are not intended to be optimal.

As a starting point, we have the following basic upper bound on the
minimal energy, which is obtained by considering a non-interacting
configuration of one large ball and $N_\eps-1$ individual discrete
charges. In particular, it gives a universal upper bound on the
minimal energy for $N_\eps \leq {1 \over \eps^2}$. 

\begin{lemma}
\label{lem1}
If $\{ (\Omega_1,X_1), \ldots, (\Omega_k,X_k)\}$ is a generalized
minimizer then
\begin{equation}
    \sum_{i=1}^{k} E_{\varepsilon}(\Omega_i,X_i)< 4 \pi (1 + 
    \varepsilon^2 N_\eps )   . 
  \end{equation}
\end{lemma}
\begin{proof}
  Testing the energy with the configuration of one charge in the
  center of a large ball and $N_\eps-1$ single charges in balls of
  radius $\eps$,
  namely, taking as a candidate \\
  $\{ (B_{r_1}(0), \{0\}), (B_\eps(0), \{0\})^{N_\eps - 1} \}$, where
  $r_1= \sqrt[3]{1-(N_\eps-1) \varepsilon^3} \leq 1$, we have
\begin{equation}
  \sum_{i=1}^k E_\eps(\Omega_i, X_i)  \leq 4 \pi r_1^2 + 4 \pi
  \varepsilon^2 (N_\eps-1),
  \end{equation}
  which yields the desired inequality.
\end{proof}

Note that as a convention from here on we order the elements of a
generalized minimizer
$\{ (\Omega_1,X_1),(\Omega_2,X_2), \ldots, (\Omega_k,X_k) \}$ in terms
of the decreasing magnitude of $|\Omega_i|$.
\begin{lemma}
\label{Lemma_2}
There exists a universal constant $C>0$ such that if
$\{ (\Omega_1,X_1), \ldots, (\Omega_k,X_k)\}$ is a generalized
minimizer then
$|\Omega_1| \geq \frac{4 \pi}{3} - C \varepsilon^3
N_\eps^{\frac{3}{2}} $.
\end{lemma}

\begin{proof}
  Without loss of generality let $k>1$. Let
  $r_i := \left (\frac{3}{4 \pi} |\Omega_i| \right)^{\frac{1}{3}} $,
  then by the isoperimetric inequality and positivity of $V_\eps$ we
  have
 \begin{equation}
    \sum_{i=1}^{k} E_{\varepsilon}(\Omega_i) \geq \sum_{i=1}^{k} 4 \pi
    r_i^2. 
 \end{equation}
 Eliminating $r_1$ via the volume constraint $\sum_{i=1}^k r_i^3=1$
 and using the fact that $t^\frac23 > t$ for $t \in (0,1)$, we obtain
 \begin{equation}
   \sum_{i=1}^{k} E_{\varepsilon}(\Omega_i) \geq \sum_{i=2}^{k} 4 \pi
   r_i^2 + 4 \pi \left (1-\sum_{i=2}^{k} r_i^3 \right )^{\frac{2}{3}}
   \geq 4 \pi + \sum_{i=2}^{k} 4 \pi r_i^2( 1 -r_i).  
    \label{Eq1}
 \end{equation}
 On the other hand, note that since $r_2 \leq \frac{1}{\sqrt[3]{2}}$,
 for all $1<i \leq k$ we have
 \begin{equation}
   r_i^2 (1 -r_i)  \geq \left (1- \frac{1}{\sqrt[3]{2}} \right) r_i^2. 
     \label{Eq2}
 \end{equation}
 Therefore, by Lemma \ref{lem1} we obtain
 \begin{equation}
       4 \pi (1+N_\eps \varepsilon^2) > \sum_{i=1}^{k}
       E_{\varepsilon}(\Omega_i) \geq 4 \pi + \sum_{i=2}^{k} 4 \pi
       r_i^2( 1 -r_i) \geq 4 \pi + C \sum_{i=2}^k r_i^2,
       \label{Eq3}
 \end{equation}
 for some universal $C > 0$.  Finally, by monotonicity of the
 $l^p$-norm in $p$, this implies
  \begin{equation}
    \sqrt[3]{\sum_{i=2}^{k} r_i^3} \leq  \sqrt{\sum_{i=2}^{k} r_i^2}
    \leq C'  \varepsilon \sqrt{N_\eps}, 
    \end{equation}
    for some $C' > 0$ universal, yielding the claim.
  \end{proof}

  Our next lemma provides further information about the volume of the
  small components of generalized minimizers. Notice that the
  conditions on $N_\eps$ throughout the rest of this section tacitly
  imply that $\eps$ is small.
  
\begin{lemma}
\label{Lem_New}
There exist universal constants $C, \delta>0$ such that for
$1 < N_\eps < \frac{\delta}{\varepsilon^2}$, if
$\{ (\Omega_1,X_1),(\Omega_2,X_2), \ldots, (\Omega_k,X_k) \} $ is a
generalized minimizer then
$| \Omega_i| \leq C |X_i|^{\frac{3}{2}} \varepsilon^3$ for all $i>1$.
\end{lemma}
\begin{proof}
  For $i>1$ create a minimizing candidate \\
  $\{ (c \Omega_1,c X_1), \ldots, (c \Omega_{i-1}, c X_{i-1}),(c
  \Omega_{i+1}, c X_{i+1}), \ldots, (c \Omega_{k}, c X_{k}), (
  B_{\varepsilon}(0), \{0\})^{|X_i|} \}$, which is obtained by
  deleting the $i$-th component, transferring its charges into $|X_i|$
  non-interacting balls of radius $\eps$ and rescaling the remaining
  components to adjust for the volume change. Here
\begin{align}
  c= \sqrt[3]{\frac{\frac{4 \pi}{3} - \frac{4 \pi}{3} |X_i| 
  \varepsilon^3}{\frac{4 \pi}{3}- |\Omega_i|}} =
  \sqrt[3]{1+\frac{|\Omega_i|- \frac{4 \pi}{3}|X_i|
  \varepsilon^3}{\frac{4 \pi}{3}-|\Omega_i|}}  
  \leq 
  \sqrt[3]{1 + \frac{3}{2 \pi} \left(|\Omega_i|-\frac{4 \pi}{3}|X_i|
  \varepsilon^3 \right)} \\ \leq 1 + \frac{1}{2 \pi}  
  \left(|\Omega_i|-\frac{4 \pi}{3}|X_i| \varepsilon^3 \right), 
\end{align}
where we used that $|\Omega_i| \leq \frac{2 \pi}{3}$ for all $i >
1$. Furthermore, from Lemma \ref{Lemma_2},
$|\Omega_i| \leq C \varepsilon^3 N_\eps^{\frac{3}{2}} \leq C
\delta^{\frac{3}{2}}$ for some universal constant $C>0$. Thus, we can
pick $\delta>0$ so that $|\Omega_i| \leq 1$, which gives us
\begin{align}
\label{equ_end}
  \sum_{j=1}^k E_{\varepsilon} (\Omega_j,X_j) \leq 4 \pi |X_i|
  \varepsilon^2 + \sum_{j \neq i} \left( c^2 P(\Omega_j) +
  V_{\varepsilon}(X_j) \right) \notag \\
  \leq  4 \pi |X_i| \varepsilon^2 +  \left (1 +2
  \left(|\Omega_i|-\frac{4 \pi}{3}|X_i| \varepsilon^3 \right) \right)
  \sum_{j \neq i} P(\Omega_j) + \sum_{j \neq i} 
  V_{\varepsilon}(X_j).   
\end{align}
From Lemma $ \ref{lem1}$, we can pick $C' > 0$ so that
$ \sum_{j=1}^k P(\Omega_j) \leq 4 \pi +4 \pi \varepsilon^2 N_\eps \leq 
4 \pi(1+ \delta) \leq C^\prime $. This gives us that
\begin{equation}
  \sum_{j=1}^k E_{\varepsilon} (\Omega_j,X_j) \leq 4 \pi |X_i|
  \varepsilon^2 -P(\Omega_i)+ 2 C^\prime  \left(|\Omega_i|-\frac{4
      \pi}{3}|X_i| \varepsilon^3 \right) +\sum_{j=1}^k 
  E_{\varepsilon} (\Omega_j,X_j). 
     \label{Lem4_eq1}
\end{equation}
Thus, with the help of the isoperimetric inequality for $\Omega_i$ we
have
\begin{equation}
\label{vol_i_est}
   \sqrt[3]{36 \pi}|\Omega_i|^{\frac{2}{3}} - 2 C^\prime  |\Omega_i|
   \leq \sqrt[3]{36 \pi}|\Omega_i|^{\frac{2}{3}} - 2 C^\prime
   \left(|\Omega_i|-\frac{4 \pi}{3}|X_i| \varepsilon^3 \right) \leq 
  4 \pi |X_i|\varepsilon^2. 
\end{equation}
Finally, since $|\Omega_i| \leq C \delta^{\frac{3}{2}}$, possibly
decreasing $\delta$ we can ensure that
$|\Omega_i|^{\frac{1}{3}}<\frac{1}{C^\prime}$, yielding the desired
inequality.
\end{proof}

Next we rule out the case where our generalized minimizer
$\{ (\Omega_1,X_1), \ldots, (\Omega_k,X_k)\}$ contains $X_i$'s that
are null, which means that each component of the generalized minimizer
has to contain at least one charge, provided that $N_\eps$ is not too
large and $\varepsilon$ is sufficiently small. In this case, if a
small component contains only one charge then it is a ball of radius
$\eps$. 

\begin{lemma}[]
\label{2}
There exist a universal constant $\delta>0$ such that for
$1 < N_\eps < \frac{\delta}{\varepsilon^2}$, if
$\{ (\Omega_1,X_1), \ldots, (\Omega_k,X_k)\}$ is a generalized
minimizer then $k \leq N_\eps$, and each $X_i$ for $1 \leq i \leq k$
is non-empty. Furthermore, if $|X_i|= 1$ for some $1<i\leq k$, then
$|\Omega_i|= \frac{4 \pi}{3} \varepsilon^3$.
\end{lemma}
\begin{proof}
  First note that if $X_i$ is empty, then Lemma \ref{Lem_New} implies
  that $|\Omega_i| = 0$, a contradiction. Thus, all that remains, is
  to show that if $|X_i|= 1$ for some $1<i\leq k$, then
  $|\Omega_i|= \frac{4 \pi}{3} \varepsilon^3$. To do this, note that
  when $|X_i|= 1$, rearranging \eqref{vol_i_est} provides
  \begin{equation}
    \label{K_1_inq}
       \sqrt[3]{36 \pi}|\Omega_i|^{\frac{2}{3}} -  4 \pi \varepsilon^2
       \leq  2 C^\prime  \left(|\Omega_i|-\frac{4 \pi}{3}
         \varepsilon^3 \right), 
   \end{equation}
   which implies that $|\Omega_i|=\frac{4 \pi}{3} \varepsilon^3 $
   whenever $\delta$ is chosen to ensure that
   $|\Omega_i|^{\frac{1}{3}}<\frac{1}{C^\prime}$. To see this, note
   that if
   $ \frac{4 \pi}{3} \varepsilon^3
   <|\Omega_i|<\frac{1}{(C^\prime)^3}$, then
\begin{equation}
  \int_{\frac{4 \pi}{3 }\eps^3}^{|\Omega_i|} C^\prime dt <
  \int_{\frac{4 \pi}{3 }\eps^3}^{|\Omega_i|} t^{-\frac{1}{3}} dt <
  \left(\frac{4 \pi}{3 }\right)^{\frac{1}{3}} \int_{\frac{4 \pi}{3
    }\eps^3}^{|\Omega_i|} t^{-\frac{1}{3}} dt ,  
\end{equation}
which implies
\begin{equation}
\label{inequality_con}
C^\prime  \left(|\Omega_i|-\frac{4 \pi}{3} \varepsilon^3 \right) <
\frac{1}{2} \left(  \sqrt[3]{36 \pi}|\Omega_i|^{\frac{2}{3}} -  4 \pi
  \varepsilon^2 \right). 
\end{equation}
Thus, \eqref{inequality_con} contradicts \eqref{K_1_inq}, and we
conclude that $|\Omega_i|=\frac{4 \pi}{3} \varepsilon^3 $.
\end{proof}

Lastly, we state a lower density estimate for generalized minimizers
which will be useful for both the $N_\eps =2$ and the
$N_\eps \gg 1$ cases.

\begin{lemma}
 \label{density_est}
 There exists a universal constant $C>0$ such that for any $M>0$, if
 $N_\eps <\frac{M}{\varepsilon^2}$,
 $\{ (\Omega_1,X_1), \ldots, (\Omega_k,X_k)\}$ is a generalized
 minimizer,
 $x_0 \in \overline \Omega_i \setminus \bigcup\limits_{x \in X_i}
 \overline{ B_{\varepsilon}(x)}$ for some $1\leq i \leq k$, and
 $r< \min \left (R, \min\limits_{x \in X_i} |x_0- x |-\varepsilon
 \right) $, where $R > 0$ depends only on $M$, then
 \begin{equation}
     | \Omega_i \cap B_{r}(x_0) |> C r^3.
 \end{equation}
 \end{lemma}
 \begin{proof}
   We may assume for convenience that $R < \frac{1}{2}$ and that
   $B_{R+\eps}(x_0) \cap X_i = \varnothing$. Consider
   $\{ (c \Omega_1,c X_1), \ldots, ( c (\Omega_i \setminus
   B_{r}(x_0)), c X_i), \ldots, (c \Omega_k, c X_k) \}$, where $c > 1$
   is defined as
\begin{equation}
  c: =\sqrt[3]{\frac{\frac{4 \pi}{3}} {\frac{4 \pi}{3}- |\Omega_i \cap
      B_{r}(x_0)|}} \leq 1+ |\Omega_i \cap B_{r}(x_0)|, 
\end{equation}
as a possible minimizing candidate. Then
\begin{multline}
  \label{density_estimate1}
  P(c( \Omega_i \setminus B_{r}(x_0)))+ \sum_{j \neq i} P(c\Omega_j)
  \\ \leq \left( 1+|\Omega_i \cap B_{r}(x_0)| \right)^2 \left (
    P(\Omega_i \setminus B_{r}(x_0)) + \sum_{j \neq i} P(\Omega_j)
  \right) \\ \leq ( 1+ 3|\Omega_i \cap B_{r}(x_0)|) \left ( P(\Omega_i
    \setminus B_{r}(x_0)) + \sum_{j \neq i} P(\Omega_j) \right).
\end{multline}
Furthermore, applying the isoperimetric inequality to the set
$ \Omega_i \cap B_{r}(x_0)$ we have that
\begin{equation}
  P(\Omega_i \setminus B_{r}(x_0)) \leq P(\Omega_i) +
  2 \mathcal{H}^2(\Omega_i \cap \partial B_{r}(x_0))- \sqrt[3]{36
    \pi} | \Omega_i \cap B_{r}(x_0)|^{\frac{2}{3}}. 
   \label{density_estimate2}
\end{equation}
Thus combining \eqref{density_estimate1} and \eqref{density_estimate2}
and using Lemma \ref{lem1}, we get
\begin{multline}
   \label{density_estimate3}
   P(c( \Omega_i \setminus B_{r}(x_0))) + \sum_{j \neq
     i} P(c\Omega_j) \\
   \leq ( 1+ 3 |\Omega_i \cap B_{r}(x_0)|) \left( 2
     \mathcal{H}^2(\Omega_i \cap \partial B_{r}(x_0))- \sqrt[3]{36
       \pi} | \Omega_i \cap B_{r}(x_0)|^{\frac{2}{3}} +\sum_{j=1}^{k}
     P(\Omega_j)\right) \\
   \leq \left( - \sqrt[3]{36 \pi} + 3|\Omega_i \cap
     B_{r}(x_0)|^{\frac{1}{3}} \sum_{j=1}^{k} P(\Omega_j)\right)
   |\Omega_i \cap B_{r}(x_0)|^{\frac{2}{3}} + C_1
   \mathcal{H}^2(\Omega_i \cap \partial B_{r}(x_0)) +\sum_{j=1}^{k}
   P(\Omega_j)\\
   \leq C_1 \mathcal{H}^2(\Omega_i \cap \partial B_{r}(x_0))-C_2
   |\Omega_i \cap B_{r}(x_0)|^{\frac{2}{3}}+\sum_{j=1}^{k} P(\Omega_j)
   ,
\end{multline}
for some universal constants $C_1$, $C_2>0$ whenever $r<R$ is small
enough depending only on $M$.

By \eqref{density_estimate3} we have
\begin{multline}
  \sum_{j=1}^{k} E_{\varepsilon}(\Omega_j,X_j) \leq E_{\varepsilon}
  \left(c( \Omega_i \setminus  B_{r}(x_0)) ,cX_i \right)+\sum_{j
    \neq i} E_{\varepsilon}(c \Omega_j,cX_j ) \\ \leq C_1
  \mathcal{H}^2(\Omega_i \cap \partial B_{r}(x_0))-C_2 |\Omega_i \cap
  B_{r}(x_0)|^{\frac{2}{3}}+\sum_{j=1}^{k}
  E_{\varepsilon}(\Omega_j,X_j) ,
\end{multline}
which implies
\begin{equation}
    C_1  \mathcal{H}^2(\Omega_i \cap \partial B_{r}(x_0)) \geq C_2
    |\Omega_i \cap B_{r}(x_0)|^{\frac{2}{3}}. 
\end{equation}
Finally, by letting $U(r) := |\Omega_i \cap B_{r}(x_0)| > 0$ and
applying Fubini's theorem with the co-area formula to obtain
$\frac{dU(r)}{dr}=\mathcal{H}^2(\Omega_i \cap \partial B_{r}(x_0)) $
for a.e. $r \in (0, R)$, we arrive at
\begin{equation}
    \frac{dU(r)}{dr} \geq C U^{\frac{2}{3}},
\end{equation}
for some universal constant $C>0$.  Integrating this inequality yields
the claim.
\end{proof}

\subsection{Localizing the minimizers}

In this subsection we perform a suitable localization of minimizers,
which leads to outer convergence of minimizers to a unit ball as
$\eps \to 0$. For $x_0 \in \R^3$ and $r > 0$, we define the spherical
cut of a set-charge pair
$(\Omega,X) \in \mathcal{A}_{m,N, \varepsilon}$ by the ball
$ B_{r}(x_0)$ to be the two set-charge pairs
$( \Omega_{x_0,r}^{\pm}, X_{x_0,r}^{\pm})$ defined as follows: If
$\mathcal{H}^2 \big(\partial B_{r}(x_0) \cap \big ( \cup_{x \in X}
B_{\varepsilon}(x) \big) \big)=0 $, then
 
 \begin{equation}
          \Omega_{x_0,r}^{+}= \Omega \cap B_{r}(x_0), 
 \end{equation}
 \begin{equation}
     X_{x_0,r}^{+}=\{x \in X \ : \ B_{\varepsilon}(x) \subset
     \Omega_{x_0,r}^{+} \}, 
 \end{equation}
 \begin{equation}
          \Omega_{x_0,r}^{-}=  \Omega \setminus \Omega_{x_0,r}^{+}, 
 \end{equation}
 and
 \begin{equation}
     X_{x_0,r}^{-}= X \setminus  X_{x_0,r}^{+}.
 \end{equation}
 If, on the contrary,
 $\mathcal{H}^2 \big(\partial B_{r}(x_0) \cap \big ( \cup_{x \in X}
 B_{\varepsilon}(x) \big) \big)>0 $, then we set
\begin{equation}
  X_{x_0,r}^{+}=\left\{x \in X \ :  \mathcal{H}^2 (\partial
    B_{\varepsilon}(x) \cap B_{r}(x_0)) > \mathcal{H}^2 (\partial
    B_{\varepsilon}(x) \cap B_{r}^{c}(x_0)) \right\}, 
 \end{equation} 
 \begin{equation}
          X_{x_0,r}^{-}=X \setminus  X_{x_0,r}^{+},
 \end{equation}
 \begin{equation}
      \Omega_{x_0,r}^{+}= \left ( \left(B_{r}(x_0) \cap \Omega \right)
        \cup \left (\cup_{x \in  X_{x_0,r}^{+}} B_{\varepsilon}(x)
        \right)  \right) \setminus \left( \cup_{x\in  X_{x_0,r}^{-}}
        B_{\varepsilon}(x) \right) 
 \end{equation}
 \begin{equation}
        \Omega_{x_0,r}^{-}= \Omega \setminus    \Omega_{x_0,r}^{+}. 
 \end{equation}
 For these spherical cuts we have the following result.

 \begin{lemma}
 \label{Cutting_perimeter}
 Let $\eps, m > 0$, $N \in \mathbb{N}$, and let $(\Omega, X)$ be a
 classical minimizer of $E_\eps$ over $\mathcal{A}_{m,N,\eps}$. Then
 if $x_0 \in \R^3$ and $r>0$, we have
\begin{equation}
  \label{PP4}
  P ( \Omega_{x_0,r}^{+}) +  P ( \Omega_{x_0,r}^{-})  \leq P(\Omega)+4
  \mathcal{H}^2 ( \Omega \cap \partial B_{r}(x_0)). 
\end{equation}
 \end{lemma}
 
 \begin{proof}
   By construction, we have
   \begin{align}
     \label{eq:PpOm}
     P ( \Omega_{x_0,r}^{+}) = \mathcal H^2(\partial \Omega \cap
     B_r(x_0)) + \mathcal H^2(\partial B_r(x_0) \cap \Omega) - \sum_{x
     \in X} \mathcal H^2(\partial B_r(x_0) \cap B_\eps(x)) \notag \\
     + \sum_{x \in X_{x_0,r}^+} \mathcal H^2(\partial B_\eps(x) \cap
     B_r^c(x_0)) +  \sum_{x \in X_{x_0,r}^-} \mathcal H^2(\partial
     B_\eps(x) \cap B_r(x_0)), 
   \end{align}
   and
   \begin{align}
     \label{eq:PmOm}
     P ( \Omega_{x_0,r}^{-}) = \mathcal H^2(\partial \Omega \cap
     B_r^c(x_0)) + \mathcal H^2(\partial B_r(x_0) \cap \Omega) - \sum_{x
     \in X} \mathcal H^2(\partial B_r(x_0) \cap B_\eps(x)) \notag \\
     + \sum_{x \in X_{x_0,r}^+} \mathcal H^2(\partial B_\eps(x) \cap
     B_r^c(x_0)) +  \sum_{x \in X_{x_0,r}^-} \mathcal H^2(\partial
     B_\eps(x) \cap B_r(x_0)).
   \end{align}
   Observe that for each $x \in X_{x_0,r}^+$ the set
   $\partial B_\eps(x) \cap B_r^c(x_0)$ is either empty or a spherical
   cap with the base radius $a_x \in (0, \eps]$ and height
   $h_x \in (0, \eps]$. Therefore, we have
   $\mathcal H^2( \partial B_\eps(x) \cap B_r^c(x_0) ) = 2 \pi \eps
   h_x$ and
   $\mathcal H^2(\partial B_r(x_0) \cap B_\eps(x)) \geq \pi
   a_x^2$. Noting that $(\eps - h_x)^2 + a_x^2 = \eps^2$, we then
   conclude that
   \begin{align}
     \frac{\mathcal{H}^2 (\partial B_{\varepsilon}(x) \cap
     B_{r}^{c}(x_0))}{\mathcal{H}^2 (\partial B_{r}(x_0) \cap
     B_{\varepsilon}(x))} \leq \frac{2 \varepsilon }{2 \varepsilon -h_x}
     \leq 2.
   \end{align}
   By a similar argument, for every $x \in X_{x_0,r}^-$ we also have
   \begin{align}
     \frac{\mathcal{H}^2 (\partial B_{\varepsilon}(x) \cap
     B_{r}(x_0))}{\mathcal{H}^2 (\partial B_{r}(x_0) \cap
     B_{\varepsilon}(x))} 
     \leq 2.
   \end{align}
   Thus, we obtain
   \begin{align}
     P ( \Omega_{x_0,r}^{+}) +  P ( \Omega_{x_0,r}^{-}) \leq P(\Omega)
     + 2 \mathcal H^2(\partial B_r(x_0) \cap \Omega) + 2 \sum_{x \in X}
     \mathcal H^2(\partial B_r(x_0) \cap B_\eps(x)) \notag \\
     \leq  P(\Omega)
     + 4 \mathcal H^2(\partial B_r(x_0) \cap \Omega),
   \end{align}
   which is the desired inequality.   
  \end{proof}

  Next we obtain an estimate for the $L^1$ convergence of classical
  minimizers to a ball as $\eps \to 0$.
  
 \begin{lemma}
\label{Ball}
There exist universal constants $C, C^\prime, \delta_0$ such that if
$\delta < \delta_0$, $1 < N_\eps < \frac{\delta}{\varepsilon^2}$, and
$(\Omega,X)$ is a classical minimizer, then
\begin{equation}
  |\Omega \cap B_{r^{\ast} }^{c}(x_0) | < C N_\eps^{\frac{3}{2}} \varepsilon^3,
\end{equation}
for some $1 \leq r^{\ast} \leq 1+C^\prime \delta^{\frac{1}{6}}$ and
some $x_0 \in \R^3$.
\end{lemma}
\begin{proof}
  First note that Lemma \ref{lem1} provides an upper bound on the
  isoperimetric deficit of $\Omega$, which is given by
\begin{equation}
  \frac{P(\Omega)-4 \pi}{4 \pi} \leq N_\eps \varepsilon^2. 
\end{equation}
In turn, by the quantitative isoperimetric inequality \cite{fusco08}
this gives us an upper bound on the Fraenkel asymmetry of $\Omega$,
which tells us that there exists $x_0 \in \R^3$ and a universal
constant $C_0 >0$ such that
\begin{equation}
\label{Sharp_iso}
| \Omega \Delta B_{1}(x_0) | \leq C_0 \sqrt{N_\eps} \varepsilon  < C_0
\sqrt{\delta}. 
\end{equation}

Arguing by contradiction, assume that 
\begin{align}
  \label{eq:contraball}
  |\Omega \cap B_{r}^{c}(x_0) | \geq C N_\eps^{\frac{3}{2}}
  \varepsilon^3 
\end{align}
for all $1 \leq r \leq 1+C^\prime \delta^{\frac{1}{6}}$ and
$C, C', \delta > 0$ arbitrary, provided that
$1 < N_\eps < {\delta \over \eps^2}$.  Picking
$C^\prime= 15 C^\frac{1}{3}$, for
$1 \leq r \leq 1+15 (C \sqrt{\delta})^{\frac{1}{3}}$ we then have that
\begin{align}
  \label{eq:Omx0rN32eps}
  |\Omega_{x_0,r}^{-}| \geq C N_\eps^{\frac{3}{2}} \varepsilon^3 - {4
  \pi \over 3} N_\eps \varepsilon^3> \frac{C}{2} N_\eps^{\frac{3}{2}}
  \varepsilon^3,  
\end{align}
provided that $C$ is sufficiently large universal. To construct a
minimizing candidate we first cut $(\Omega,X)$ into
$(\Omega_{x_0,r}^{+},X_{x_0,r}^{+})$ and
$(\Omega_{x_0,r}^{-},X_{x_0,r}^{-})$ and then split off the individual
charges from $(\Omega_{x_0,r}^{-},X_{x_0,r}^{-})$ and move any
remaining mass into $\Omega_{x_0,r}^{+}$. More precisely, this
minimizing candidate is
$\left \{(c \Omega_{x_0,r}^{+}, c X_{x_0,r}^{+}), (B_{\varepsilon}(0),
  \{0\})^k \right\}$, where
\begin{equation}
  c= \sqrt[3]{\frac{|\Omega_{x_0,r}^{+}|+|\Omega_{x_0,r}^{-}|- \frac{4
        \pi}{3} k \varepsilon^3}{|\Omega_{x_0,r}^{+}|}}, 
\end{equation}
and $k = |X_{x_0,r}^{-}|$. Letting
$\delta < \delta_0< \frac{1}{C_0^2}$, from \eqref{Sharp_iso} we get
that
$|\Omega_{x_0,r}^{+}| \geq \frac{4 \pi}{3}- 1 - \frac{4 \pi}{3} N_\eps
\varepsilon^3 > 2$ for all $\eps$ sufficiently small universal, which
gives
\begin{equation}
  c \leq  \sqrt[3]{1+ \frac{|\Omega_{x_0,r}^{-}|}{2}} \leq 1+
  \frac{1}{6} |\Omega_{x_0,r}^{-}|.  
\end{equation}
Now, by cutting and re-scaling in this way we have that
\begin{multline}
  E_{\varepsilon}(c(\Omega_{x_0,r}^{+},X_{x_0,r}^{+})) +4 \pi k
  \varepsilon^2 \leq V_{\varepsilon}(X_{x_0,r}^{+})+c^2
  P(\Omega_{x_0,r}^{+}) + 4 \pi k \varepsilon^2 \\
  \leq V_{\varepsilon}(X)+P(\Omega_{x_0,r}^{+}) + |\Omega_{x_0,r}^{-}|
  P(\Omega_{x_0,r}^{+})+4 \pi k \varepsilon^2.
    \label{E_Omega+}
\end{multline}
Furthermore, from Lemma \ref{Cutting_perimeter} we have that  
\begin{equation}
  P ( \Omega_{x_0,r}^{+})   \leq P(\Omega)+4  \mathcal{H}^2 ( \Omega
  \cap \partial B_{r}(x_0))-P ( \Omega_{x_0,r}^{-}), 
\end{equation}
which together with the isoperimetric inequality applied to
$\Omega_{x_0,r}^{-}$ gives
\begin{equation}
  P ( \Omega_{x_0,r}^{+})   \leq P(\Omega)+4  \mathcal{H}^2 ( \Omega
  \cap \partial B_{r}(x_0)) -\sqrt[3]{36 \pi} |
  \Omega_{x_0,r}^{-}|^{\frac{2}{3}}. 
          \label{P_Omega+}
\end{equation}
Thus, combining \eqref{E_Omega+} and \eqref{P_Omega+} and using Lemma
\ref{lem1} we get that
\begin{multline}
  E_{\varepsilon} \left(c(\Omega_{x_0,r}^{+},X_{x_0,r}^{+}) \right) +4
  \pi k \varepsilon^2 \leq \\ E_{\varepsilon}(\Omega,X) +4
  \mathcal{H}^2 ( \Omega \cap \partial B_{r}(x_0)) -\sqrt[3]{36 \pi} |
  \Omega_{x_0,r}^{-}|^{\frac{2}{3}} +4 \pi N_\eps \varepsilon^2 +
  |\Omega_{x_0,r}^{-}| \left(P(\Omega)+4 \mathcal{H}^2 ( \Omega \cap
    \partial B_{r}(x_0)) \right) \\
  \leq E_{\varepsilon}(\Omega,X) +5 \mathcal{H}^2 ( \Omega \cap
  \partial B_{r}(x_0)) -3 | \Omega_{x_0,r}^{-}|^{\frac{2}{3}} +4 \pi
  N_\eps \varepsilon^2,
          \label{E_Omega}
\end{multline}
provided that $\delta_0$ and, hence, $| \Omega_{x_0,r}^{-}| $ is
sufficiently small universal (see \eqref{Sharp_iso}).

Finally, by \eqref{eq:Omx0rN32eps} we can pick $C$ large enough so
that
$ | \Omega_{x_0,r}^{-}|^{\frac{2}{3}} \geq 4 \pi N_\eps
\varepsilon^2$. Hence from \eqref{E_Omega} and the minimality of
$(\Omega, X)$ we obtain
\begin{equation}
  E_{\varepsilon}(\Omega,X) \leq
  E_{\varepsilon}(c(\Omega_{x_0,r}^{+},X_{x_0,r}^{+})) +4 \pi k
  \varepsilon^2 \leq E_{\varepsilon}(\Omega,X) +5 \mathcal{H}^2 (
  \Omega \cap \partial B_{r}(x_0)) -2 |
  \Omega_{x_0,r}^{-}|^{\frac{2}{3}}, 
\end{equation}
which with the help of \eqref{eq:contraball} gives us
\begin{align}
  \label{eq:H225}
  \mathcal{H}^2( \Omega \cap \partial B_{r}(x_0))
  \geq \frac{2}{5} |
  \Omega_{x_0,r}^{-}|^{\frac{2}{3}}  \geq \frac25 |\Omega \cap
  B_{r}^{c}(x_0)|^{\frac{2}{3}} - {8 \pi \over 15} N_\eps \eps^3 
  \geq \frac{1}{5} |\Omega \cap B_{r}^{c}(x_0)|^{\frac{2}{3}},
\end{align}
whenever $1 \leq r \leq 1+15 (C \sqrt{\delta})^{\frac{1}{3}}$ and $C$
large enough.  Now, letting $U(r) := |\Omega \cap B_{r}^{c}(x_0)|$ and
applying Fubini's theorem and the co-area formula, from \eqref{eq:H225}
we get that
\begin{equation}
 U^{\frac{2}{3}}(r) \leq  -5 \frac{dU(r)}{dr}
\end{equation}
for a.e. $1 \leq r \leq 1+15 (C \sqrt{\delta})^{\frac{1}{3}}$.
Integrating over this interval gives us
\begin{equation}
  0 \leq U^{\frac{1}{3}} \left( 1+15 (C \sqrt{\delta})^{\frac{1}{3}} \right) \leq
  U^{\frac{1}{3}}(1) - \left( C  \sqrt{\delta} \right)^{\frac13}. 
\end{equation}
But this is a contradiction for $C$ large enough, since from
\eqref{Sharp_iso} we know that $U(1) <C_0 \sqrt{\delta}$.
\end{proof}

Using the $L^1$ convergence from Lemma \ref{Ball} and the density
estimate from Lemma \ref{density_est}, we have that classical
minimizers converge to a ball from the outside.
\begin{lemma}
\label{Outter_Ball}
For $\delta >0$ there exists $\delta_0$ depending only on $\delta$ and
$\gamma$ such that if $1 < N_\eps < \frac{\delta_0}{\varepsilon^2}$,
and $(\Omega, X)$ is a classical minimizer then
$\Omega \subset B_{1+\delta}(\hat{x})$ for some $\hat{x} \in \R^3$.

\end{lemma}
\begin{proof}
  Without loss of generality, we may assume that $\delta$ is
  sufficiently small universal. From Lemma \ref{Ball}, the constant
  $\delta_0>0$ can be picked so that
  $|\Omega \cap B_{1+\frac{\delta}{2}}^c(\hat{x})| <C
  N_\eps^{\frac{3}{2}} \varepsilon^{3} < C \delta_0^{\frac{3}{2}}$,
  where $C>0$ is a universal constant and $\hat{x} \in \R^3$. Now let
  $L= \frac{1}{6} \left( \sup_{x \in \Omega } |x-\hat{x}|-1 -
    \frac{\delta}{2} \right)$, and, arguing by contradiction, assume
  that $L>\frac{\delta}{12} $.

  For $r > 0$ and $y \in \R^3$, define $k_{y}(r)$ to be the number of
  charges inside of $\Omega_{y,r}^+$. We claim that there exists
  $x_0 \in \partial \Omega$ such that
  $B_{L}(x_0) \cap B_{1+\frac{\delta}{2}}(\hat{x})= \varnothing$ and
  $k_{x_0}(L) \leq \frac{N_\eps}{3}$. Indeed, let
    $x_4^\star \in \partial \Omega$ satisfy
    $|x_4^\star- \hat{x}| = \sup_{x \in \Omega} |x-\hat{x} |=1 + 6L +
    \frac{\delta}{2}$ (from the definition of $L$). Now pick
    $x_0^\star \in \partial B_{1 + \frac{\delta}{2}}(\hat{x})$ to
    satisfy \begin{equation} |x_0^\star- x_4^\star| = \inf_{x \in B_{1
          + \frac{\delta}{2}}(\hat{x})} |x-x_4^\star |= 6 L,
    \end{equation} 
    and for $x \in \mathbb R^3$ define a family of parallel planes
    $P(x)$ passing through $x$ orthogonally to $x_4^\star-x_0^\star$.
    Finally, define the three points
    \begin{equation}
      x_i^\star := x_0^\star + \frac{x_4^\star-x_0^\star}{6} + (i-1) 
      \frac{x_4^\star-x_0^\star}{3},
    \end{equation}
    for $i=1,2,3$. Note that for $i,j=1,2,3$ with $i \neq j$ we have
    that
    \begin{equation}
      \label{L_dist}
      \mathrm{dist}(P(x_{i}^\star),P( x_j^\star)) \geq \left |
        \frac{x_4^\star-x_0^\star}{3}\right| = 2L, 
    \end{equation}
    and
    \begin{equation}
      \label{LL_dist}
      \mathrm{dist}(P(x_{i}^\star),B_{1 + \frac{\delta}{2}}(\hat{x})) \geq
      \mathrm{dist}(x_{i}^\star,B_{1 +  \frac{\delta}{2}} (\hat{x})) \geq
      \left |   \frac{x_4^\star-x_0^\star}{6}\right| = L.
    \end{equation}
    Since $\Omega$ is bounded and connected, there exist
    $\hat{x}_i \in P(x_i^\star) \cap \partial \Omega$ for
    $i=1,2,3$. Thus, from \eqref{L_dist} we have that
    $B_{L}(\hat{x}_i) $ are pairwise disjoint for $i=1,2,3$, which
    implies that $\Omega_{\hat{x}_i ,L}^{+}$ are pairwise disjoint. It
    follows that
    $k_{L}(\hat{x}_{i^\star}) \leq \frac{N_{\varepsilon}}{3}$ for some
    $i^\star=1,2,3$. Lastly, with $x_0=\hat{x}_{i^\star}$, and taking
    into account that
    $ B_{L}(x_0) \cap B_{1 + \frac{\delta}{2}}(\hat{x}) = \varnothing$
    by \eqref{LL_dist} , the claim is proved.

  Now, define $U(r) := |\Omega \cap B_{r}(x_0) |$. Then since $U(r)$
  is a continuous monotone increasing function and $k_{x_0}(r)$ is a
  lower semi-continuous piecewise-constant function with a finite
  number of jumps, we have that
  $S := \left \{ r \in [0, L]: U(r) \geq (4 \pi k_{x_0}(r)
    \varepsilon^2)^{\frac{3}{2}} \right \}=[a_1,b_1] \cup [a_2,b_2]
  \cup \ldots \cup [a_{q},b_{q}]$.  By cutting and rescaling we will
  show that the set $S$ is small whenever $\delta$ is small. Let
  $r \in S$. Create a minimizing candidate
  $ \left \{ \left(c \Omega_{x_0,r}^-, c X_{x_0,r}^- \right),
    (B_{\varepsilon}^{1}(0), \{0\})^{k_{x_0} (r)} \right \}$, where
\begin{equation}
  c= \sqrt[3]{ \frac{ \frac{4 \pi}{3}-\frac{4 \pi}{3} k_{x_0}(r)
      \varepsilon^3}{\frac{4 \pi}{3}- |\Omega_{x_0,r}^+ |}} \leq
  \sqrt[3]{ \frac{ \frac{4 \pi}{3}}{\frac{4 \pi}{3}- |\Omega_{x_0,r}^+
      |}} \leq
  1+ \frac{3}{2 \pi} |\Omega_{x_0,r}^+ |, 
    \label{c_b}
\end{equation}
where from Lemma \ref{Ball}, $\delta_0$ is chosen so that
$|\Omega_{x_0,r}^+ | <C \delta_0^\frac{3}{2}$ is small universal.
Then by the minimality of $(\Omega, X)$ we have
\begin{multline}
  P(\Omega \cap B_{r}(x_0)) + P(\Omega \cap B_{r}^c(x_0))- 2
  \mathcal{H}^2( \partial B_{r}(x_0) \cap \Omega)+ V_{\varepsilon}(X)
  \leq P(\Omega)+V_{\varepsilon}(X)= E_{\varepsilon}(\Omega,X)
  \\
  \leq c^2 P(\Omega_{x_0,r}^-) + 4 \pi k_{x_0}(r) \varepsilon^2
  +V_{\varepsilon}(X).
        \label{1}
\end{multline}
Furthermore, from the argument in the proof of Lemma
\ref{Cutting_perimeter} we have that
$P(\Omega_{x_0,r}^-) < P(\Omega \cap B_{r}^c(x_0)) +2 \mathcal{H}^2(
\Omega \cap \partial B_{r}(x_0))$, and from \eqref{1} and \eqref{c_b}
we get that
\begin{multline}
  P(\Omega \cap B_{r}(x_0)) + P(\Omega \cap B_{r}^c(x_0))- 2
  \mathcal{H}^2( \Omega \cap \partial B_{r}(x_0)) \\ \leq \left (1+
    \frac{3}{2 \pi} |\Omega_{x_0,r}^+ | \right)^2 \left(P(\Omega \cap
    B_{r}^c(x_0)) +2 \mathcal{H}^2( \Omega \cap \partial B_{r}(x_0))
  \right) + 4 \pi k_{x_0}(r) \varepsilon^2
  \\
  \leq \left(1+2 |\Omega_{x_0,r}^+ | \right) \left (P(\Omega \cap
    B_{r}^c(x_0)) +2 \mathcal{H}^2( \Omega \cap \partial B_{r}(x_0))
  \right) + 4 \pi k_{x_0}(r) \varepsilon^2.
\end{multline}
Since
$|\Omega_{x_0,r}^+ | \leq U(r)+ \frac{4 \pi}{3} k_{x_0}(r)
\varepsilon^3$, we get that
\begin{multline}
  P(\Omega \cap B_{r}(x_0)) + P(\Omega \cap B_{r}^c(x_0))- 2
  \mathcal{H}^2( \Omega \cap \partial B_{r}(x_0))
  \\
  \leq \left (1+2U(r)+ \frac{8 \pi}{3} k_{x_0}(r) \varepsilon^3
  \right) \left(P(\Omega \cap B_{r}^c(x_0)) +2 \mathcal{H}^2( \Omega
    \cap \partial B_{r}(x_0)) \right) + 4 \pi k_{x_0}(r) \varepsilon^2
  ,
\end{multline}
and using the assumption that $r \in S$ gives
\begin{multline}
  \label{eq:isopkeps0}
  P(\Omega \cap B_{r}(x_0)) + P(\Omega \cap B_{r}^c(x_0))- 2
  \mathcal{H}^2( \Omega \cap \partial B_{r}(x_0))
  \\
  \leq (1+3 U(r)) \left(P(\Omega \cap B_{r}^c(x_0)) +2 \mathcal{H}^2(
     \Omega \cap \partial B_{r}(x_0)) \right) + 4 \pi k_{x_0}(r)
  \varepsilon^2 \\ \leq P(\Omega \cap B_{r}^c(x_0)) +2 \mathcal{H}^2(
  \partial B_{r}(x_0) \cap \Omega) + 3 U(r) P(\Omega) + 9 U(r)
  \mathcal{H}^2(  \Omega \cap \partial B_{r}(x_0))+ 4 \pi k_{x_0}(r)
  \varepsilon^2
  \\
  \leq P(\Omega \cap B_{r}^c(x_0)) + 3 \mathcal{H}^2( \Omega \cap
  \partial B_{r}(x_0)) +15 \pi U(r)+ 4 \pi k_{x_0}(r) \varepsilon^2 ,
\end{multline}
after possibly decreasing the value of $\delta_0$.

We now apply the isoperimetric inequality to $\Omega \cap B_{r}(x_0)$
in \eqref{eq:isopkeps0} to obtain
\begin{equation}
  \sqrt[3]{36 \pi} U^{\frac{2}{3}}(r)- 15 \pi U(r) -4 \pi k_{x_0}(r)
  \varepsilon^2 \leq 5 \mathcal{H}^2(  \Omega \cap \partial B_{r}(x_0)). 
\end{equation}
Since $U(r) < C \delta_0^{\frac{3}{2}}$, we can pick $\delta_0$ to give   
\begin{equation}
  \frac{\sqrt[3]{36 \pi}}{2} U^{\frac{2}{3}} (r) -4 \pi k_{x_0}(r)
  \varepsilon^2 \leq 5 \mathcal{H}^2(  \Omega \cap \partial B_{r}(x_0)). 
\end{equation}
Finally, since $r \in S$ we have that
$ \frac{\sqrt[3]{36 \pi}}{4} U^{\frac{2}{3}}(r) \geq
U^{\frac{2}{3}}(r) \geq 4 \pi k_{x_0}(r) \varepsilon^2$. Thus
\begin{equation}
  U^{\frac{2}{3}}(r) \leq 5 \mathcal{H}^2( \Omega \cap \partial
  B_{r}(x_0)) \qquad \forall r \in S. 
\end{equation}
Noting that
$d U(r) / dr = \mathcal{H}^2( \Omega \cap \partial B_{r}(x_0))$ for
a.e. $r$ and integrating this expression for $r \in [a_i,b_i]$ then
gives us that
\begin{equation}
  U^{\frac{1}{3}}(b_i)- U^{\frac{1}{3}}(a_i) \geq \frac{1}{15}(b_i-a_i).
\end{equation}
However, since by monotonicity of $U(r)$ we have $U(a_i) \geq
U(b_{i-1})$, it holds that
\begin{equation}
    U^{\frac{1}{3}}(b_q) \geq \frac{1}{15} \sum_{i=1}^{q} (b_i-a_i).
\end{equation}
At the same time, since we also have
$U^{\frac{1}{3}}(b_q) < C \sqrt{\delta_0}$ for some $C > 0$ universal,
we obtain that
\begin{equation}
  \sum_{i=1}^{q} (b_i-a_i) \leq 15 C \sqrt{\delta_0} \leq
  \frac{\delta}{24}< \frac{L}{2}, 
     \label{S_bbd}
\end{equation}
whenever $\delta_0 \leq \left( \frac{\delta}{360 C} \right)^2$, where
$C>0$ is a universal constant. In particular, the set $S$ has a small
measure controlled by $\delta$, as claimed. 

Now let $r<L$ and $r \in S^c$. Then from Lemma \ref{Cutting_perimeter}
and a comparison of the energy of $(\Omega, X)$ with that of
$\{ (\Omega_{x_0, r}^+, X_{x_0, r}^+), (\Omega_{x_0, r}^-, X_{x_0,
  r}^-)\}$ we get
\begin{multline}
  P ( \Omega_{x_0,r}^{+}) + P ( \Omega_{x_0,r}^{-}) -4 \mathcal{H}^2 (
  \Omega \cap \partial B_{r}(x_0))+V_{\varepsilon}(X) \leq
  P(\Omega)+V_{\varepsilon}(X)=E_{\varepsilon}(\Omega,X)  \\
  \leq P ( \Omega_{x_0,r}^{+}) + P ( \Omega_{x_0,r}^{-})
  +V_{\varepsilon}(X)- \frac{\gamma \varepsilon^3}{ (3+ 12 L)}
  k_{x_0}(r) (N_\eps-k_{x_0}(r)),
           \label{V_cut}
\end{multline}
since from the definition of $L$ the diameter of $\Omega$ is less or
equal than $2+ \delta+ 12 L < 3+12L $. Thus, \eqref{V_cut} implies
that
\begin{equation}
  \frac{\gamma \varepsilon^3}{(3+ 12 L)} k_{x_0}(r)
  (N_\eps-k_{x_0}(r)) \leq 4  \mathcal{H}^2 ( \Omega \cap \partial
  B_{r}(x_0)). 
\end{equation}
However, we chose $x_0$ and $L$ so that
$k_{x_0}(L) \leq \frac{N_\eps}{3}$, which gives us that
\begin{equation}
  \frac{ \gamma \varepsilon^3 N_\eps}{1+L} k_{x_0}(r)  \leq C^\prime
  \mathcal{H}^2 ( \Omega \cap \partial B_{r}(x_0)). 
\end{equation}
for some new universal constant $C^\prime>0$ that will change from
line to line in the remainder of the proof.  Furthermore, since
$r \in S^c$,
$U(r) < \left(4 \pi k_{x_0}(r) \varepsilon^2 \right)^{\frac{3}{2}}$,
which implies that
$ \frac{U^{\frac{2}{3}}(r)}{4 \pi \varepsilon^2} \leq
k_{x_0}(r)$. Thus
\begin{equation}
  \frac{\gamma \varepsilon N_\eps U^{\frac{2}{3}}(r)}{1+L  } \leq
  C^\prime \mathcal{H}^2 ( \Omega \cap \partial B_{r}(x_0)). 
\end{equation}
Integrating this expression from $[b_i,a_{i+1}]$, with $a_{q+1}:=L$,
gives us that
\begin{equation}
  C^\prime \left( U^{\frac{1}{3}}(a_{i+1})-U^{\frac{1}{3}}(b_{i})
  \right) \geq  \frac{\gamma \varepsilon N_\eps}{1+L }(a_{i+1}-b_i), 
\end{equation}
which again by monotonicity of $U(r)$ implies that
\begin{equation}
  \label{eq:monolowerU}
  U^{\frac{1}{3}}(L) \geq  \frac{\gamma \varepsilon N_\eps}{C^\prime
    (1+L)  } \sum_{i=1}^{q} (a_{i+1}-b_i).
\end{equation}
From \eqref{S_bbd}, we have that
$\sum_{i=1}^{q} (a_{i+1}-b_i) \geq \frac{L}{2}$, which gives that
\begin{equation}
  U^{\frac{1}{3}}(L) \geq \frac{\gamma \varepsilon N_\eps L}{C^\prime
    (1+L)  }. 
\end{equation}
However, from Lemma \ref{Ball} we have that
$U(L) <C N_\eps^{\frac{3}{2}} \varepsilon^{3}$, which implies that
\begin{equation}
  N_\eps \leq C' \left  (\frac{1+L}{\gamma L} \right)^2.
\end{equation}
Since $L> \frac{\delta}{12}$, this gives that
$N_\eps \leq \frac{C'}{\delta^2 \gamma^2}$.  Then by Lemma \ref{Ball}
we have
\begin{equation}
  |\Omega \cap B_{1+\frac{\delta}{2}}^c(\hat{x})| \leq  \frac{C' 
    \varepsilon^{3}}{\delta^3 \gamma^3} 
\end{equation}

Finally, to arrive at a contradiction observe that the set
$\Omega \cap B_r^c(\hat x)$ must have a connected component whose
diameter exceeds $\frac12 \delta$. A slicing argument at scale $\eps$
then gives
\begin{equation}
  |\Omega \cap B_{1+\frac{\delta}{2}}^c(\hat{x})| \geq C
  \varepsilon^2 \delta ,
\end{equation}
for some universal constant $C > 0$. Indeed, if a slice contains a
charge then it trivially contains a volume of at least of order
$\eps^3$. If, however, the slice does not contain any charges, then it
still contains at least that much volume by Lemma
\ref{density_est}. Together, the above two inequalities give a
contradiction for $ \varepsilon < C \delta^4 \gamma^3$, with $C > 0$
universal.
\end{proof}

\subsection{Existence results}
\label{sec:Existence-Results}

We now proceed to proving our existence and non-existence results for
$\eps \ll 1$. We begin by adapting the arguments in the proof of Lemma
\ref{Outter_Ball} to show that for not too large values of $N_\eps$
the diameter of all but the first component of a generalized minimizer
must be small.

\begin{lemma}
\label{Min_Component} 
There exist universal constants $C,\delta >0 $ such that for
$1 < N_\eps < \frac{\delta}{\varepsilon^2}$, if
$\{ (\Omega_1,X_1),... (\Omega_k,X_k) \}$ is a generalized minimizer
and $i>1$, then
$\mathrm{diam} (\Omega_i) \leq C \max \left( r_i, {\eps \over
    \gamma^3} \right)$, where
$r_i := \left( {3 \over 4 \pi} |\Omega_i| \right)^\frac13$.
\end{lemma}
\begin{proof}
  Let $i>1$. First note that Lemma \ref{Lem_New} provides an upper
  bound on $| \Omega_i|$, which implies smallness of $| \Omega_i|$ for
  $\delta$ sufficiently small universal. If
  $L_i := \mathrm{diam} (\Omega_i)$, then arguing by contradiction we
  may assume that
  $L_i > C \max \left( r_i, {\eps \over \gamma^3} \right)$, where
  $C>0$ is arbitrary.  Arguing exactly as in the proof of Lemma
  \ref{Outter_Ball} for the components $\Omega_i$, we may then obtain
  the estimate
  \begin{equation}
    L_i < \frac{C^\prime \varepsilon}{\gamma^3},
  \end{equation}
  for some $C' > 0$ universal: we can first find a point
  $x_0 \in \partial \Omega_i$ such that $(\Omega_i)_{x_0,r}^+$
  contains only a universally small fraction of the total number of
  charges $|X_i|$, then by cutting and rescaling we get an analog of
  the estimate in \eqref{S_bbd} in which the right-hand side is
  instead bounded by a universal multiple of $|\Omega_i|^\frac13$, and
  finally by cutting and separating the pieces we get an analog of the
  estimate in \eqref{eq:monolowerU} with $C' \gamma \eps |X_i| / L_i$
  multiplying the sum in the right-hand side instead. This provides a
  contradiction by choosing $C > C'$.
\end{proof}

Our next lemma gives a precise characterization of the generalized
minimizers when $\gamma$ is sufficiently large universal.

\begin{lemma}
\label{Min_class}
There exist universal constants $\delta, \gamma_0>0 $ such that for
$1 < N_\eps < \frac{\delta}{\varepsilon^2}$ and $\gamma> \gamma_0$, if
$\{ (\Omega_1,X_1),... (\Omega_k,X_k) \}$ is a generalized minimizer
then up to translations it has the form
$\{ (\Omega_1,X_1), (B_{\varepsilon}(0), \{0\})^{k-1} \}$.
\end{lemma}
\begin{proof}
  Without loss of generality, assume that $k>1$ and consider
  $(\Omega_i,X_i)$ for $i>1$. Arguing by contradiction, assume that
  $|X_i|>1$ (the case of $| X_i| \leq 1 $ is taken care by Lemma
  \ref{2}). First note that from the definition of generalized
  minimizers and Lemma \ref{Lem_New} we have that
  $c \eps^3 < | \Omega_i|< C |X_i|^{\frac{3}{2}} \varepsilon^3 $ for
  some universal $c, C>0$. Therefore, from Lemma \ref{Min_Component}
  there exists $\gamma_0>0$ such that if $\gamma> \gamma_0$ then
  $ \mathrm{diam} (\Omega_i) \leq C |\Omega_i|^\frac13$ for some
  universal $C > 0$. Thus
  \begin{equation}
    \mathrm{diam} (\Omega_i) \leq C \varepsilon \sqrt{|X_i|}.
  \end{equation}
   
  Now construct a minimizing candidate by cutting one charge at
  $x_0 \in X_i$ together with the ball $B_\eps(x_0)$ from $\Omega_i$
  and adding a new component $(B_\eps(0), \{0\})$, i.e., consider a
  competitor
  $\left \{ (\Omega_1,X_1), \ldots, (\Omega_i \setminus
    B_{\varepsilon}(x_0), X_i \setminus x_0), \ldots, (\Omega_k,X_k),
    (B_{\varepsilon}(0), \{0\}) \right\}$. Comparing the energies then
  yields
\begin{multline}
  8 \pi \varepsilon^2 - \frac{\varepsilon^2 \gamma (|X_i|-1)}{C
    \sqrt{|X_i|} } \geq 8 \pi \varepsilon^2 - \frac{\varepsilon^3
    \gamma (|X_i|-1)}{\mathrm{diam} (\Omega_i)}
  \\
  \geq E_\eps(\Omega_i \backslash B_\eps(x_0), X_i \backslash x_0) +
  E_\eps(B_\eps(0), \{0\}) - E_\eps(\Omega_i, X_i) \geq 0.
\end{multline}
Since $|X_i|\geq 2$, this gives $\gamma \leq 16 \pi C$, contradicting
the assumption on $\gamma$. Thus, for $i>1$ we have that $|X_i|=1$.
\end{proof}

\begin{proof}[Proof of Theorem \ref{t:exist}]
  First note that from Lemma \ref{Min_class}, without loss of
  generality we may assume that our generalized minimizer takes the
  form $\{ (\Omega_1,X_1), (B_{\varepsilon}(0), \{0\})^{k-1}
  \}$. Arguing by contradiction, assume that $k > 1$. We will
  construct a competitor by bringing one of the isolated charges into
  $(\Omega_1,X_1)$ and reducing the total energy. Since
  $(\Omega_1,X_1)$ is a classical minimizer of $E_\eps$ in the
  admissible class $\mathcal{A}_{|\Omega_1|,|X_1|,\eps}$, by
  considering a ball with $|X_1|$ approximately hexagonally packed
  charges, from \cite[Theorem C]{Wagner92} we obtain
\begin{equation}
  E_{\varepsilon}( \Omega_1,X_1) \leq \sqrt[3]{36 \pi}
  |\Omega_1|^{\frac{2}{3}} + \frac{\gamma \varepsilon^3 }{2
    \left(\sqrt[3]{\frac{3}{4 \pi} |\Omega_1|}- \varepsilon
    \right)}|X_1|^2. 
\end{equation}
Note that in view of Lemma \ref{Lemma_2} the distance between the
charges in this construction is at least of order
$\frac{1}{\sqrt{N_\eps}}> \sqrt{\gamma \varepsilon} \gg \varepsilon$,
for $N_\eps< \frac{1}{\gamma \varepsilon}$, $\gamma > 1$ and $\eps$
small enough universal.  Thus, letting
$X_1=\{x_1,x_2, \ldots, x_{|X_1|}\}$, by isoperimetric inequality we
get that
\begin{equation}
  V_{\varepsilon}(X_1) = \frac{\gamma \varepsilon^3}{2}
  \sum_{i=1}^{|X_1|} \sum_{j \neq 
    i}\frac{1}{|x_i-x_j|} \leq  \frac{\gamma 
    \varepsilon^3 }{2 \left(\sqrt[3]{\frac{3}{4 \pi} |\Omega_1|}-
      \varepsilon \right)}|X_1|^2 \leq \gamma \varepsilon^3 |X_1|^2, 
\end{equation}
in view of
\begin{align}
  \label{eq:Om1vol}
  \sqrt[3]{\frac{3}{4 \pi} |\Omega_1|}- \varepsilon \geq \frac{1}{2},
\end{align}
for all $\eps$ sufficiently small universal. This implies that there
exists $x_{i^\ast} \in X_1$ such that
\begin{equation}
\label{V_bdd}
     \sum_{j \neq i^{\ast}} \frac{1}{|x_{i^{\ast}}-x_j|} \leq 2 |X_1|.
   \end{equation}
   
   Consider
   $d :=\frac{1}{2} \min\limits_{j \neq i^{\ast}}
   |x_{i^{\ast}}-x_j|$. First note that arguing as in
   \eqref{eq:evapor} we have that
\begin{equation}
\label{Charge_sep}
\frac{ \gamma \varepsilon^3 }{2 d} \leq \gamma \varepsilon^3  \sum_{j
  \neq i^{\ast}} \frac{1}{|x_{i^{\ast}}-x_j|} \leq 8 \pi
\varepsilon^2, 
\end{equation}
which implies that 
\begin{equation}
\label{ch_dis}
    d \geq \frac{\gamma \varepsilon}{16 \pi } \geq 4 \varepsilon ,
\end{equation}
whenever $\gamma_0 \geq 64 \pi$. In addition, \eqref{eq:evapor}
implies $|x_i -x_j| > c \gamma \eps$ for $i \neq j$ and some universal
$c>0$, which proves the statement about charge separation.

Let $x_0 \in \Omega_1 \cap \partial B_{d}(x_{i^{\ast}})$, which exists
since by Theorem \ref{exgen} the set $\Omega_1$ is connected.  Now our
hope is to place a charge inside of $\Omega_1$ at $x_0$ and lower the
energy. Using \eqref{ch_dis}, we have that $x_0$ is sufficiently far
away from all the charges:
\begin{equation}
\label{x_0_dis}
    |x_0-x_j| \geq 4 \varepsilon \qquad \forall x_j \in X_1.
\end{equation}
Therefore, by Lemma \ref{density_est} we have that
\begin{equation}
\label{VV_max}
\frac{4 \pi}{3} \varepsilon^3 \geq  | \Omega_1 \cap
B_{\varepsilon}(x_0) |> C \varepsilon^3, 
\end{equation}
for some universal constant $C > 0$. Thus, for $\eps$ small enough
universal we can create a new minimizing candidate
$\{ c(\Omega_1 \cup B_{\varepsilon}(x_0), X_1 \cup \{x_0 \}),
(B_{\varepsilon}(0), \{0\})^{k-2} \}$, where
\begin{equation}
  c= \sqrt[3]{\frac{|\Omega_1|+ \frac{4 \pi}{3}
      \varepsilon^3}{|\Omega_1|-| \Omega_1 \cap B_{\varepsilon}(x_0)
      |+\frac{4 \pi}{3} \varepsilon^3}} \leq
  \sqrt[3]{\frac{1}{1- 8 \varepsilon^3}} < 1+3
  \varepsilon^3,
     \label{C_bbd}
\end{equation}
in which the first inequality is due to \eqref{eq:Om1vol}. Thus,
\eqref{VV_max} with the isoperimetric inequality and \eqref{C_bbd}
gives us the following upper bound on the perimeter of
$c(\Omega_1 \cup B_{\varepsilon}(x_0))$:
\begin{multline}
\label{P_bb}
P(c(\Omega_1 \cup B_{\varepsilon}(x_0))) = c^2 P(\Omega_1 \cup
B_{\varepsilon}(x_0)) \leq c^2 \Big(P(\Omega_1) +4 \pi
\varepsilon^2 - P(\Omega_1 \cap B_{\varepsilon}(x_0)) \Big) \\
\leq (1+10 \varepsilon^3) \left(P(\Omega_1) +4 \pi \varepsilon^2 -
  \sqrt[3]{36 \pi C^2 } \varepsilon^2 \right) \\
\leq P(\Omega_1) +4 \pi \varepsilon^2 - \sqrt[3]{36 \pi C^2 }
\varepsilon^2 + 40 \pi \varepsilon^3 (1+N_\eps \varepsilon^2) \\
\leq P(\Omega_1) +4 \pi \varepsilon^2 - C' \varepsilon^2,
\end{multline}
for $\eps$ small enough universal, where $C' > 0$ is a universal
constant and in the third inequality we used Lemma \ref{lem1}.

Lastly, we will obtain an upper bound on
$V_{\varepsilon}(c(X_1 \cup \{x_0 \}) ) \leq V_{\varepsilon}(X_1 \cup
\{x_0 \}).$ To do this, first note that from the definition of $d$ we
have
\begin{equation}
\label{Last1}
|x_0-x_j| \geq \frac12 |x_{i^{\ast}}-x_j|
\end{equation}
for all $j \neq i^{\ast}$, while
\begin{equation}
\label{Last2}
|x_0-x_{i^\ast}| =  \frac12 |x_{j^\ast} - x_{i^\ast}|
  \end{equation}
  for some $j^\ast \not= i^\ast$.  This gives us that
\begin{multline}
  V_{\varepsilon}(X_1 \cup \{x_0 \}) = V_{\varepsilon}(X_1 )+ \gamma
  \varepsilon^3 \sum_{x_j \in X_1}
  \frac{1}{|x_0-x_j|} \\
  \leq V_{\varepsilon}(X_1 )+ \frac{2 \gamma
    \varepsilon^3}{|x_{j^\ast}-x_{i^\ast}|}+ 2 \gamma \varepsilon^3
  \sum_{j \not = i^\ast}
  \frac{1}{|x_{i^\ast}-x_j|} \\
  \leq V_{\varepsilon}(X_1 ) + 4 \gamma \varepsilon^3 \sum_{j \not =
    i^\ast} \frac{1}{|x_{i^\ast}-x_j|}.
\end{multline}
Now using \eqref{V_bdd}, this gives that
\begin{equation}
\label{V_BB}
V_{\varepsilon}(X_1 \cup \{x_0 \}) \leq  V_{\varepsilon}(X_1 )+  8
\gamma \varepsilon^3 |X_1|. 
\end{equation}
Finally, combining the bound on the perimeter of
$c(\Omega_1 \cup B_{\eps}(x_0))$ given in \eqref{P_bb} with the bound
on $V_{\varepsilon}(c(X_1 \cup \{x_0 \})$ given in \eqref{V_BB}, we
obtain
\begin{multline}
  E_{\varepsilon}\left(c(\Omega_1 \cup B_{\varepsilon}(x_0)),
    c(X_1 \cup \{x_0 \}) \right) + 4 \pi (k-2) \varepsilon^2 \\
  \leq P(\Omega_1) +4 \pi(k-1) \varepsilon^2 - C' \varepsilon^2+
  V_{\varepsilon}(X_1 )+ 8 \gamma \varepsilon^3 |X_1|
  \\=E_{\varepsilon}(\Omega_1,X_1)+4 \pi (k-1) \varepsilon^2 -C'
  \varepsilon^2 + 8 \gamma \varepsilon^3 |X_1| < \sum_{j=1}^k
  E_{\varepsilon}(\Omega_j,X_j)
\end{multline}
whenever $|X_1|<N_\eps<\frac{C' }{ 8 \gamma \varepsilon}$, a
contradiction. Thus, generalized minimizers that are not classical
minimizers cannot exist for these values of $N_\eps$, with $\eps$
sufficiently small and $\gamma$ sufficiently large universal, and the
existence statement of the theorem holds.

To conclude the proof of Theorem \ref{t:exist}, let
$X = \cup_{i=1}^{N_\eps} \{ x_i \}$ be the minimizing set of the
positions of the charges. 
To prove that each $B_\eps(x_i)$ touches $\partial \Omega$,
suppose that, to the contrary, there exists $1 \leq i^* \leq N_\eps$
such that $B_\eps(x_{i^*}) \Subset \Omega$. Therefore, $x_{i^*}$ is a
local minimizer of $v_{i^*}(x) := \sum_{j \not= i^*} |x - x_j|^{-1}$,
since otherwise it would be possible to lower the energy by slightly
displacing the ball $B_\eps(x_{i^*}) \subset \Omega$ without touching
the other balls $B_\eps(x_j)$ with $j \not= i^*$. However, $v_{i^*}$
is a harmonic function in some neighborhood of $x_{i^*}$ and must,
therefore, be constant there, which is impossible as $v_{i^*}$ is a
real analytic function in $\R^3 \backslash \{X \backslash x_{i^*}\}$
that goes to infinity at $X \backslash x_{i^*}$.
\end{proof}

\begin{proof}[Proof of Theorem \ref{t:non}]
  Assume that $(\Omega,X)$ is a classical minimizer. First note that
  for a fixed $\gamma>0$, from Lemma \ref{Outter_Ball} we have that
  $\varepsilon_0,\delta>0$ can be chosen to make
  $\Omega \subset B_2(x)$ for some $x \in \R$. Thus, Lemma \ref{lem1},
  and \cite[Theorem 2]{Wagner90} imply that there exists a universal
  constant $C_0>0$ such that
\begin{equation}
  4 \pi +4 \pi N_\eps \varepsilon^2 \geq E_{\varepsilon}(\Omega,X) \geq 4
  \pi + \frac{\gamma \varepsilon^3 }{2} \left(\frac{N_\eps^2}{2}-C_0
    N_\eps^{\frac{3}{2}} \right) \geq 4 \pi + \frac18 \gamma N_\eps^2
  \varepsilon^3,  
\end{equation}
whenever $N_\eps > \frac{C}{\gamma \varepsilon}$ is large enough
universal. Thus
\begin{equation}
  N_\eps \leq \frac{32 \pi}{\varepsilon \gamma},
\end{equation}
a contradiction.


\end{proof}

\subsection{Convergence}

\begin {proof}[Proof of Theorem \ref{t:ball}]
  The fact that $\Omega_n \subset B_{1+\delta}(0)$ for any
  $\delta > 0$ and all $n \in \mathbb N$ large enough, after suitable
  translations, follows from Lemma \ref{Outter_Ball}. Now, for $N$
  distinct points $x_i \in \R^3$ define
  \begin{align}
    \label{eq:FN}
    F_N(\mu) :=
    \begin{cases}
      {2 \over N^2} \sum_{i=1}^{N-1} \sum_{j=i+1}^{N} {1 \over |x_i -
        x_j|} & \text{if} \ \mu = {1 \over N} \sum_{i=1}^N
      \delta_{x_i}, \\
      +\infty & \text{otherwise},
    \end{cases}
  \end{align}
  Then by \cite[Proposition 2.8]{serfaty-coulomb} we have that
  $\Gamma - \lim_{N \to \infty} F_N = F_\infty$ with respect to the
  weak convergence of probability measures in $\R^3$, where
  \begin{align}
    \label{eq:Finf}
    F_\infty(\mu) :=
    \int_{\R^3} \int_{\R^3} {d \mu(x) \, d \mu(y)
    \over |x - y|}.
  \end{align} 
  Therefore, for
  $\mu_n := {1 \over N_n} \sum_{x_i \in X_n} \delta_{x_i}$ we have
  that, upon suitable translations and extraction of subsequences,
  $\mu_n \to \mu_\infty$ in the sense of measures as $n \to \infty$,
  where $\mu_\infty$ is a probability measure supported on
  $\overline{B}_1(0)$, and
  $\liminf_{n \to \infty} F_{N_n}(\mu_n) \geq
  F_\infty(\mu_\infty)$. At the same time, testing the energy with
  $\Omega = B_1(0)$ and uniformly distributed $N_n$ points supported
  on $\partial B_{1-\eps_n}(0)$ (the existence of the latter follows
  from the construction of the recovery sequence for the above
  $\Gamma$-convergence, together with a scaling argument and the fact
  that $N_n \ll \eps_n^{-2}$, or an explicit construction in
  \cite{Wagner92}), with the help of the isoperimetric inequality we
  obtain
  \begin{align}
    4 \pi + \tfrac12 \gamma \eps_n^3 N_n^2 \left(
    F_{\infty}(\tfrac{1}{4 \pi} \mathcal 
    H^2\lfloor_{\partial \Omega}) + o_n(1) \right) 
    \geq
    E_{\eps_n}(\Omega_n, X_n) \notag \\
    \geq 4 \pi + \tfrac12 \gamma \eps_n^3
    N_n^2 \left( F_\infty(\mu_\infty) - o_n(1) \right). 
  \end{align}
  Thus
  $F_\infty( \tfrac{1}{4 \pi} \mathcal H^2\lfloor_{\partial \Omega})
  \geq F_\infty(\mu_\infty)$. However,
  $\mu = \tfrac{1}{4 \pi} \mathcal H^2\lfloor_{\partial \Omega}$ is
  the unique minimizer of $F_\infty$ among all probability measures
  supported on $\overline{B}_1(0)$. Hence
  $\mu_\infty = \tfrac{1}{4 \pi} \mathcal H^2\lfloor_{\partial
    \Omega}$ and the limit is in fact a full limit.
  \end{proof}

  \section{Case of two charges} \label{s:two}

  In this section, we give an explicit characterization of the
  minimizers in the simplest non-trivial case of $N = 2$ point
  charges. Note that when $N = 1$, the minimizer of $E_\eps$ is always
  a unit ball with the charge located anywhere inside.
  
\subsection{Existence results}
For $N=2$ and $X=\{x_1,x_2\}$, the energy in \eqref{eq:Eeps} becomes
simply 
\begin{equation}
\label{En_N2}
  E_\eps(\Omega, X) = P(\Omega) +\frac{ \gamma \eps^3}{|x_2-x_1|}.
\end{equation}
In this case, the energy of a generalized minimizer that is not
classical is known explicitly and satisfies the estimate below. 
\begin{lemma}
\label{generalized_en}
There exists a universal constant $\varepsilon_0>0$ such that if
$\varepsilon<\varepsilon_0$ and a classical minimizer of the energy in
\eqref{En_N2} does not exist, then the generalized minimizer has the
form $\{(\Omega_1,\{x_1\}),(\Omega_2,\{x_2\}) \}$ with
$x_1, x_2 \in \R^3$, and
\begin{equation}
  4 \pi (1+ \varepsilon^2 - \varepsilon^3)   <
  E_{\varepsilon}(\Omega_1, \{x_1\})+E_{\varepsilon}(\Omega_2, \{x_2\}
  )  < 4 \pi (1+ \varepsilon^2). 
\end{equation}
\end{lemma}
\begin{proof}
  From Lemma \ref{2}, we have that all the components of a generalized
  minimizer have at least one charge. Hence in the absence of
  classical minimizers the generalized minimizer consists of precisely
  two components, each of which has exactly one charge. Thus each
  component of the generalized minimizer is a ball, and again by Lemma
  \ref{2} we have $|\Omega_2|= \frac{4\pi}{3} \varepsilon^3$. Then
  \begin{align}
    E_{\varepsilon}(\Omega_1, \{x_1\})+E_{\varepsilon}(\Omega_2, \{x_2\} )
  = 4 \pi \left( \varepsilon^2+(1- \varepsilon^3)^{\frac{2}{3}}
  \right),
  \end{align}
  and the statement follows.
\end{proof}
Using the density estimates from Lemma \ref{density_est}, we have that
if $(\Omega,X)$ is a minimizer to \eqref{En_N2} then it is contained
inside a ball of radius close to one.

\begin{lemma}\label{N2bbd}
There exist universal constants $\varepsilon_0, C >0$ such that if
$\varepsilon < \varepsilon_0$ and $(\Omega,X)$ is a minimizer to
\eqref{En_N2} then
$\Omega \subset B_{1+ C \sqrt[3]{\varepsilon}}(x_0)$ for some
$x_0 \in \mathbb{R}^3$.
\end{lemma}
\begin{proof}
  Let $(\Omega,X)$ be a minimizer. Then from Lemma
  \ref{generalized_en} we have that
\begin{equation}
    P(\Omega) < E_{\varepsilon} (\Omega,X) < 4 \pi (1+\varepsilon^2).
\end{equation}
By quantitative isoperimetric inequality \cite{fusco08}, this gives us
a bound on the Fraenkel asymmetry of $\Omega$, namely, that there
exists $x_0 \in \mathbb{R}^3$ and a universal constant $C_0 > 0$ such
that
\begin{equation}
\label{volume_Oball}
     | \Omega \Delta B_{1}(x_0) | < C_0 \varepsilon.
\end{equation}
Now assume that the claim of the lemma is false, i.e., that $\Omega$
is not contained in $B_{1+ C \sqrt[3]{\varepsilon}}(x_0)$ for
arbitrary $C > 0$ and $\eps$ small enough. Then Lemma
\ref{density_est} tells us that there exist universal constants
$C_1 > 0$ and $\varepsilon_0 > 0$ such that when
$\varepsilon < \varepsilon_0$ we have
\begin{equation}
     | \Omega \Delta B_{1}(x_0) | > C_1 C^3 \varepsilon,
\end{equation}
which contradicts \eqref{volume_Oball} for a suitable choice of $C$.
\end{proof}

From the above lemmas, we have the following existence/non-existence
result for classical minimizers.
\begin{lemma}
\label{Exist_MinN2}
There exist universal constants $\varepsilon_0, C>0$ such that if
$\varepsilon < \varepsilon_0$ and
$\gamma < \frac{8 \pi}{\varepsilon}-C$ then there exists a minimizer
$(\Omega,X)$ to \eqref{En_N2} among all $(\Omega, X)$ admissible, and
if
$\gamma > \frac{8 \pi}{\varepsilon}+
\frac{C}{\varepsilon^{\frac{2}{3}}}$ then there is no minimizer.
\end{lemma}
\begin{proof}
  If $\gamma < \frac{8 \pi}{\varepsilon}-C$ and $\eps \ll 1$, then
  consider an admissible test configuration in the form of a ball with
  two charges at the opposite extremes,
  $(\Omega, X) = \big (B_{1}(0), \{ -(1 - \varepsilon) e_1, (1-
  \varepsilon) e_1 \} \big)$, for which we have
\begin{equation}
    E_{\varepsilon}(\Omega, X) = 4 \pi + \frac{\gamma \varepsilon^3}{2
      -  2 \varepsilon} < 4 \pi + \frac{8 \pi \varepsilon^2- C
      \varepsilon^3}{2 -  2 \varepsilon} . 
    \label{E_bbd}
\end{equation}
Picking $C>16 \pi$, from \eqref{E_bbd} we then get that
\begin{equation}
  E_{\varepsilon}(\Omega, X) < 4 \pi + 4 \pi \varepsilon^2 - \frac{8
    \pi \varepsilon^3}{2- 2 \varepsilon} <  4 \pi + 4 \pi
  \varepsilon^2 -  4 \pi \varepsilon^3. 
          \label{E_bbd1}
\end{equation}
Thus from Theorem \ref{exgen} and Lemma \ref{generalized_en}, we have
that the energy in \eqref{En_N2} must have a classical minimizer.

To prove non-existence, note that when
$\gamma > \frac{8 \pi}{\varepsilon}+
\frac{C}{\varepsilon^{\frac{2}{3}}}$, from Lemma \ref{N2bbd} we have
that for $(\Omega,X)$ admissible and $\eps \ll 1$ there exists a
universal constant $C_0 > 0$ such that
\begin{equation}
  E_{\varepsilon}(\Omega,X) \geq 4 \pi + \frac{\gamma \varepsilon^3}{2
    + C_0 \sqrt[3]{\varepsilon}} > 4 \pi + \frac{8 \pi \varepsilon^2
    +C \varepsilon^{\frac{7}{3}}}{2 + C_0 \sqrt[3]{\varepsilon}}> 4
  \pi + 4 \pi \varepsilon^2, 
\end{equation}
whenever $C> 4 \pi C_0$ and $\eps$ is sufficiently small. However,
Lemma \ref{generalized_en} implies that $(\Omega,X)$ cannot be a
minimizer and the Lemma is proved.
\end{proof}

In the case $N=2$, we can use rotational symmetry of the problem about
the axis passing through the two charges to explicitly solve for
global minimizers. Without loss of generality, let $X = \{x_1, x_2\}$,
with $x_1$ and $x_2$ located on the $x$-axis. Furthermore, let
$ \mathcal T \Omega$ denote the Schwarz symmetrization of $\Omega$
with respect to the $x$-axis, i.e. let
\begin{equation}
     \mathcal T \Omega = \{ (x, y,z) \in \mathbb R^3 : (y,z) \in A_x^\ast  \},
\end{equation}
where $A_x= \{(y,z) \in \mathbb R^2 : (x,y,z) \in \Omega \}$ and
$A_x^\ast= B_{r}(0) \in \mathbb R^2$ such that
$ \mathcal{L}^2 (A_x^\ast) = \mathcal{L}^2( A_x)$ denotes the two
dimensional symmetric rearrangement of $A_x$.
\begin{lemma}
\label{rotational}
Let $N=2$, let $x_1, x_2 \in \R \times (0,0)$ and let
$(\Omega, \{x_1,x_2\}) \in \mathcal A_{m,N,\rho} $ be a minimizer of
$E_{\rho,\lambda,N}$.  Then $\Omega = \mathcal T \Omega$.
\end{lemma}

\begin{proof}
  Note that from Fubini's theorem we have that
  $|\mathcal T \Omega| =| \Omega|=m $. Furthermore,
  $|\mathcal T \Omega \cap B_{\rho}(x_1)| =|\mathcal T \Omega \cap
  B_{\rho}(x_2)| = | B_{\rho}(0) |$. To see this, note that
  $| \Omega \cap B_{\rho}(x_1)|=| B_{\rho}(x_1) |$ implies that
  $ \mathcal{L}^2 (A_x) \geq \mathcal{L}^2( \{(y,z) \in \mathbb R^2 :
  (x,y,z) \in B_{\rho}(x_1) \}) $ for almost every $x \in \mathbb R
  $. Since $\{(y,z) \in \mathbb R^2 : (x,y,z) \in B_{\rho}(x_1) \}$ is
  also a ball in $\mathbb R^2$, we have that
  $\{(y,z) \in \mathbb R^2 : (x,y,z) \in B_{\rho}(x_1) \} \subset
  A_x^{\ast}$ for almost every $x \in \mathbb R $. Thus, by Fubini’s
  theorem $|\mathcal T \Omega \cap B_{\rho}(x_1)| =| B_{\rho}(x_1) |$,
  which gives that $(\mathcal T \Omega, \{x_1,x_2\})$ is also
  admissible.

  By \cite[Theorem 1.1]{barchiesi13}, we have that
  $P(\mathcal T \Omega) \leq P(\Omega)$, so $\mathcal T \Omega$ is
  also a minimizer, and by minimality of the energy this inequality is
  in fact an equality. Therefore, by \cite[Theorem 1.2]{barchiesi13}
  the sets $\Omega$ and $\mathcal T \Omega$ are equal up to a
  translation in the $yz$-plane. The latter follows from the fact that
  as a minimizer the set $\mathcal T \Omega$ is open and connected,
  and away from $B_\rho(x_1)$ and $B_\rho(x_2)$ the set $\Omega$ is a
  local volume-constrained minimizer of the perimeter, implying that
  $\partial \Omega$ is analytic \cite{maggi} and, hence, that the
  non-degeneracy assumptions of \cite[Theorem 1.2]{barchiesi13} are
  satisfied.
  
  Finally, assume by contradiction that
  $\Omega = \mathcal T \Omega + v$ for some vector $v\ne 0$ contained
  in the $yz$-plane.  Since $\mathcal T \Omega$ contains the two balls
  $B_\rho(x_{1,2})$, it follows that $\Omega$ also contains the
  translated balls $B_\rho(x_{1,2})+ \lambda v$, for all
  $\lambda\in [-1,1]$. Therefore, each $\partial B_\rho(x_{1,2})$
  could touch $\partial (\mathcal T \Omega)$ only at a point lying on
  the $x$-axis. But that is also impossible, since in that case
  $\partial \Omega$ would be flat near those points, contradicting
  once again the analyticity of $\partial \Omega$. Thus,
  $B_\rho(x_{1,2})$ are both strictly contained in
  $\mathcal T \Omega$. However, the latter contradicts the minimizing
  property of $(\mathcal T \Omega, \{x_1, x_2\})$, since one could
  reduce the energy by moving $B_\rho(x_{1,2})$ slightly further apart
  while still keeping them in $\mathcal T \Omega$.
\end{proof}

Since, according to Lemma \ref{rotational}, every minimizer to
\eqref{En_N2} coincides with its Schwarz symmetrization around the
axis connecting the two charges, it can be defined with the help of a
profile function $\varphi : \R \to [0, \infty)$, defining $\Omega$, up
to a rotation, as
\begin{align}
  \label{eq:prof}
  \Omega = \left\{ (x, y, z) \in \R^3: 0 < \sqrt{y^2 + z^2} <
  \varphi(x) \right\}.
\end{align}
By Theorem \ref{exgen}, the function $\varphi$ is of class $C^{1,1}$
on the set $\{x \in \R : \varphi(x) > 0 \}$. Furthermore, the support
of $\varphi$ is a single bounded interval, and $\varphi$ is smooth
  and solves the constant mean curvature equation outside the set of
  charges \cite[Section 3.6]{delaunay41,oprea}:
  \begin{align}
    \label{eq:H}
    -{\varphi''(x) \over (1 + {\varphi'}^2(x))^{3/2}} + {1 \over \varphi(x) (1
    + {\varphi'}^2(x))^{1/2}} = 2 H \quad \text{if} \quad (x -
    x_{1,2})^2 + \varphi^2(x) > \eps^2. 
  \end{align}
  Here $H \in \R$ is the mean curvature of
  $\partial \Omega \setminus (\partial B_{\varepsilon}(x_1) \cup
  \partial B_{\varepsilon}(x_2))$ defined as the average of the
  principal curvatures, positive if $\Omega$ is convex.

\begin{lemma}
\label{unduloid_arc}
There exists a universal constant $\varepsilon_0 > 0$ such that if
$\varepsilon<\varepsilon_0$, $N=2$,
$(\Omega, \{x_1,x_2\}) \in \mathcal A_{\frac{4 \pi}{3},N, \varepsilon}
$, and $(\Omega, \{x_1,x_2\})$ is a minimizer to \eqref{En_N2}, then
its free surface
$\partial \Omega \setminus (\partial B_{\varepsilon}(x_1) \cup
\partial B_{\varepsilon}(x_2))$ is a single section of an unduloid.
\end{lemma}
\begin{proof}
  By Lemma \ref{rotational}, every minimizer $(\Omega, \{x_1,x_2\})$
  to \eqref{En_N2} is rotationally symmetric, and away from the
  obstacles $\partial B_\eps(x_{1,2})$ the surface $\partial \Omega$
  satisfies \eqref{eq:H} as a local minimizer of the perimeter
  \cite{maggi}. Since by the volume constraint we have
    $\varphi(x_0) > \eps$ for all $\eps$ small enough universal, where
    $x_0$ is the point of maximum of $\varphi$, there holds $H >
    0$. Furthermore, we have $H \leq M$ for all $\eps$ small enough
    universal and some universal constant $M > 0$. To see the latter,
    suppose to the contrary that $H > M$ for any $M > 0$ and some
    $\eps$ sufficiently small. Then for all
    $\varphi(x) \geq 1 / M > \eps$ in a neighborhood of $x_0$ we have
    by \eqref{eq:H}:
    \begin{align}
      \label{eq:HH}
      -{\varphi''(x) \over (1 + {\varphi'}^2(x))^{3/2}} \geq M.
    \end{align}
    In particular, $\varphi$ takes the value
    $\varphi(x_1) = 1/M > \eps$ at some $x_1 < x_0$ for $\eps$
    sufficiently small and is concave on $[x_1, x_0]$. We now multiply
    \eqref{eq:HH} by $\varphi'(x) \geq 0$ and integrate over
    $[x_1, x_0]$ to obtain
  \begin{align}
    \label{eq:HHH}
    1 - {1 \over \sqrt{1 + {\varphi'}^2(x_1)}} \geq M
    \varphi(x_0) - 1,
  \end{align}
  which yields $\varphi(x_0) \leq 2 / M$, again contradicting the
  volume constraint for large enough $M$.

  In view of $H > 0$, the profile function $\varphi$ must satisfy
  \cite[Section 3.6]{delaunay41,oprea}
\begin{equation}
   \varphi' = \pm \sqrt{\frac{  \varphi^2} {( H \varphi^2+C_0 )^2}-1}
   \label{Euler-Lagrange}
 \end{equation}
 for some $C_0 \in \R$ away from the set of charges. Furthermore, when
 the free surface
 $\partial \Omega \backslash (\partial B_\eps(x_1) \cup \partial
 B_\eps(x_2))$ touches the obstacle, say,
 $\partial B_\varepsilon(x_1)$ at height $h$ (distance from the
 $x$-axis), then from the $C^{1,1}$ regularity of the minimizers the
 tangency condition at the point of contact gives
\begin{equation}
    C_0= \frac{ \pm 1- H \varepsilon }{\varepsilon} h^2.
    \label{height} 
\end{equation}
Rewriting this equation by solving for the positive height of contact
$h$, gives that $h$ is unique for fixed $C_0$, $H$, and $\varepsilon$,
whenever $\eps$ small enough universal. This tells us that a segment
of $\varphi$ satisfying \eqref{Euler-Lagrange} must leave and connect
to the obstacles at the same height.

If $C_0 <0$, then \eqref{Euler-Lagrange} gives that the graph of
$\varphi$ is an arc of a nodary curve. However, this is impossible. To
see why, first notice that a nodary curve can only touch the upper
  hemisphere of $B_\eps(x_1)$. Hence the value of $\varphi$ at which
  the nodary has a vertical slope satisfies
  $\varphi^2 = -C_0 / H < h^2$. At the same time, since $H$ is
  uniformly bounded, from \eqref{height} with $C_0 < 0$ and all $\eps$
  sufficiently small universal we obtain
  \begin{align}
    \label{eq:hvarphi}
    h^2 = -{ \eps C_0 \over 1 + \eps H} < -{C_0 \over H},
  \end{align}
  a contradiction. 
Thus, $C_0 \geq 0$.

\par If $C_0=0$, we get that $\Omega$ is a ball of radius one, which
is the limit case of $C_0 >0$ and for which the charges would touch
the boundary of the ball from inside at the diametrically opposite
points. However, it is not difficult to see that this is impossible,
using a first variation type argument by displacing the charges
further away by distance $0 < \delta \ll 1$ and gaining $O(\delta)$ in
the Coulombic energy, while losing only $O(\delta^2)$ in the
perimeter. Hence $C_0 >0$, and we get that the graph of $\varphi$ is
an arc of an elliptic catenary creating a corresponding unduloid
section as its surface of revolution. Up to translations we will
characterize an unduloid by its minimum height $a$ and its maximum
height $c$. Thus, $\Omega$ is the graph of $\varphi$ that consists of
arcs of elliptic catenary curves.
\par
To show that $\varphi$ contains only one section of an elliptic
catenary arc, note that the maximum height $c_1$ of at least one
elliptic catenary arc contained in $\varphi$
satisfies \begin{equation} c_1=1 + o(1),
\end{equation}
where $o(1)$ is with respect to $\varepsilon \ll 1$. This follows
directly from our volume constraint and Lemma
\ref{N2bbd}. Furthermore, let $a_1 \leq \varepsilon$ denote the
minimum height of this elliptic catenary.  Thus, the mean curvature of
the unduloid formed by this elliptic catenary arc is given by
\begin{equation}
H_1= \frac{1}{a_1+c_1}.
\end{equation}
Since the mean curvature of the free surface is constant, this implies
that
\begin{equation}
\label{MC_1}
 H= 1+o(1).
\end{equation}

Assume that the graph of $\varphi$ contains more than one elliptic
catenary arc, then at least one arc must contact the same charge at
two distinct points. Furthermore, \eqref{height} implies that both
contact points happen at the same height $0<h< \varepsilon$. Now let
$a_2$ and $c_2$ denote the minimum and maximum of this elliptic
catenary, then for $\varepsilon$ sufficiently small \eqref{MC_1} gives
us that $c_2= 1+ o(1)$. However, for sufficiently small $\varepsilon$
this is impossible, since the elliptic catenary contacts the same
charge at two distinct points.
\end{proof}

Here we state the following parametrization for an elliptic catenary,
which is obtained from \cite{hadzhilazova07}.
\begin{lemma}
\label{parametrization}
Up to translations, one period of an elliptic catenary with minimum
height $a$ and maximum height $c$ has the following parametrization:
\begin{equation}
\label{x_parametrization}
x(t)= a F \left(\frac{t}{2}-\frac{\pi}{4} , \frac{c^2-a^2}{c^2}
\right) + c  E\left(\frac{t}{2}-\frac{\pi}{4} , \frac{c^2-a^2}{c^2}
\right) , 
\end{equation}
and 
\begin{equation}
\label{z_parametrization}
    z(t)= \sqrt{\frac{c^2-a^2}{2} \sin{(t}) + \frac{c^2+a^2}{2}},
\end{equation}
where $-\frac{\pi}{2} \leq t \leq \frac{3 \pi}{2}$, $F(u,k)$ is the
elliptic integral of the first kind, and $E(u,k)$ is the elliptic
integral of the second kind which are defined as,
\begin{equation}
    F(u,k) := \int_{0}^{u}  \frac{1}{\sqrt{1- k \sin^2(\theta)}} d \theta,
    \label{EllipF}
\end{equation}and
\begin{equation}
    \label{EllipE}
    E(u,k) :=\int_{0}^{u} \sqrt{1- k \sin^2(\theta)} d \theta.
\end{equation}
\end{lemma}
From the parametrization given in Lemma \ref{parametrization} and
Lemmas \ref{unduloid_arc} and \ref{N2bbd} we have the following
result.
\begin{lemma}
\label{One_ArcN2}
There exists a universal constant $\varepsilon_0>0$ such that if
$\varepsilon < \varepsilon_0$ and if $(\Omega,\{x_1, x_2\})$ is a
minimizer to \eqref{En_N2}, then the profile function of the free
surface
$\partial \Omega \setminus (\partial B_{\varepsilon}(x_1) \cup
\partial B_{\varepsilon}(x_2))$ is a graph of a single arc of an
elliptic catenary that has exactly one maximum.
\end{lemma}
\begin{proof}
  Let $\varphi$ be the profile function defined in
  \eqref{eq:prof}. Then by Lemma \ref{unduloid_arc} the graph of
  $\varphi$ must consist of a single arc of an elliptic
  catenary. Thus, we must show that $\varphi$ has exactly one
  maximum. To do this, define $M$ to be the number of maximum points
  of $\varphi$. First, note that $M$ is nonzero, since if $M=0$ then
  Lemma \ref{N2bbd} implies that we cannot satisfy our volume
  constraint.
  
  Now assume that $M>1$, and let $h$ be the height of contact between
  $\varphi$ and the charges. Using the parametrization given in Lemma
  \ref{parametrization}, we can find a lower bound on the distance
  between the two contact points, which gives us that
\begin{align}
   \mathrm{diam} (\Omega) \geq 2 \left ( a F\left(\frac{t_0}{2}-\frac{\pi}{4}
  , \frac{c^2-a^2}{c^2} \right) + c
  E\left(\frac{t_0}{2}-\frac{\pi}{4} , \frac{c^2-a^2}{c^2} \right)
  \right) \notag \\
  +2(M-1)  \left ( a F\left(\frac{\pi}{2} , \frac{c^2-a^2}{c^2}
  \right) + c  E\left(\frac{\pi}{2}, \frac{c^2-a^2}{c^2} \right)
  \right), 
        \label{D_OmgaN2}
\end{align}
where
\begin{equation}
    t_0= \pi- \arcsin \left(\frac{2h^2-(c^2+a^2)}{c^2-a^2} \right).
\end{equation}
However, as $\varepsilon \xrightarrow[]{} 0$ we have that
$a \xrightarrow[]{} 0$ and $c \xrightarrow[]{} 1$. Thus,
\eqref{D_OmgaN2} implies that
\begin{equation}
  \mathrm{diam} (\Omega) \geq 2 M- o(1),
\end{equation}
and we have that $\varepsilon_0$ can be chosen small enough universal
to provide a contradiction to Lemma \ref{N2bbd}.
\end{proof}
Lastly, we state some expansions for elliptic integrals.

\begin{lemma}
  Let $F(u,k)$ and $E(u,k)$ be the incomplete elliptic integrals of
  the first and second kind as defined in \eqref{EllipF} and
  \eqref{EllipE}, then
\begin{equation}
\label{F_comp}
F \left(\frac{\pi}{2},z  \right)= - \frac{1}{2} \log{(1-z)} + \BigO{1},
\end{equation}

\begin{equation}
\label{E_comp}
E \left(\frac{\pi}{2},z   \right)= 1 +\frac{z-1}{4} \log{(1-z)} + \BigO{(z-1) },
\end{equation}
 \begin{equation}
  \label{F_ex}
    F(\arcsin{(u)},z)= \arctanh{(u)}- \BigO{ \frac{z-1}{u-1}},
\end{equation}
\begin{equation}
\label{E_ex}
    E(\arcsin{(u)},z)=u- \frac{z-1}{2}\arctanh{(u)}+\BigO{z-1, \frac{(z-1)^2}{u-1}},
\end{equation}
\begin{equation}
\label{F_h_big}
 F(\arcsin{(u)},z)= - \frac{1}{2} \log{(1-z)} + \BigO{1, \frac{\sqrt{ 1-u}}{\sqrt{1-z}}},
\end{equation}
and 
\begin{equation}
\label{E_h_big}
 E(\arcsin{(u)},z)=1 + \frac{z-1}{4}  \log{(1-z)}  +
 \BigO{(z-1),\sqrt{(1-u) (1-zu^2)} }. 
\end{equation}
\end{lemma}

Here and below we are using the notation $\BigO{a, b}$, etc., to
denote the quantities that are bounded by universal multiples of
$\max( |a|, |b| )$ for $|a|, |b| \ll 1$.


\subsection{Classification of cases}

Let $(\Omega^\ast,X) \in \mathcal A_{\frac{4 \pi}{3},2,\varepsilon}$
be a minimizer to \eqref{En_N2}, then from Lemma
\ref{One_ArcN2}, we have that $\Omega^\ast$ falls into one of three cases. In case 1, we have that the free
surface of $\Omega^\ast$ has the profile function whose graph is an
arc of an elliptic catenary that does not attain its minimum value
$a$. In case 2, we have that this elliptic catenary arc attains its
minimum value $a$ at exactly one point. Lastly, in case 3, this
elliptic catenary arc obtains its minimum value $a$ at exactly two
points. These three alternatives are illustrated in Figure
\ref{N2_cases_plot}.  Thus, from here on we only need to consider test
configurations
$(\Omega,X) \in \mathcal A_{\frac{4 \pi}{3},2,\varepsilon}$, which
fall into one of the three cases defined above. Furthermore, without
loss of generality we define the maximum of the profile function to be
located at $(0,c)$, which is consistent with the parametrization given
in Lemma
\ref{parametrization} and fixes translations along the $x$-axis.
\par 

In case 1, the unduloid arc joining the two charges does not attain
its minimum and the minimizer is symmetric about the $z$-axis. In this
case, we have that our minimizer is of the form
$\left (\Omega^\ast, \{ (-\frac{L}{2},0,0),(\frac{L}{2},0,0) \}
\right)$, where $L$ is the distance between the charges. For a case
1 test configuration we have the following lemma, which implicitly
expresses the energy as a function of the contact height $h$ (see Figure
\ref{schematic_case1} for a schematic of a case 1 test
configuration).

\begin{figure}
    \centering
 
\includegraphics[width=15cm]{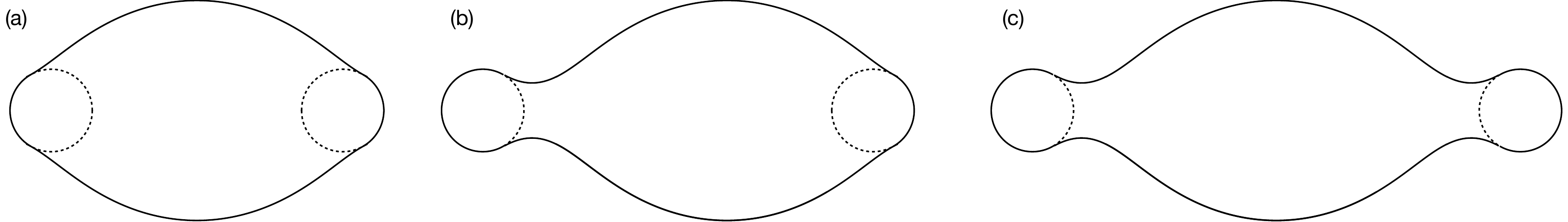}

\caption{Three possibilities of the energy minimizing candidates for
  $N = 2$: (a) case 1, (b) case 2, (c) case 3.}
 \label{N2_cases_plot}

\end{figure}

\begin{lemma}
\label{case1}
Let $(\Omega, X) \in \mathcal A_{\frac{4 \pi}{3},2,\varepsilon}$ be a
case 1 test configuration with contact height $h$. Then
\begin{align}
  E_{\varepsilon} \left(\Omega, X \right)=
  \hat{E_\varepsilon}(h) =4 \pi \left ( (a+c) c E \left
  (\frac{t_0}{2}- \frac{\pi}{4},   \frac{c^2-a^2}{c^2} \right) +
  \varepsilon \left(\varepsilon - \sqrt{\varepsilon^2-h^2} \right)
  \right) \notag \\
  +\frac{\gamma \varepsilon^3}{2 \left(a
  F\left(\frac{t_0}{2}-\frac{\pi}{4} , \frac{c^2-a^2}{c^2} \right) + c
  E\left(\frac{t_0}{2}-\frac{\pi}{4} , \frac{c^2-a^2}{c^2} \right)  -
  \sqrt{\varepsilon^2-h^2}\right)}, 
            \label{Energy_fun}
\end{align}
where
\begin{equation}
\label{a_equation1}
    a= \frac{c h^2- \varepsilon h^2}{c \varepsilon -h^2},
\end{equation}
\begin{equation}
  \label{t0_equation1}
     t_0= \pi- \arcsin \left (\frac{2h^2-(c^2+a^2)}{c^2-a^2} \right),
\end{equation}
and $c$ is given implicitly by 
\begin{align}
  2=  2 \varepsilon^3-\sqrt{\varepsilon^2-h^2}(2 \varepsilon^2+h^2)+
  \frac{h(c^2-a^2)}{2}\sqrt{1-\left(\frac{2 h^2-c^2-a^2}{c^2-a^2}
  \right)^2}- \notag \\ 
  a^2 c F \left( \frac{t_0}{2} - \frac{\pi}{4},\frac{c^2-a^2}{c^2}
  \right)+ (2c(a^2+c^2)+3ac^2) E \left( \frac{t_0}{2} -
  \frac{\pi}{4},\frac{c^2-a^2}{c^2} \right). 
    \label{V}
\end{align}
\end{lemma}

\begin{figure}
    \centering
     
\includegraphics[width=8cm]{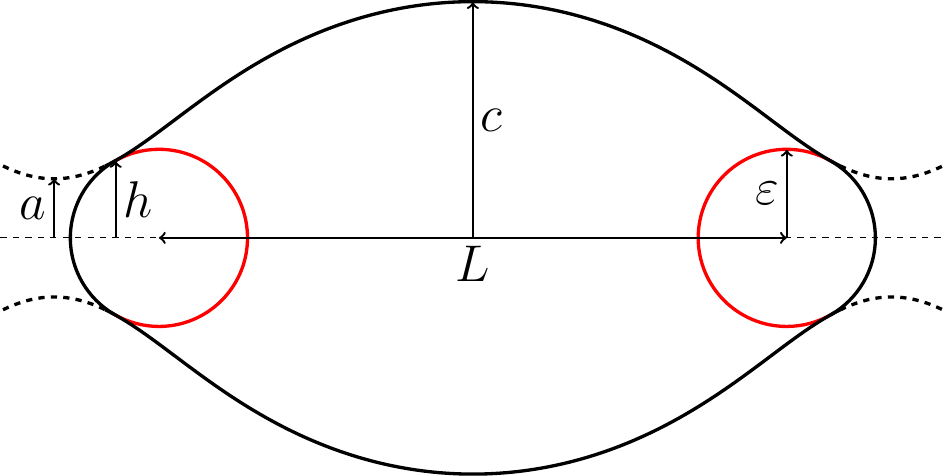}
\caption{The schematics of the cross-section of the case 1 candidate
  for a minimizer.}
\label{schematic_case1}
\end{figure}

\begin{proof}
  Let the unduloid section that joins the two charges have minimum
  height $a$ and maximum height $c$. Since the charges contact the
  unduloid section at height $h$, we can use the parameterization
  given in \eqref{z_parametrization} to find that the contact between
  the unduloid and the right charge happens when $t = t_0$, with $t_0$
  defined in \eqref{t0_equation1}.  Then from
  \eqref{x_parametrization} we obtain that
\begin{equation}
\label{L}
L = 2 \left ( a F\left(\frac{t_0}{2}-\frac{\pi}{4} ,
    \frac{c^2-a^2}{c^2} \right) + c
  E\left(\frac{t_0}{2}-\frac{\pi}{4} , \frac{c^2-a^2}{c^2} \right) -
  \sqrt{\varepsilon^2-h^2}\right), 
\end{equation}
where $L$ is the distance between the charges.  Now, since our unduloid
has mean curvature $H= \frac{1}{a+c} $, from \eqref{Euler-Lagrange} we
have that each monotone arc of our elliptic catenary is given by the
equation,
\begin{equation}
       \varphi' = \pm \sqrt{\frac{ (a+c)^2  \varphi^2} {(
           \varphi^2+ac )^2}-1}. 
\end{equation}
Thus, our tangency condition between the charges and the elliptic
catenary implies \eqref{a_equation1}. In addition, calculating the
volume of the unduloid section, which is given in
\cite{hadzhilazova07}, and accounting for the volume of the excess
charges gives \eqref{V}. Finally, \eqref{L} explains the interaction
energy given in \eqref{Energy_fun}, and the perimeter term is derived
directly from accounting for the surface area of the unduloid section
(given in \cite{hadzhilazova07}) and the surface area over the
charges.
\end{proof}

We now proceed to cases 2 and 3.

\begin{lemma}
\label{case2}
Let $(\Omega, X) \in \mathcal A_{\frac{4 \pi}{3},2,\varepsilon}$ be a
case 2 test configuration with contact height $h$. Then
\begin{align}
  E_{\varepsilon} \left(\Omega, X \right)= \hat{E_\varepsilon}(h) =4
  \pi c(a+c) E \left( \frac{\pi}{2},{\frac{c^2-a^2}{c^2}} \right ) + 
  4 \pi \varepsilon^2 \notag \\
  +\lambda \varepsilon^3  \left ( 2a F \left(
  \frac{\pi}{2},\frac{c^2-a^2}{c^2} \right ) + 2 c E \left
  (\frac{\pi}{2}, \frac{c^2-a^2}{c^2} \right)  \right)^{-1}  , 
            \label{Energy_fun_case2}
\end{align}
where $a$ is given by \eqref{a_equation1}
and $c$ is given implicitly by
\begin{align}
  2 = ( 2(a^2+c^2)c+3a c^2) E \left(\frac{\pi }{2},
  {\frac{c^2-a^2}{c^2}}  \right)  -a^2 c F \left(\frac{\pi
  }{2},{\frac{c^2-a^2}{c^2}} \right) + 2 \varepsilon^3 . 
    \label{V_case2}
\end{align}
\end{lemma}

\begin{proof}
  Note that the equation for $a$ comes from the tangency condition, as
  in Lemma \ref{case1}. Furthermore, \eqref{Energy_fun_case2} and
  \eqref{V_case2} follow directly from \cite{hadzhilazova07}.
\end{proof}

\begin{lemma}
\label{case3}
Let $(\Omega, X) \in \mathcal A_{\frac{4 \pi}{3},2,\varepsilon}$ be a
case 3 test configuration with contact height $h$. Then
\begin{equation}
    \begin{aligned}
      E_{\varepsilon} \left(\Omega,X \right)= \hat{E_\varepsilon}(h)
      =8 \pi c(a+c) E \left ( \frac{\pi}{2},{\frac{c^2-a^2}{c^2}}
      \right ) \\ - 4 \pi c(a+c) E \left ( \frac{t_0}{2}-
        \frac{\pi}{4}, \frac{c^2-a^2}{c^2} \right ) +
      4 \pi \varepsilon (\varepsilon +\sqrt{\varepsilon ^2-h^2}) \\
      +\lambda \varepsilon^3 \left ( 2 \bigg (\sqrt{\varepsilon^2-h^2}
        +2a F \bigg( \frac{\pi}{2},\frac{c^2-a^2}{c^2} \bigg ) + 2 c E
        \bigg (\frac{\pi}{2}, \frac{c^2-a^2}{c^2} \bigg) - \right. \\
      \left. a F \bigg(\frac{t_0}{2}- \frac{\pi}{4} ,
        {\frac{c^2-a^2}{c^2}} \bigg) -c E \bigg( \frac{t_0}{2}-
        \frac{\pi}{4}, {\frac{c^2-a^2}{c^2}}\bigg) \bigg )
      \right)^{-1},
            \label{Energy_fun_case3}
\end{aligned}
\end{equation}
where the equations for $a$ and $t_0$ are given by \eqref{a_equation1}
and \eqref{t0_equation1}, respectively, and $c$ is given implicitly by
\begin{equation}
    \begin{aligned}
 2 = ( 2(a^2+c^2)c+3a c^2) \left (2 E \left(
     {\frac{\pi}{2}},{\frac{c^2-a^2}{c^2}} \right) - E \left(
     \frac{t_0}{2} - \frac{\pi }{4}, {\frac{c^2-a^2}{c^2}}  \right)
 \right) \\ -a^2 c \left (2 F \left(
     {\frac{\pi}{2}},{\frac{c^2-a^2}{c^2}} \right) -F
   \left(\frac{t_0}{2} - \frac{\pi }{4},{\frac{c^2-a^2}{c^2}} \right)
 \right) \\-\frac{(c^2-a^2)h}{2} \sqrt{1-\left(\frac{2
       h^2-c^2-a^2}{c^2-a^2} \right)^2} +2
 \varepsilon^3+\sqrt{\varepsilon^2-h^2}(2 \varepsilon^2+h^2) . 
 \label{V_case3}
\end{aligned}
\end{equation}

\end{lemma}

\begin{proof}
  Note that for case 3 candidates our tangency condition between
  the charges and the elliptic catenary remains unchanged.
  Furthermore, the distance between charges and the volume and surface
  area of the unduloid section follow directly from
  \cite{hadzhilazova07}.
\end{proof}

In the remainder of the paper we carry out an asymptotic expansion of
the energy for the three cases considered above and characterize the
global energy minimizers for $\eps$ small.

\begin{proposition}
\label{Theorem_case1}
There exists a universal $\varepsilon_0>0$ such that if
$\eps < \eps_0$ and
$(\Omega, X)\in \mathcal A_{\frac{4 \pi}{3},2,\varepsilon}$ is a case
1 test configuration with contact height $h \leq \varepsilon$ then
 \begin{equation}
 \label{E_case1_hb}
 E_{\varepsilon} \left(\Omega,X \right)> 4 \pi-\frac{2 \pi h^4
 }{\varepsilon^2}\log{ \left( \frac{h^2}{ \varepsilon} \right) }+
 \BigO{\varepsilon^2} , 
\end{equation}
whenever $h> \frac{\varepsilon}{2}$. Furthermore, we have
\begin{equation}
\label{Case_one_exE}
E_{\varepsilon} \left(\Omega,X\right)=4 \pi + \frac{\gamma
  \varepsilon^3}{2}(1+\varepsilon+ \varepsilon^2)- \left (\frac{2\pi
    h^4}{\varepsilon^2}-\frac{\gamma h^2 \varepsilon^2}{2} \right)
\log{( h)}+\BigO{\gamma \varepsilon^6,(\gamma+1) h^2 \varepsilon^2,
  \frac{h^4}{\varepsilon^2}}, 
\end{equation}
whenever $h \leq \frac{\varepsilon}{2}$.

\end{proposition}
  \begin{proof}
    Let $(\Omega,X)$ be a case 1 test configuration with contact
    height $h$. Now we will expand the expressions for the energy of
    $(\Omega, X)$ and volume of $\Omega$ given in Lemma \ref{case1} in
    terms of $\varepsilon$ and the contact height
    $0< h \leq \varepsilon$. To do this, first we will obtain a lower
    bound on the energy in the regime where
    $h > \frac{\varepsilon}{2}$. \par
    Let $h > \frac{\varepsilon}{2}$.
    Using the notation from Lemma \ref{case1}, we have that the
    minimum $a$ of the extended unduloid section of $\Omega$ is given
    by,
\begin{equation}
\label{a_equ}
a= \frac{c h^2- \varepsilon h^2}{c \varepsilon -h^2}=
\frac{h^2}{\varepsilon} -\frac{h^2}{c}+ \frac{h^4}{c \varepsilon^2}+
\BigO{\frac{h^4}{\varepsilon}}. 
\end{equation}
This gives us that
\begin{multline}
\label{t_0ex}
\frac{t_0}{2}- \frac{\pi}{4} = \frac{\pi}{4} - \frac{1}{2}\arcsin
\left(\frac{2h^2-(c^2+a^2)}{c^2-a^2} \right) =\arcsin \left (
  \sqrt{\frac{c^2-h^2}{c^2-a^2}} \right ) \\ =\arcsin \left ({1-
    \frac{h^2}{2 c^2} + \frac{h^4}{2 c^2 \varepsilon^2}} +
  \BigO{\frac{h^4}{\varepsilon},\frac{h^6}{\varepsilon^3}} \right ).
\end{multline}
Since $h > \frac{\varepsilon}{2} $, from \eqref{a_equ}, \eqref{t_0ex},
and \eqref{F_h_big} we get that
\begin{equation}
\label{F_hB}
F\left(\frac{t_0}{2}-\frac{\pi}{4} , \frac{c^2-a^2}{c^2} \right)=
F\left(\frac{t_0}{2}-\frac{\pi}{4} , 1- \frac{h^4}{\varepsilon^2 c^2
  }+ \BigO{\frac{h^4}{\varepsilon}}\right)=- \log{ \left(
    \frac{h^2}{\varepsilon c} \right)} + \BigO{1}, 
\end{equation}
and from \eqref{a_equ}, \eqref{t_0ex}, and \eqref{E_h_big} we get
\begin{multline}
\label{E_hB}
E \left(\frac{t_0}{2}-\frac{\pi}{4} , \frac{c^2-a^2}{c^2} \right)= E
\left(\frac{t_0}{2}-\frac{\pi}{4} , 1- \frac{h^4}{\varepsilon^2
    c^2 }+ \BigO{\frac{h^4}{\varepsilon}}\right) \\
= 1 - \frac{h^4}{2 c^2 \varepsilon^2 } \log{ \left (
    \frac{h^2}{\varepsilon c} \right)} + \BigO{h^2}.
\end{multline}
Expanding our volume constraint given in \eqref{V} we obtain  
\begin{equation}
\label{vol_ex1}
    2=2c^3 + \frac{3h^2c^2}{\varepsilon}+ \BigO{\varepsilon^2}.
\end{equation}
Thus, \eqref{vol_ex1} implies 
\begin{equation}
\label{c_hb}
c > 1- \frac{h^2 }{2 \varepsilon} + \BigO{\varepsilon^2}. 
\end{equation}
Lastly, using \eqref{F_hB}, \eqref{E_hB}, and \eqref{c_hb} to expand
the contribution of the perimeter to our energy given in
\eqref{Energy_fun}, we obtain
\begin{align}
\label{Eenergy_exhB}
  E_{\varepsilon} \left(\Omega,X \right)> 4 \pi c^2  + 4 \pi \frac{h^2
  c }{\varepsilon}-\frac{2 \pi h^4 }{\varepsilon^2}\log{ \left(
  \frac{h^2}{c \varepsilon} \right) } + \BigO{\varepsilon^2} \notag \\
  > 4
  \pi-\frac{2 \pi h^4 }{\varepsilon^2}\log{ \left( \frac{h^2}{c
  \varepsilon} \right) }+ \BigO{\varepsilon^2} , 
\end{align}
for $\varepsilon < \varepsilon_0$ with $\eps_0$ small enough. This
proves \eqref{E_case1_hb}.

Now we move on to the case where 
\begin{equation}
\label{H_small}
h  \leq \frac{\varepsilon}{2}. 
\end{equation} 
First note that from \eqref{F_ex}, \eqref{a_equ}, \eqref{t_0ex}, and
\eqref{H_small} we get that
 \begin{multline}
 \label{F_expanded}
 F\left(\frac{t_0}{2}-\frac{\pi}{4} , \frac{c^2-a^2}{c^2} \right)=
 F\left(\frac{t_0}{2}-\frac{\pi}{4} , 1- \frac{h^4}{\varepsilon^2 c^2
   }+ \BigO{\frac{h^4}{\varepsilon}}\right) \\
 = \arctanh{ \left (1+ \frac{h^2}{2 c^2}
     \left(-1+\frac{h^2}{\varepsilon^2} \right) \right) } + \BigO{1},
 \end{multline}
 and from \eqref{E_ex}, \eqref{a_equ}, \eqref{t_0ex}, and
 \eqref{H_small} we also have that
\begin{multline}
\label{E_expanded}
E \left(\frac{t_0}{2}-\frac{\pi}{4} , \frac{c^2-a^2}{c^2} \right)= E
\left(\frac{t_0}{2}-\frac{\pi}{4} , 1- \frac{h^4}{\varepsilon^2 c^2 }+
  \BigO{\frac{h^4}{\varepsilon}}\right)= \\ 1+ \frac{h^2}{2 c^2} \left
  (-1 + \frac{h^2}{\varepsilon^2} \left (1+ \arctanh { \left ( 1+
        \frac{h^2}{2 c^2} \left (-1 + \frac{h^2}{\varepsilon^2}
        \right) \right) } \right ) \right ) +
\BigO{\frac{h^4}{\varepsilon^2}}.
\end{multline}

Thus, from \eqref{F_expanded}, \eqref{E_expanded}, and
\eqref{Energy_fun} we get that
\begin{align}
\label{Eenergy_ex}
  E_{\varepsilon} \left(\Omega,X \right)=  4 \pi c^2  + 4 \pi
  \frac{h^2 c }{\varepsilon}- 4 \pi h^2+\frac{2 \pi h^4
  }{\varepsilon^2}\arctanh{ \left (1 - \frac{h^2}{2 c^2} \left (1-
  \frac{h^2}{\varepsilon^2} \right ) \right) }
  +\BigO{\frac{h^4}{\varepsilon^2}} \notag \\
  + \gamma \varepsilon^3 \left
  (\frac{1}{2c} - \frac{h^2}{2c^2 \varepsilon} \arctanh{ \left (1 -
  \frac{h^2}{2 c^2} \left (1- \frac{h^2}{\varepsilon^2} \right )
  \right) } + \frac{\sqrt{\varepsilon^2-h^2}}{2c^2}+
  \frac{\varepsilon^2}{2c^3}+\BigO{\varepsilon^3,
  \frac{h^2}{\varepsilon}  }\right). 
\end{align}
Furthermore, from expanding our volume constraint given in \eqref{V}
we get that
\begin{equation}
\label{vol_ex}
2=2c^3 + \frac{3h^2c^2}{\varepsilon}-3h^2c+
\BigO{\frac{h^4}{\varepsilon^2}}. 
\end{equation}
Expanding \eqref{vol_ex} we get that   
\begin{equation}
\label{h_exC}
h= \sqrt[]{\frac{2 \varepsilon}{3} \left( {1 \over c^2}-c \right)
  \left(1 + \frac{\varepsilon}{c} + \frac{\varepsilon^2}{c^2}+
    \BigO{\frac{h^2}{\varepsilon}, \varepsilon^3 } \right)}. 
\end{equation}
Now we introduce the new variable $\alpha$, which is defined via
\begin{equation}
c = 1- \alpha \varepsilon.
\end{equation}
Thus, \eqref{h_exC} allows us to express $h$ in terms of $\alpha$,
\begin{equation}
\label{h_alpha}
h=\sqrt[]{2 \alpha \varepsilon^2 + 2 \alpha \varepsilon^3 + 2 \alpha
  \varepsilon^4+ \BigO{\alpha \varepsilon^5, \alpha^2 \varepsilon^3,
    \alpha \varepsilon h^2}}  
\end{equation}
Finally, plugging \eqref{h_alpha} into \eqref{Eenergy_ex} we obtain
the following leading order equation for the energy: 
\begin{align}
  E_{\varepsilon} \left(\Omega, X \right)= 4 \pi+ 8 \pi \alpha^2
  \varepsilon^2 \arctanh{(1-\alpha \varepsilon^2)} + \BigO{\alpha
  \varepsilon^4, \alpha^2 \varepsilon^2, \alpha h^2} \notag \\
  + \gamma
  \varepsilon^3 \left(\frac{1}{2}- \alpha \varepsilon
  \arctanh{(1-\alpha \varepsilon^2)}+\frac{\varepsilon}{2}+
  \frac{\varepsilon^2}{2}+ \BigO{\varepsilon^3,\alpha \varepsilon}
  \right). 
\end{align}
Simplifying further gives
\begin{align}
  E_{\varepsilon} \left(\Omega,X\right)=4 \pi + \frac{\gamma
  \varepsilon^3}{2}(1+\varepsilon+\varepsilon^2)+(8 \pi \alpha^2
  \varepsilon^2-\gamma \alpha \varepsilon^4)\arctanh{(1-\alpha
  \varepsilon^2)} \notag \\
  +\BigO{ \gamma \varepsilon^6 , (\gamma+1) \alpha
  \varepsilon^4, \alpha^2 \varepsilon^2, \alpha h^2} , 
\end{align}
which gives us
\begin{equation}
\label{En_alpha}
E_{\varepsilon} \left(\Omega,X \right)=4 \pi + \frac{\gamma
  \varepsilon^3}{2}(1+\varepsilon+ \varepsilon^2)-(4 \pi \alpha^2
\varepsilon^2-\frac{\gamma \alpha \varepsilon^4}{2})\log{(\alpha
  \varepsilon^2)}   +\BigO{ \gamma \varepsilon^6 , (\gamma+1) \alpha
  \varepsilon^4, \alpha^2 \varepsilon^2, \alpha h^2}. 
\end{equation}
Furthermore, we use \eqref{h_alpha} to convert \eqref{En_alpha} into
\begin{equation}
\label{E_expand_h}
E_{\varepsilon} \left(\Omega,X\right)=4 \pi + \frac{\gamma
  \varepsilon^3}{2}(1+\varepsilon+ \varepsilon^2)- \left(\frac{2\pi
    h^4}{\varepsilon^2}-\frac{\gamma h^2 \varepsilon^2}{2} \right)
\log{( h)}+\BigO{\gamma \varepsilon^6,(\gamma+1) h^2 \varepsilon^2,
  \frac{h^4}{\varepsilon^2}}, 
\end{equation}
which completes the proof. 
\end{proof}

\begin{proposition}
\label{Case2_non}
There exists a universal $\varepsilon_0>0$ such that if
$\varepsilon < \varepsilon_0$ and
$(\Omega, X) \in \mathcal A_{\frac{4 \pi}{3},2,\varepsilon}$ is a case
2 test configuration then $(\Omega, X) $ cannot be a minimizer to
\eqref{En_N2}.
\end{proposition}
\begin{proof}
  We will argue by contradiction. Thus, first assume that
  $(\Omega, X)$ is a case 2 minimizer with contact height $h$. Now
  we will expand the expressions for the energy of $(\Omega, X)$ and
  volume of $\Omega$ given in Lemma \ref{case2} in terms of
  $\varepsilon$ and the contact height $0< h \leq \varepsilon$. To do
  this, first note that from \eqref{F_comp} and \eqref{a_equation1} we
  have

\begin{equation}
\label{F_complete0}
F \left(\frac{\pi}{2},\frac{c^2-a^2
}{c^2}  \right)= F \left(\frac{\pi}{2},1- \frac{h^4}{\varepsilon^2 c^2
}+ \BigO{\frac{h^4}{\varepsilon}} \right) =- \log{\left(
  \frac{h^2}{\varepsilon c} \right)}+ \BigO{1}, 
\end{equation}
and from \eqref{E_comp}  and \eqref{a_equation1} we get
\begin{equation}
\label{E_complete0}
E \left(\frac{\pi}{2},\frac{c^2-a^2
  }{c^2}  \right)= E \left(\frac{\pi}{2},1- \frac{h^4}{\varepsilon^2
    c^2 }+ \BigO{\frac{h^4}{\varepsilon}} \right) =1 - \frac{h^4}{2
  \varepsilon^2 c^2} \log{\left( \frac{h^2}{\varepsilon c} \right)} +
\BigO{\frac{h^4}{\varepsilon^2}}.  
\end{equation}
Expanding the the volume constraint given in \eqref{V_case2} gives us 
\begin{equation}
\label{V_eq2}
2= 2 c^3 + \frac{3 c^2 h^2}{\varepsilon} + \BigO{h^2, \varepsilon^3}.
\end{equation}
Using \eqref{F_complete0} and \eqref{E_complete0} we expand the energy
given in \eqref{Energy_fun_case2}, to obtain
\begin{equation}
\label{E_expand2}
E_{\varepsilon} \left(\Omega, X \right)> 4 \pi c^2+ \frac{4 \pi c
  h^2}{\varepsilon}- \frac{2 \pi h^4}{\varepsilon^2} \log{\left(
    \frac{h^2}{\varepsilon c} \right)} + 4 \pi \varepsilon^2 +
\frac{\gamma \varepsilon^3}{2}+ \BigO{h^2, \gamma \varepsilon^2 h^2
  \log{\left( \frac{h^2}{\varepsilon c} \right)} }. 
\end{equation}

Furthermore,  \eqref{V_eq2} implies that  
\begin{equation}
\label{c_big2}
c > 1 - \frac{c^2 h^2}{2 \varepsilon} + \BigO{h^2, \varepsilon^3},
\end{equation}
and from \eqref{E_expand2} and \eqref{c_big2} we get 
\begin{equation}
  E_{\varepsilon}  \left(\Omega, X \right) >4 \pi - \frac{2 \pi
    h^4}{\varepsilon^2} \log{\left( \frac{h^2}{\varepsilon c} \right)} +
  4 \pi \varepsilon^2 + \frac{\gamma \varepsilon^3}{2}+
  \BigO{\varepsilon^3, h^2, \gamma \varepsilon^2 h^2 \log{\left(
        \frac{h^2}{\varepsilon c} \right)} }. 
\end{equation}
Lastly, since we assumed that $\left(\Omega, X \right) $ is a
minimizer, from considering the test configuration given by
$\big (B_{1}(0,0,0), \{ (\varepsilon-1,0,0),(1- \varepsilon,0,0) \}
\big)$ we obtain
\begin{align}
\label{lambda_big}
  4 \pi + \frac{\gamma \varepsilon^3}{2(1-\varepsilon)} \geq
  E_{\varepsilon}  \left(\Omega, X \right) =  4 \pi - \frac{2 \pi
  h^4}{\varepsilon^2} \log{\left( \frac{h^2}{\varepsilon c} \right)}
  \notag \\
  + 4 \pi \varepsilon^2 + \frac{\gamma \varepsilon^3}{2}+
  \BigO{\varepsilon^3, h^2, \gamma \varepsilon^2 h^2 \log{\left(
  \frac{h^2}{\varepsilon c} \right)} }. 
\end{align}
Thus, for $\varepsilon < \varepsilon_0$ small enough, \eqref{lambda_big} implies that 
\begin{equation}
  \gamma \varepsilon^4 \geq 2 \pi \varepsilon^2+\BigO{ \gamma
    \varepsilon^2 h^2 \log{\left( \frac{h^2}{\varepsilon c} \right)}
  }. 
\end{equation} 
However, this implies that
$\gamma > -\frac{c_0}{\varepsilon^2 \log(\varepsilon)}$, where $c_0>0$
is a universal constant, which contradicts Lemma \ref{Exist_MinN2} for
$\varepsilon< \varepsilon_0$ small enough. Thus, $(\Omega, X ) $
cannot be a minimizer.
\end{proof}

\begin{proposition}
\label{Case3_non}
There exists a universal $\varepsilon_0>0$ such that if
$\varepsilon < \varepsilon_0$ and
$(\Omega, X) \in \mathcal A_{\frac{4 \pi}{3},2,\varepsilon}$ is a case
3 test configuration then $(\Omega, X) $ cannot be a minimizer to
\eqref{En_N2}.
\end{proposition}
\begin{proof}
  Assume that $(\Omega, X)$ is a case 3 minimizer with contact
  height $h$. Now we will expand the expressions for the energy of
  $(\Omega, X)$ and volume of $\Omega$ given in Lemma \ref{case3} in
  terms of $\varepsilon$ and the contact height
  $0< h \leq \varepsilon$. \par However, first we will eliminate the
  regime where $h > \frac{\varepsilon }{2} $. To do this, assume that
  $h > \frac{\varepsilon }{2} $, then from \eqref{F_comp} and
  \eqref{a_equation1} we get

\begin{equation}
\label{F_complete}
F \left(\frac{\pi}{2},\frac{c^2-a^2
  }{c^2}  \right)= F \left(\frac{\pi}{2},1- \frac{h^4}{\varepsilon^2
    c^2 }+ \BigO{\frac{h^4}{\varepsilon}} \right) =- \log{\left(
    \frac{h^2}{\varepsilon c} \right)}+ \BigO{1}, 
\end{equation}
and from \eqref{E_comp} and \eqref{a_equation1} we get
\begin{multline}
\label{E_complete}
E \left(\frac{\pi}{2},\frac{c^2-a^2
}{c^2}  \right)= E \left(\frac{\pi}{2},1- \frac{h^4}{\varepsilon^2 c^2
}+ \BigO{\frac{h^4}{\varepsilon}} \right) \\
=1 - \frac{h^4}{2 \varepsilon^2 c^2} \log{\left(
    \frac{h^2}{\varepsilon c} \right)} +
\BigO{\frac{h^4}{\varepsilon^2}}.  
\end{multline}

Since by assumption $h > \frac{\varepsilon}{2} $, from
\eqref{F_h_big}, \eqref{a_equation1} and \eqref{t0_equation1} we get
that
\begin{equation}
\label{F_hB3}
 F\left(\frac{t_0}{2}-\frac{\pi}{4} , \frac{c^2-a^2}{c^2} \right)=
 F\left(\frac{t_0}{2}-\frac{\pi}{4} , 1- \frac{h^4}{\varepsilon^2 c^2
   }+ \BigO{\frac{h^4}{\varepsilon}}\right)=- \log{ \left(
     \frac{h^2}{\varepsilon c} \right)} + \BigO{1}, 
\end{equation}
and from \eqref{E_h_big}, \eqref{a_equation1} and \eqref{t0_equation1}
we get
\begin{multline}
\label{E_hB3}
E \left(\frac{t_0}{2}-\frac{\pi}{4} , \frac{c^2-a^2}{c^2} \right)= E
\left(\frac{t_0}{2}-\frac{\pi}{4} , 1- \frac{h^4}{\varepsilon^2
    c^2 }+ \BigO{\frac{h^4}{\varepsilon}}\right) \\
= 1 - \frac{h^4}{2 c^2 \varepsilon^2 } \log{ \left(
    \frac{h^2}{\varepsilon c} \right)} + \BigO{h^2}.
\end{multline}
Thus, using \eqref{F_complete}, \eqref{E_complete}, \eqref{F_hB3}, and
\eqref{E_hB3} to expand the contribution of the perimeter to
\eqref{Energy_fun_case3}, we obtain
\begin{align}
\label{E_bigg_hb3}
 E_{\varepsilon} \left(\Omega, X\right)> 4 \pi c^2 + \frac{4 \pi c
  h^2}{\varepsilon}- \frac{2 \pi h^4}{\varepsilon^2} \log{ \left(
  \frac{h^2}{\varepsilon c} \right)} + \BigO{\varepsilon^2}, 
\end{align}
and from our volume constraint given in \eqref{V_case3} we obtain 
\begin{equation}
2= 2 c^3 + \frac{3 h^2 c^2}{\varepsilon} + \BigO{\varepsilon^2},
\end{equation}
which implies that 
\begin{equation}
\label{hbcb}
c> 1- \frac{c^2 h^2}{2 \varepsilon} + \BigO{\varepsilon^2}.
\end{equation}
Thus, from \eqref{E_bigg_hb3} and \eqref{hbcb} we get that 
\begin{equation}
  E_{\varepsilon} \left(\Omega, X  \right)>  4 \pi - \frac{2 \pi
    h^4}{\varepsilon^2} \log{ \left( \frac{h^2}{\varepsilon c}
    \right)} + \BigO{\varepsilon^2},  
\end{equation}
which contradicts Lemma \ref{generalized_en} whenever
$\varepsilon< \varepsilon_0$ with $\eps_0$ small, and we conclude that
\eqref{H_small} holds.

Now from \eqref{F_ex}, \eqref{E_ex}, \eqref{a_equation1},
\eqref{t0_equation1} and \eqref{H_small} we obtain
\begin{equation}
\label{F_expand3}
 F\left(\frac{t_0}{2}-\frac{\pi}{4} , \frac{c^2-a^2}{c^2} \right)= - \frac{1}{2} \log{\left(\frac{h^2}{2c^2} \left(1- \frac{h^2}{\varepsilon^2} \right) \right)}+ \BigO{1},
\end{equation}
and
\begin{equation}
\label{E_expand3}
 E\left(\frac{t_0}{2}-\frac{\pi}{4} , \frac{c^2-a^2}{c^2} \right)= 1- \frac{h^2}{2c^2}- \frac{h^4}{4c^2 \varepsilon^2} \log{\left(\frac{h^2}{2c^2} \left(1- \frac{h^2}{\varepsilon^2} \right) \right)}+ \BigO{\frac{h^4}{\varepsilon^2}}.
\end{equation}
Using \eqref{F_complete}, \eqref{E_complete}, \eqref{F_expand3}, and \eqref{E_expand3} to expand the contribution of perimeter to \eqref{Energy_fun_case3} we obtain 
\begin{align}
\label{E_lb3}
  E_{\varepsilon} \left(\Omega, X \right)> 4 \pi c^2 + 4 \pi c
  \frac{h^2}{\varepsilon} -2 \pi h^2 + \frac{\pi h^4 }{\varepsilon^2}
  \log{\left( \frac{\varepsilon^4 c^2}{2h^6} \left(1-
  \frac{h^2}{\varepsilon^2} \right)\right)} \notag \\
  + 4 \pi \varepsilon^2 + 4 \pi \varepsilon \sqrt{\varepsilon^2-h^2} +
  \BigO{\frac{h^4}{\varepsilon^2},h^2}.  
\end{align}

and from our volume constraint given in \eqref{V_case3} we obtain 
\begin{align}
2= 2 c^3 + \frac{3 h^2 c^2}{\varepsilon} + \BigO{\varepsilon^3, h^2}.
\label{Ex__V3}
\end{align}
Thus, from \eqref{Ex__V3} we have 
\begin{equation}
\label{c_bigg3}
c \geq 1 - \frac{h^2 c^2}{2 \varepsilon} + \BigO{\varepsilon^3, h^2 }.
\end{equation}
Lastly, from \eqref{c_bigg3}, \eqref{H_small}, and \eqref{E_lb3} we obtain
\begin{align}
  E_{\varepsilon} \left(\Omega, X \right)>  4 \pi + \frac{\pi h^4
  }{\varepsilon^2} \log{\left( \frac{\varepsilon^4 c^2}{2h^6} \left(1-
  \frac{h^2}{\varepsilon^2} \right)\right)} \notag \\
  + 4 \pi \varepsilon^2 + 4 \pi \varepsilon \sqrt{\varepsilon^2-h^2} +
  \BigO{\varepsilon^3,h^2 , \frac{h^4}{\varepsilon^2}} \\ \geq 4 \pi +
  4 \pi \varepsilon^2, 
\end{align}
for $\varepsilon < \varepsilon_0$ with $\eps_0$ small. Thus, from
Lemma \ref{generalized_en} we conclude that $(\Omega, X )$ cannot be a
minimizer.
\end{proof}

\begin{theorem}
\label{Theorem_main2}
There exist universal constants $ C, \; C_1, \; \varepsilon_0 >0$ such
that if $\varepsilon < \varepsilon_0$ and
$\gamma < \frac{8 \pi}{\varepsilon}-C$ then there exists a unique, up
to translations and rotations, minimizer
$(\Omega^\ast, \{ (-\frac{L^\ast}{2},0,0),(\frac{L^\ast}{2},0,0)
\})$ to \eqref{En_N2}. Furthermore, it is a case 1 configuration, and if
$\frac{C_1}{\log{\varepsilon^{-1}}}<\gamma < \frac{8
  \pi}{\varepsilon}-C$, then
 \begin{align}
   E_{\varepsilon}
   & \left(
     \Omega^\ast, \left \{ \left
     (-\frac{L^\ast}{2},0,0 \right),\left(\frac{L^\ast}{2},0,0
     \right) \right \} \right)  \notag \\
   & =4 \pi + \frac{\gamma \varepsilon^3}{2}(1+\varepsilon+
     \varepsilon^2)+\frac{ \gamma^2 \varepsilon^6}{64 \pi} \log{(\gamma
     \varepsilon^4)}+\BigO{\varepsilon^6 \gamma^{\frac{7}{4}} (1+
     \gamma)^{\frac{1}{4}} (\log{\varepsilon^{-1}})^{\frac{3}{4}}}, \\
   L^\ast
   & = 2-2 \varepsilon - \frac{\gamma \varepsilon^3}{8 \pi} \log{(
     \gamma \varepsilon^4)} + \BigO{ \gamma^{\frac{3}{4}} (1+
     \gamma)^{\frac{1}{4}} \varepsilon^3
     (\log{\varepsilon^{-1}})^{\frac{3}{4}}}, 
\end{align}
and
$(\Omega^\ast, \{ (-\frac{L^\ast}{2},0,0),(\frac{L^\ast}{2},0,0) \}
)$ has contact height $h^\ast$ satisfying
\begin{equation}
  h^\ast=\sqrt[]{\frac{\gamma \varepsilon^4}{8 \pi}} + 
  \BigO{ \frac{  \varepsilon^2 \gamma ( \gamma+1)}{
      \log{{\varepsilon}^{-1}}}^{\frac{1}{4}}}.  
\end{equation}
\end{theorem}



  \begin{proof}
    First note that the existence and rotational symmetry of a
    minimizer follows directly from Theorem \ref{exgen}, Lemma
    \ref{Exist_MinN2} and Lemma \ref{rotational}. Furthermore, Lemma
    \ref{One_ArcN2} implies that this minimizer is either a case 1,
    case 2, or case 3 configuration. Thus, Propositions
    \ref{Case2_non} and \ref{Case3_non} rule out all but a case 1
    minimizer
    $(\Omega^\ast, \{
    (-\frac{L^\ast}{2},0,0),(\frac{L^\ast}{2},0,0) \})$.
  
    Let $h^\ast$ be the contact height of the minimizer above. Then
    from \eqref{E_case1_hb} and Lemma \ref{generalized_en} we conclude
    that
  \begin{equation}
  h^\ast \leq \frac{\varepsilon}{2}.
  \end{equation}
Thus, from \eqref{Case_one_exE} we conclude that 
\begin{multline}
 \label{E_expand_h_st}
 E_{\varepsilon} \left(\Omega^\ast, \left \{
     \left(-\frac{L^\ast}{2},0,0 \right), \left(\frac{L^\ast}{2},0,0
     \right) \right \} \right) \\
 =4 \pi + \frac{\gamma \varepsilon^3}{2}(1+\varepsilon+
 \varepsilon^2)- \left(\frac{2\pi
     {h^\ast}^4}{\varepsilon^2}-\frac{\gamma {h^\ast}^2
     \varepsilon^2}{2} \right) \log{( h^\ast)}\\
 + \BigO{\gamma
   \varepsilon^6,(\gamma+1) {h^\ast}^2 \varepsilon^2,
   \frac{{h^\ast}^4}{\varepsilon^2}}.
\end{multline}

Now let
$ \left(\Omega, \left\{ \left(-\frac{L}{2},0,0 \right),
    \left(\frac{L}{2},0,0 \right) \right\} \right)$ be a case 1
candidate minimizer with a contact height $h$ given by
\begin{equation}
h=\sqrt[]{\frac{\gamma \varepsilon^4}{8 \pi}}.
\end{equation}
From the minimality of
$\left(\Omega^\ast, \left \{ \left(-\frac{L^\ast}{2},0,0 \right),
    \left(\frac{L^\ast}{2},0,0 \right) \right\} \right)$ and from
\eqref{Case_one_exE} and \eqref{E_expand_h_st} we obtain
\begin{equation}
    \begin{aligned}
4 \pi + \frac{\gamma \varepsilon^3}{2}(1+\varepsilon+ \varepsilon^2)-
\left(\frac{2\pi {h^\ast}^4}{\varepsilon^2}-\frac{\gamma {h^\ast}^2
    \varepsilon^2}{2} \right) \log{( h^\ast)}+\BigO{\gamma
  \varepsilon^6,(\gamma+1) {h^\ast}^2 \varepsilon^2,
  \frac{{h^\ast}^4}{\varepsilon^2}} \\ 
\leq 4 \pi + \frac{\gamma \varepsilon^3}{2}(1+\varepsilon+
\varepsilon^2) +\frac{\gamma^2 \varepsilon^6}{32 \pi} \log{ \left(
    \sqrt[]{\frac{\gamma \varepsilon^4}{8 \pi}}\right)}+ \BigO{\gamma
  (\gamma+1) \varepsilon^6}. 
\end{aligned}
\end{equation}
Thus
\begin{equation}
\label{N2min_eq}
- \left(\frac{2\pi {h^\ast}^4}{\varepsilon^2}-\frac{\gamma
    {h^\ast}^2 \varepsilon^2}{2} \right) \log{(
  h^\ast)}+\BigO{\gamma (\gamma+1)\varepsilon^6,(\gamma+1)
  {h^\ast}^2 \varepsilon^2, \frac{{h^\ast}^4}{\varepsilon^2}} \leq
\frac{\gamma^2 \varepsilon^6}{32 \pi} \log{ \left(
    \sqrt[]{\frac{\gamma \varepsilon^4}{8 \pi}}\right)}. 
\end{equation}
Now note that \eqref{N2min_eq} implies that 
\begin{equation}
\label{h_bigg}
h^\ast \geq \frac{\sqrt[]{\gamma \varepsilon^4} }{8},
\end{equation}
since otherwise \eqref{N2min_eq} implies
\begin{equation}
  \frac{\gamma^2 \varepsilon^6}{128} \log{\left (\frac{\sqrt[]{\gamma
          \epsilon^4}}{8} \right) }+ \BigO{\gamma(\gamma+1)
    \varepsilon^6} \leq  \frac{\gamma^2 \varepsilon^6}{32 \pi} \log{
    \left( \sqrt[]{\frac{\gamma \varepsilon^4}{8 \pi}}\right)}. 
\end{equation}
However, this implies 
\begin{equation}
\label{h_small_con}
\frac{\gamma^2 \varepsilon^6}{256} \log{\left (\gamma \varepsilon^4
  \right) }+ \BigO{\gamma(\gamma+1) \varepsilon^6} \leq
\frac{\gamma^2 \varepsilon^6}{64 \pi} \log{ \left(\gamma \varepsilon^4
  \right)}. 
\end{equation} 
Now pick $\varepsilon_0>0$ so that
$| \log{(\gamma \varepsilon^4)}|> \BigO{\frac{\gamma+1}{\gamma}}$,
then \eqref{h_small_con} provides a contradiction. Thus,
\eqref{h_bigg} holds.


 Finally, let 
\begin{equation}
\label{h_star_he}
 h^\ast= \sqrt[]{\frac{\gamma \varepsilon^4}{8 \pi}}+ h_e,
\end{equation} 
then from \eqref{N2min_eq} we obtain 
\begin{multline}
\label{h_e_inq}
- \left ( \gamma \varepsilon^2 h_e^2 + \sqrt{8 \pi \gamma} h_e^3 +
  \frac{2 \pi}{\varepsilon^2} h_e^4 \right)  \log{h^\ast}  \\
\leq - \frac{ \gamma^2 \varepsilon^6}{32 \pi} \log{ \left(1 +
    \sqrt{\frac{8 \pi}{\gamma \varepsilon^4} } h_e \right )
}+\BigO{\gamma (\gamma+1)\varepsilon^6,(\gamma+1) {h^\ast}^2
  \varepsilon^2, \frac{{h^\ast}^4}{\varepsilon^2}}
\end{multline}
If $h_e>0$, then \eqref{h_star_he} implies $h^\ast>  \sqrt[]{\frac{\gamma \varepsilon^4}{8 \pi}}$, and \eqref{h_e_inq} gives
\begin{equation}
- \frac{2 \pi}{\varepsilon^2} h_e^4  \log{h^\ast} \leq \BigO{\frac{{h^\ast}^4 (\gamma+1)}{\gamma \varepsilon^2}}.
\end{equation}
Thus, 
\begin{equation}
\label{h_e_upp_b}
h_e \leq  \BigO{ h^\ast \left( \frac{   \gamma+1}{\gamma \log{\left({h^\ast}^{-1} \right)}} \right)^{\frac{1}{4}}}.
\end{equation}

Now \eqref{h_e_upp_b} and \eqref{h_star_he} imply that 
\begin{equation}
\label{h_star_upper}
h^\ast \leq  \sqrt[]{\frac{\gamma \varepsilon^4}{8 \pi}} \left( 1 +  \BigO{
    \left( \frac{   \gamma+1}{\gamma \log{({h^\ast}^{-1})}}
    \right)^{\frac{1}{4}}  } \right),
\end{equation} 
whenever $h^\ast \leq \varepsilon < \varepsilon_0$ is sufficiently small. 
\par If $h_e <0$, then from \eqref{h_star_he} we have
$0<h^\ast < \sqrt[]{\frac{\gamma \varepsilon^4}{8 \pi}}$ and
\eqref{h_e_inq} implies
\begin{equation}
\label{h_e_llb}
-  \frac{2 \pi}{\varepsilon^2} h_e^4   \log{h^\ast}  \leq  - \frac{
  \gamma^2 \varepsilon^6}{32 \pi} \log{ \left(1 + \sqrt{\frac{8
        \pi}{\gamma \varepsilon^4} } h_e \right ) }+\BigO{\gamma
  (\gamma+1)\varepsilon^6} . 
\end{equation}
Furthermore, \eqref{h_bigg} implies that
$h_e \geq \frac{\sqrt{\gamma \varepsilon^4}}{8} - \frac{\sqrt{\gamma
    \varepsilon^4}}{\sqrt{8 \pi}} $. Thus, from \eqref{h_e_llb} we
obtain
\begin{equation}
\label{h_e_lower_bbd}
h_e \geq  \BigO{ \varepsilon^2 \left( \frac{  \gamma ( \gamma+1)}{
      \log{({h^\ast}^{-1})}} \right)^{\frac{1}{4}}}.  
\end{equation}
Hence, \eqref{h_star_upper} and \eqref{h_e_lower_bbd} imply that
\begin{equation}
\label{h_star_ex}
h^\ast =\sqrt[]{\frac{\gamma \varepsilon^4}{8 \pi}} + 
\BigO{  \varepsilon^2 \left( \frac{  \gamma ( \gamma+1)}{
      \log{{\varepsilon}^{-1}}} \right)^{\frac{1}{4}}}.  
\end{equation}
Thus, from \eqref{E_expand_h_st} and \eqref{h_star_ex} we obtain
\begin{multline}
  E_{\varepsilon} \left(\Omega^\ast, \left\{
      \left(-\frac{L^\ast}{2},0,0 \right),
      \left(\frac{L^\ast}{2},0,0
      \right) \right \} \right) \\
  =4 \pi + \frac{\gamma \varepsilon^3}{2}(1+\varepsilon+
  \varepsilon^2)+\frac{ \gamma^2 \varepsilon^6}{64 \pi} \log{(\gamma
    \varepsilon^4)}+\BigO{\varepsilon^6 \gamma^{\frac{7}{4}} (1+
    \gamma)^{\frac{1}{4}} (\log{\varepsilon^{-1}})^{\frac{3}{4}}}.
\end{multline}
Lastly, from \eqref{F_expanded}, \eqref{E_expanded}, and
\eqref{Energy_fun} we obtain
\begin{equation}
  L^\ast= 2-2 \varepsilon - \frac{\gamma \varepsilon^3}{8 \pi} \log{(
    \gamma \varepsilon^4)} + \BigO{ \gamma^{\frac{3}{4}} (1+
    \gamma)^{\frac{1}{4}} \varepsilon^3
    (\log{\varepsilon^{-1}})^{\frac{3}{4}}}. 
\end{equation}

To conclude the proof, observe that
$(\Omega^\ast, \{ (-\frac{L^\ast}{2},0,0),(\frac{L^\ast}{2},0,0) \})$
is unique, since expanding the second derivative of \eqref{Energy_fun}
gives
\begin{equation}
  \hat E_\varepsilon'' (h) = \left (\frac{24 \pi h^2}{\varepsilon^2} -
    \gamma \varepsilon^2 \right) \log h^{-1} +
  \BigO{\frac{h^2}{\varepsilon^2}, \gamma \varepsilon^2 },
\end{equation}
which implies that $\hat E_\eps(h)$ is strictly convex in the neighborhood of
the minimum.
  \end{proof}

  \begin{proof}[Proof of Theorem \ref{t:2}]
    The proof is obtained by combining the statements of Propositions
    \ref{Case2_non} and \ref{Case3_non} with that of Theorem
    \ref{Theorem_main2}.
  \end{proof}

\paragraph{Conflict of interest.} The Authors declare no conflict of
interest.

\paragraph{Data availability.} The manuscript contains no associated
data.

\bibliographystyle{plain}


\begin{thebibliography}{10}

\bibitem{abbas67}
M.~A. Abbas and J.~Latham.
\newblock The instability of evaporating charged drops.
\newblock {\em J. Fluid Mech.}, 30:663--670, 1967.

\bibitem{abdurahimov12}
L.~V. Abdurahimov, A.~A. Levchenko, L.~P. Mezhov-Deglin, and I.~M. Khalatnikov.
\newblock The surface instability of liquid hydrogen and helium.
\newblock {\em Low Temp. Phys.}, 38:1013--1025, 2012.

\bibitem{almgren76}
F.~J. Almgren, Jr.
\newblock Existence and regularity almost everywhere of solutions to elliptic
  variational problems with constraints.
\newblock {\em Mem. Amer. Math. Soc.}, 4:viii+199, 1976.

\bibitem{barchiesi13}
M.~Barchiesi, F.~Cagnetti, and N.~Fusco.
\newblock Stability of the steiner symmetrization of convex sets.
\newblock {\em J. Eur. Math. Soc.}, 15:1245--1278, 2013.

\bibitem{barranco06}
M.~Barranco, R.~Guardiola, S.~Hern{\'a}ndez, R.~Mayol, J.~Navarro, and M.~Pi.
\newblock Helium nanodroplets: An overview.
\newblock {\em J. Low Temp. Phys.}, 142:1--81, 2006.

\bibitem{cmt:nams17}
R.~Choksi, C.~B. Muratov, and I.~Topaloglu.
\newblock An old problem resurfaces nonlocally: Gamow's liquid drops inspire
  today's research and applications.
\newblock {\em Notices Amer. Math. Soc.}, 64:1275--1283, 2017.

\bibitem{dephilippis23}
G.~De~Philippis, J.~Hirsch, and G.~Vescovo.
\newblock Regularity of minimizers for a model of charged droplets.
\newblock {\em Commun. Math. Phys.}, 401:33--78, 2023.

\bibitem{delaunay41}
C.~Delaunay.
\newblock Sur la surface de r{\'e}volution dont la courbure moyenne est
  constante.
\newblock {\em J. Math. Pures Appl.}, 6:309--315, 1841.

\bibitem{doyle64}
A.~Doyle, D.~R. Moffett, and B.~Vonnegut.
\newblock Behavior of evaporating electrically charged droplets.
\newblock {\em J. Colloid Sci.}, 19:136--143, 1964.

\bibitem{duft03}
D.~Duft, T.~Achtzehn, R.~M\"uller, B.~A. Huber, and T.~Leisner.
\newblock Coulomb fission: {Rayleigh} jets from levitated microdroplets.
\newblock {\em Nature}, 421:128--128, 2003.

\bibitem{fenn93}
J.~B. Fenn.
\newblock Ion formation from charged droplets: roles of geometry, energy, and
  time.
\newblock {\em J. Amer. Soc. Mass Spectrom.}, 4:524--535, 1993.

\bibitem{fernandezdelamora07}
J.~Fern\'andez de~la Mora.
\newblock The fluid dynamics of {Taylor} cones.
\newblock {\em Ann. Rev. Fluid Mech.}, 39:217--243, 2007.

\bibitem{frank21}
R.~L. Frank and P.~T. Nam.
\newblock Existence and nonexistence in the liquid drop model.
\newblock {\em Calc. Var. Partial Differential Equations}, 60:223, 2021.

\bibitem{fusco08}
N.~Fusco, F.~Maggi, and A.~Pratelli.
\newblock The sharp quantitative isoperimetric inequality.
\newblock {\em Ann. of Math.}, 168:941--980, 2008.

\bibitem{gamow30}
G.~Gamow.
\newblock Mass defect curve and nuclear constitution.
\newblock {\em Proc. Roy. Soc. London A}, 126:632--644, 1930.

\bibitem{gaskell97}
S.~J. Gaskell.
\newblock Electrospray: Principles and practice.
\newblock {\em J. Mass Spectrom.}, 32:677--688, 1997.

\bibitem{giglio08}
E.~Giglio, B.~Gervais, J.~Rangama, B.~Manil, B.~A. Huber, D.~Duft, R.~M\"uller,
  T.~Leisner, and C.~Guet.
\newblock Shape deformations of surface-charged microdroplets.
\newblock {\em Phys. Rev. E}, 77:036319, 2008.

\bibitem{goldman15}
M.~Goldman, M.~Novaga, and B.~Ruffini.
\newblock Existence and stability for a non-local isoperimetric model of
  charged liquid drops.
\newblock {\em Arch. Rational Mech. Anal.}, 217:1--36, 2015.

\bibitem{goldman22}
M.~Goldman, M.~Novaga, and B.~Ruffini.
\newblock Rigidity of the ball for an isoperimetric problem with strong
  capacitary repulsion.
\newblock Preprint: arXiv:2201.04376, 2022.

\bibitem{goldman17}
M.~Goldman and B.~Ruffini.
\newblock Equilibrium shapes of charged droplets and related problems: (mostly)
  a review.
\newblock {\em Geom. Flows}, 2:94--104, 2017.

\bibitem{gomez94}
A.~Gomez and K.~Tang.
\newblock Charge and fission of droplets in electrostatic sprays.
\newblock {\em Phys. Fluids}, 6:404--414, 1994.

\bibitem{gorkov73}
L.~P. Gor'kov and D.~M. Chernikova.
\newblock Concerning the structure of a charged surface of liquid helium.
\newblock {\em JETP Lett.}, 18:68--70, 1973.

\bibitem{grimes79}
C.~C. Grimes and G.~Adams.
\newblock Evidence for a liquid-to-crystal phase transition in a classical,
  two-dimensional sheet of electrons.
\newblock {\em Phys. Rev. Lett.}, 42:795--798, 1979.

\bibitem{hadzhilazova07}
M.~Hadzhilazova, I.~M. Mladenov, and J.~Oprea.
\newblock Unduloids and their geometry.
\newblock {\em Archivum Mathematicum (Brno)}, 43:417--429, 2007.

\bibitem{iribane76}
J.~V. Iribarne and B.~A. Thomson.
\newblock On the evaporation of small ions from charged droplets.
\newblock {\em J. Chem. Phys.}, 64:2287--2294, 1976.

\bibitem{kebarle00}
P.~Kebarle and M.~Peschke.
\newblock On the mechanisms by which the charged droplets produced by
  electrospray lead to gas phase ions.
\newblock {\em Analytica Chimica Acta}, 406:11--35, 2000.

\bibitem{kmn:cmp16}
H.~Kn\"upfer, C.~B. Muratov, and M.~Novaga.
\newblock Low density phases in a uniformly charged liquid.
\newblock {\em Comm. Math. Phys.}, 345:141--183, 2016.

\bibitem{labowsky98}
M.~Labowsky.
\newblock Discrete charge distributions in dielectric droplets.
\newblock {\em J. Colloid Interface Sci.}, 206:19--28, 1998.

\bibitem{rayleigh1882}
{Lord Rayleigh}.
\newblock On the equilibrium of liquid conducting masses charged with
  electricity.
\newblock {\em Phil. Mag.}, 14:184--186, 1882.

\bibitem{maggi}
F.~Maggi.
\newblock {\em Sets of Finite Perimeter and Geometric Variational Problems}.
\newblock Cambridge Studies in Advanced Mathematics, 135. Cambridge University
  Press, Cambridge, 2012.

\bibitem{merlet22}
B.~Merlet and M.~Pegon.
\newblock Large mass rigidity for a liquid drop model in $2${D} with kernels of
  finite moments.
\newblock {\em Journal de l{\textquoteright}\'Ecole polytechnique {\textemdash}
  Math\'ematiques}, 9:63--100, 2022.

\bibitem{mukoseeva23}
E.~Mukoseeva and G.~Vescovo.
\newblock Minimality of the ball for a model of charged liquid droplets.
\newblock {\em Ann. Inst. H. Poincar{\'e} Anal. Non Lin{\'e}aire}, 40:457--509,
  2023.

\bibitem{mn:prsa16}
C.~B. Muratov and M.~Novaga.
\newblock On well-posedness of variational models of charged drops.
\newblock {\em Proc. R. Soc. A}, 472:20150808, 2016.

\bibitem{novaga22}
M.~Novaga and F.~Onoue.
\newblock Existence of minimizers for a generalized liquid drop model with
  fractional perimeter.
\newblock {\em Nonlinear Anal.}, 224:113078, 2022.

\bibitem{oprea}
J.~Oprea.
\newblock {\em Differential Geometry and Its Applications.}
\newblock Mathematical Association of America, 2007.

\bibitem{serfaty-coulomb}
S.~Serfaty.
\newblock {\em Coulomb Gases and Ginzburg-Landau Vortices}, volume~21 of {\em
  Zurich Lectures in Advanced Mathematics}.
\newblock Eur. Math. Soc., 2015.

\bibitem{taylor64}
G.~I. Taylor.
\newblock Disintegration of water drops in an electric field.
\newblock {\em Proc. Roy. Soc. A}, 280:383--397, 1964.

\bibitem{tsao98}
C.-C. Tsao, J.~D. Lobo, M.~Okumura, and S.-Y. Lo.
\newblock Generation of charged droplets by field ionization of liquid helium.
\newblock {\em J. Phys. D: Appl. Phys.}, 31:2195--2204, 1998.

\bibitem{Wagner90}
G.~Wagner.
\newblock On means of distances on the surface of a sphere (lower bounds).
\newblock {\em Pac. J. Math.}, 144:389--398, 1990.

\bibitem{Wagner92}
G.~Wagner.
\newblock On means of distances on the surface of a sphere. {II} (upper
  bounds).
\newblock {\em Pac. J. Math.}, 154:381--395, 1992.

\bibitem{zeleny14}
J.~Zeleny.
\newblock The electrical discharge from liquid points, and a hydrostatic method
  of measuring the electric intensity at their surfaces.
\newblock {\em Phys. Rev.}, 3:69--91, 1914.

\bibitem{zeleny17}
J.~Zeleny.
\newblock Instability of electrified liquid surfaces.
\newblock {\em Phys. Rev.}, 10:1--6, 1917.

\end{thebibliography}


\end{document}